\documentclass[10pt, a4paper, oneside]{article}

\usepackage[scaled]{helvet} 
\usepackage{courier} 
\usepackage{eulervm} 
\usepackage{lscape}

\normalfont
\usepackage[english]{babel}
\usepackage[utf8]{inputenc}
\usepackage{indentfirst}				
\usepackage{booktabs}				
\usepackage{multirow}					
\usepackage{tabularx}					
\usepackage{graphicx}					
\usepackage{caption}					
\usepackage{amsmath,amssymb,amsthm,mathrsfs,mathtools}	
\usepackage{bm,faktor,braket,stmaryrd,esint}
\usepackage[english]{varioref}			
\usepackage{mparhack,relsize}			
\usepackage[dvipsnames]{xcolor}			
\usepackage{hyperref}					
\usepackage{bookmark}					
\usepackage{tikz}					
\usepackage{fullpage}
\usepackage{tcolorbox}
\usepackage[export]{adjustbox}
\usepackage{rotating}
\usepackage{pbox}
\usepackage{float}
\newsavebox{\mytable}

\usepackage[textsize=footnotesize]{todonotes}
\usepackage{tikz-cd}
\usetikzlibrary{arrows,decorations.pathreplacing,decorations.markings,calc}

\usepackage{enumitem}
\numberwithin{equation}{section}
\usepackage[labelsep=endash]{caption}

\newcommand{\dd}{\mathrm{d}}

\newcommand{\calP}{\mathcal{P}}
\newcommand{\calV}{\mathcal{V}}

\newcommand{\iu}{\mathrm{i}}

\DeclareMathOperator{\Aut}{Aut}

\DeclareMathOperator{\tr}{tr}

\DeclareMathOperator*{\Res}{Res}

\newcommand{\beq}{\begin{equation}}
\newcommand{\eeq}{\end{equation}}
\newcommand{\bea}{\begin{eqnarray}}
\newcommand{\eea}{\end{eqnarray}}

\theoremstyle{definition}
\newtheorem{thm}{Theorem}[section]

\newtheorem{defn}[thm]{Definition}

\newtheorem{rem}[thm]{Remark}
\newtheorem{lem}[thm]{Lemma}
\newtheorem{cor}[thm]{Corollary}

\newtheorem*{que*}{\textcolor{BrickRed}{Question}}

\newtheorem*{ex}{Example}
\def\bd{\begin{defn}}
\def\ed{\end{defn}}
\def\br{\begin{rem}}
\def\er{\end{rem}}
\def\bex{\begin{ex}}
\def\eex{\end{ex}}

\definecolor{webbrown}{rgb}{0.65, 0.16, 0.16}
\newcommand{\ben}{\begin{eqnarray*}}
\newcommand{\een}{\end{eqnarray*}}
\newcommand{\be}{\begin{equation}}
\newcommand{\ee}{\end{equation}}

\makeatletter
\renewenvironment{abstract}{%
    \if@twocolumn
      \section*{\abstractname}%
    \else \normalsize 
      \begin{center}%
        {\bfseries \normalsize\abstractname\vspace{\z@}}
      \end{center}%
      \quotation
    \fi}
    {\if@twocolumn\else\endquotation\fi}
\makeatother

\hypersetup{
colorlinks=true, linktocpage=true, pdfstartpage=1, pdfstartview=FitV,breaklinks=true, pdfpagemode=UseNone, pageanchor=true, pdfpagemode=UseOutlines,plainpages=false, bookmarksnumbered, bookmarksopen=true, bookmarksopenlevel=1,hypertexnames=true, pdfhighlight=/O,urlcolor=webbrown, linkcolor=RoyalBlue, citecolor=ForestGreen}

\title{\vspace{-0.25cm}\textsc{Topological recursion for generalised Kontsevich graphs and r-spin intersection numbers}\vspace{0.25cm}}
\date{\vspace{-6ex}}

\author{Rapha\"{e}l Belliard\footnote{Humboldt Universit\"at zu Berlin, Unter den Linden 6, 10117, Berlin, Germany}\,\,, S\'{e}verin Charbonnier\footnote{Max Planck Institut f\"ur Mathematik, Vivatsgasse 7, 53111 Bonn, Germany.}\,\,, Bertrand Eynard\footnote{Institut de Physique Th\'eorique-CEA Saclay, Orme des Merisiers, 91191 Gif-sur-Yvette, France}\,\,\footnote{Institut des Hautes \'Etudes Scientifiques, le Bois-Marie, 35 route de Chartres, 91440 Bures-sur-Yvette, France.}\,\,, Elba Garcia-Failde\footnotemark[3]\,\,\footnotemark[4]\,\,\footnote{Universit\'{e} de Paris, B\^{a}timent Sophie Germain, 8 Place Aur\'{e}lie Nemours.}}

\begin{document}

\maketitle 

%

\vspace{0.25cm}

\begin{abstract}
\noindent \small{In 1992, Kontsevich introduced certain ribbon graphs \cite{Kon92} as cell decompositions for combinatorial models of moduli spaces of complex curves with boundaries in his proof of Witten’s conjecture \cite{Witt90}. In this work, we define four types of generalised Kontsevich graphs and find combinatorial relations among them. We call the main type \textit{ciliated maps} and use the auxiliary ones to show they satisfy a Tutte recursion that we turn into a combinatorial interpretation of the loop equations of topological recursion for a large class of spectral curves. It follows that ciliated maps, which are Feynman graphs for the Generalised Kontsevich matrix Model (GKM), are computed by topological recursion. The GKM relates to the $r$-KdV integrable hierarchy and since the string solution of the latter encodes intersection numbers with Witten’s $r$-spin class, we find an expression for the generating series of ciliated maps in terms of $r$-spin intersection numbers, implying that they are also governed by topological recursion. In turn, this paves the way towards a combinatorial understanding of Witten’s class. This new topological recursion perspective on the GKM also provides concrete tools to explore the conjectural symplectic invariance property of topological recursion for large classes of spectral curves.}

\end{abstract}

\begin{small}
\tableofcontents
\end{small}

\thispagestyle{empty}
\bigskip



\newpage

\section{Introduction}
\label{SIntro}
In the last 30 years, deep ties have emerged among the mathematical areas of enumerative geometry (A-side), complex geometry (B-side), intersection theory of moduli spaces of Riemann surfaces (C-side) and integrable hierarchies (D-side). In 1990, Witten \cite{Witt90} conjectured that a generating function of intersection numbers of the so-called $\psi$-classes on the moduli space of Riemann surfaces (C) is a $\tau$-function for the KdV integrable hierarchy (D), an infinite set of compatible non-linear partial differential equations. It carried on the program developed by the Japanese school in the late 70's to construct solutions to non-linear equations using methods from quantum field theory \cite{SMJ79,JMU81,SegWil85}, by means of infinite-dimensional representation theory.

\medskip

Witten's conjecture was proved by Kontsevich \cite{Kon92}, using a matrix model (A) and the so-called Airy curve (B), thus connecting the four areas enumerated above. Let us describe in more details those links here, and the path taken by Kontsevich to realise the relations.
\begin{itemize}
	\item Starting from the C-side, let $\overline{\mathfrak{M}}_{g,n}$ be the Deligne--Mumford compactification of the moduli space of Riemann surfaces of genus $g$ with $n$ marked points and let $\omega_{g,n}$ be the $n$-differential over $\mathbb{C}$ defined by
	\[
	\omega_{g,n}(z_1,\dots,z_n) = 2^{2-2g-n} \sum\limits_{d_1+\dots + d_n=3g-3+n} \left(\int_{\overline{\mathfrak{M}}_{g,n}} \psi_1^{d_1}\dots \psi_n^{d_n}\right) \prod_{i=1}^{n} \frac{(2d_i+1)!!\, d z_i}{z_i^{2d_i+2}}\;,
	\]
	where $\psi_i\in H^2(\overline{\mathfrak{M}}_{g,n}, \mathbb{Q})$ is the first Chern class of the line bundle $\mathcal{L}_i\to \overline{\mathfrak{M}}_{g,n}$ whose fiber at $\Sigma\in\overline{\mathfrak{M}}_{g,n}$ is the cotangent line to the $i^{\textup{th}}$ marked point. The differentials $\omega_{g,n}$ are meromorphic generating differentials for intersection numbers of $\psi$-classes.
	\item Under the identification of $\mathfrak{M}_{g,n}$ with the combinatorial moduli space $\mathcal{M}_{g,n}(L_1,\dots,L_n)$ \emph{via} Jenkins-Strebel's theorem \cite{Jen57, Str67}, Kontsevich \cite{Kon92} compared the measure $\left(1/2\, \sum_{i=1}^{n}L_i^2 \psi_i\right)^{3g-3+n}/(3g-3+n)!$ on $\overline{\mathfrak{M}}_{g,n}$ and the Lebesgue measure $\mu_{\textup{Leb}}$ on $\mathcal{M}_{g,n}(L_1,\dots,L_n)$ given by
	\[
	d\mu_{\textup{Leb}} = \frac{2^{2-2g-n}}{(3g-3+n)!}\left(\frac{1}{2} \sum\limits_{i=1}^{n}L_i^2 \psi_i\right)^{3g-3+n} \;,
	\]
	where we identified the two spaces $\overline{\mathfrak{M}}_{g,n}$ and $\mathcal{M}_{g,n}(L_1,\dots,L_n)$ by slightly abusing notations and ignoring infamous boundary issues. It follows that the differentials $\omega_{g,n}$ can be interpreted as Laplace transforms of the corresponding symplectic volumes
	\[
	\omega_{g,n}(z_1,\dots,z_n) = (-2)^{2-2g-n} \left[\int_{\mathbb{R}_+^{n}}\left(\int_{\mathcal{M}_{g,n}(L)} d\mu_{\textup{Leb}}\right)\prod_{i=1}^{n} L_i \, e^{-L_i\, z_i} d L_i\right] \, \prod_{i=1}^{n}dz_i\;.
	\]
	The last equation allows to identify the $\omega_{g,n}$'s with correlation functions of a matrix model (called the Kontsevich matrix model in the present article), and thus have a combinatorial interpretation in terms of ribbon graphs. The differentials defined on the C-side are therefore also generating functions for combinatorial ribbon graphs: the problem had turned into an enumerative geometric one (A-side).
	\item Kontsevich then proved \cite{Kon92} that the matrix model could be solved from the knowledge of a single complex curve, the Airy curve. This can be rephrased in terms of the \emph{topological recursion}, a procedure developed by Chekhov, one of the authors and Orantin, initially in the context of large size asymptotic expansions in random matrix theory \cite{CE06,CEO06,CE062}. It was later established as an independent universal theory \cite{EO07inv}. From a \emph{spectral curve}
\begin{eqnarray}
\mathcal{S}=(\Sigma,\Sigma_0,x:\Sigma\to\Sigma_0,y: \Sigma\to\Sigma_0, \omega_{0,2})\;,
\end{eqnarray}
it defines differentials $\omega_{g,n}$ (indexed by two integers $g,n$) recursively on $2g+n-2$, using complex geometric tools on a Riemann surface such as residue computations. Translated to this formalism, what Kontsevich showed corresponds to his matrix model being governed by topological recursion (belonging to the B-side) with the Airy curve as spectral curve:
	\[
	\mathcal{S}_{\textup{K}}=\left(\,\Sigma = \overline{\mathbb{C}}\,\, ;\,\, \Sigma_0 = \overline{\mathbb{C}}\,\, ;\,\, x(z) =\frac{z^2}{2} \,\, ;\,\, y(z) = z \,\, ;\,\, \omega_{0,2} (z_1,z_2) =\frac{d z_1 \otimes d z_2}{(z_1-z_2)^2}\,\right)\;.
	\]
	\item From an argument of Dijkgraaf, Verlinde and Verlinde \cite{DVV91}, the last result is equivalent to the fact that the $\omega_{g,n}$'s satisfy the Virasoro constraints of the KdV hierarchy (D-side), which allows to deduce the veracity of Witten's conjecture \cite{Witt90}.  The Virasoro constraints encode highest-weight vector conditions for the Virasoro algebra \textbf{Vir}, namely the infinite-dimensional Lie algebra generated by elements $(L_n)_{n\in\mathbb Z}$ together with a central element acting as a scalar $\mathbf{c}\in\mathbb C$, and satisfying the commutation relations
\begin{eqnarray}
[L_n,L_m] = (n-m)L_{n+m}+\frac{\mathbf{c}}{12}n(n^2-1)\delta_{n+m,0}\;.
\end{eqnarray}
The Virasoro constraints then take the form $L_nZ=0$ for all $n\geq1$ together with an eigenvalue condition for the action of the zero mode $L_0$.
\end{itemize}
This can be summarised in the diagram:
$$\includegraphics[width=\textwidth]{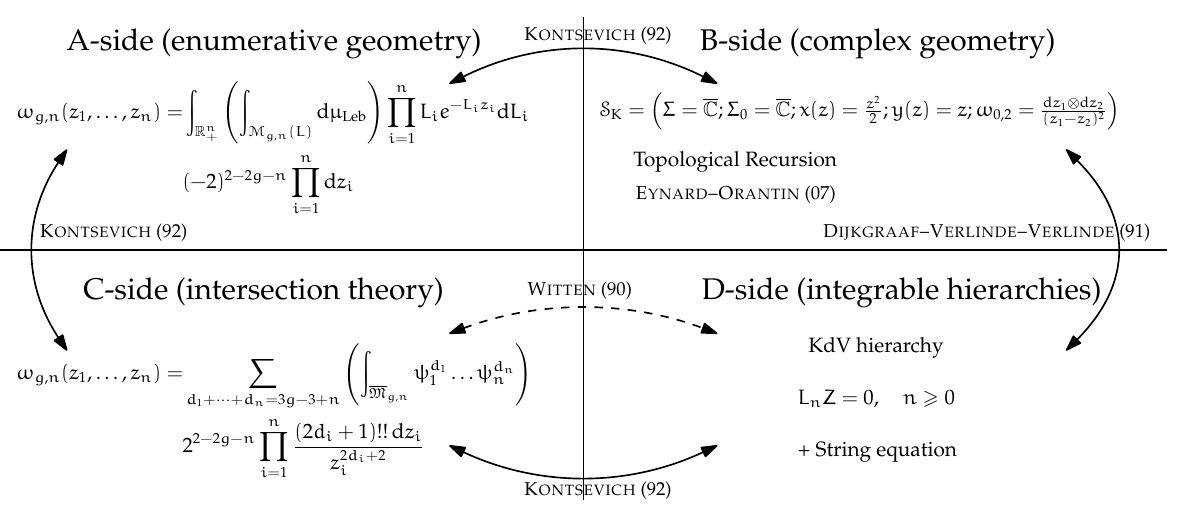} $$
Since the early 90's, other correspondences among the four sides A, B, C, and D have been found such as simple Hurwitz numbers (A), which are computed by topological recursion (conjectured in \cite{BM07} and proved in \cite{BEMS,EMS11}) on the Lambert curve (B). They are related by the ELSV formula to the Hodge class \cite{ELSV01} (C) and determined by the Toda hierarchy \cite{Oko00} (D). Many other relations between the A side and the B side have been established since then, successfully unveiling topological recursion as a universal procedure for counting problems, for instance in knot theory \cite{DFM11, BE15}, Gromov--Witten invariants \cite{BKMP09, EO15, FLZ16}, or the enumeration of maps \cite{Eyn04, AMM05, CE06, CEO06}. This culminated with the notion of quantum Airy structures \cite{KS18}, with spectral curve version developed in \cite{Eyn19}. Their higher-rank generalisations \cite{BBCCN18} are relevant here in that they offer an alternative point of view on the problem of the present article, namely that of a continuation of Kontsevich's theorem involving moduli spaces of $r$-spin structures over stable (possibly nodal) Riemann surfaces. 

\medskip

Let $r\geq 2$ be an integer, $\Sigma$ be a smooth complex curve of genus $g$ with $n$ marked points $x_1,\dots,x_n\in \Sigma$, and $a_1,\dots,a_n\in\{0,\dots,r-1\}$ such that $\frac{1}{r}\big[(r-2)(g-1)+\sum_{i=1}^{n}a_i\big]\in\mathbb{Z}$. An $r$-spin structure over $\Sigma$ with data $a_1,\dots,a_n$ is a sheaf $\mathcal{T}$ together with an isomorphism:
\[
\phi \,\, : \,\, \mathcal{T}^{\otimes r} \longrightarrow \omega_{\Sigma} \otimes \mathcal{O}(x_1)^{-a_1}\otimes\dots\otimes \mathcal{O}(x_n)^{-a_n}\,,
\]
where $\omega_{\Sigma}$ is the dualising sheaf of $\Sigma$, whose sections are holomorphic $1$-forms. The moduli space of $r$-spin structures on smooth curves with data $a_1,\dots,a_n$ is denoted $\mathfrak{M}^{1/r}_{g;a_1,\dots,a_n}$, and its compactification $\overline{\mathfrak{M}}^{1/r}_{g;a_1,\dots,a_n}$. Forgetting about the $r$-spin structure and keeping only the surface with marked points gives the covering
\[
\pi\,\,:\,\, \overline{\mathfrak{M}}^{1/r}_{g;a_1,\dots,a_n} \overset{r^{2g}:1}{\longrightarrow} \overline{\mathfrak{M}}_{g,n}\;.
\]
In 1993, Witten \cite{Witt93} introduced the idea of a cohomology class $c_{\textup{W}}(a_1,\dots,a_n)\in H^{\bullet}(\overline{\mathfrak{M}}_{g,n}, \mathbb{Q})$ of pure degree, Witten's top Chern class, or Witten's class, or $r$-spin class for short. The construction of the $r$-spin class in genus $0$ was carried out by Witten using $r$-spin structures. In higher genus, the definition of this class is rather involved. It was first properly constructed algebraically by Polishchuk and Vaintrob \cite{PoVa01}. The construction was simplified by Chiodo \cite{Chio06}; at the same time, another construction was proposed by Mochizuki \cite{Moc06}, and later by Fan, Jarvis and Ruan \cite{FJR13}. Witten's class $c_{\textup{W}}(a_1,\dots,a_n)$ has degree:
\[
\deg \left(c_{\textup{W}}(a_1,\dots,a_n) \right)=\frac{(r-2)(g-1)+\sum\limits_{i=1}^{n}a_i}{r}\;.
\]
Intersection numbers of Witten's class with $\psi$-classes are encoded in the generating potential
\[
F^{[r]}(t_0,t_1,\dots) = \sum\limits_{\substack{g\geq 0 \\ n\geq 1}}\hslash^{g-1} \sum\limits_{\substack{0\leq d_1,\dots,d_n\\ 0\leq a_{1},\dots,a_n \leq r-1}} \frac{t_{d_1}^{a_1}\dots t_{d_n}^{a_n}}{n!} \int_{\overline{\mathfrak{M}}_{g,n}} c_{\textup{W}}(a_1,\dots,a_n) \psi_1^{d_1}\dots\psi_n^{d_n}\;,
\]
a formal $\hslash$-series where the integral over $\overline{\mathfrak{M}}_{g,n}$ is worth 0 unless the degree condition
\begin{eqnarray}
(r-2)(g-1)+\sum_{i=1}^{n}(a_i+ r\, d_i)=r(3g-3+n)
\end{eqnarray}
is satisfied. This potential can be recovered from the knowledge of the following differentials:
\[
\begin{split}
\omega_{g,n}^{[r]}(z_1,\dots,z_n)=&(-1)^{g}r^{g-1+n} \sum\limits_{\substack{0\leq d_1,\dots,d_n \\ 0\leq a_1,\dots,a_n\leq r-1}}\prod\limits_{i=1}^{n}\frac{(-1)^{d_i+1} \Gamma\left(d_i+1+\frac{a_i+1}{r}\right)d z_i}{\Gamma\left(\frac{a_i+1}{r}\right) z_i^{r d_i + a_i+2}} \\
&\qquad\qquad\qquad\qquad\qquad\qquad\times\quad \int_{\overline{\mathfrak{M}}_{g,n}} c_{\textup{W}}(a_1,\dots,a_n) \psi_1^{d_1}\dots \psi_{n}^{d_n}\;.
\end{split}
\]
After introducing the class $c_{\textup{W}}(a_1,\dots,a_n)$ and the potential $F^{[r]}$, Witten \cite{Witt93} generalised his previous conjecture to the following one: the potential $F^{[r]}$ (belonging to the C-side) provides a solution of the $r$-KdV hierarchy, also known as the $r^{\textup{th}}$ higher Gelfand--Dikii hierarchy (D-side), generalising the KdV hierarchy. This notion comes with a generalisation of the Virasoro constraints, called $\mathbf W_r$-constraints, highest-weight vector conditions for an infinite-dimensional associative extension $\mathbf W_r$ of the Virasoro algebra, $\mathbf{Vir}\subset\mathbf W_r$, the generators of which have Lie brackets that need not be linear anymore \cite{BoSch95}. If we denote those generators by $\big(\mathcal W_n^{(k)}\big){}_{\underset{2\leq k\leq r}{n\in\mathbb Z}}$, together with a central element acting as a scalar $\mathbf{c}\in\mathbb C$, the $\mathbf W_r$-constraints take the form $\mathcal W_n^{(k)}Z=0$, for all $n\geq 1$ and $2\leq k\leq r$, together with an eigenvector condition for the actions of the zero modes $\big(\mathcal W_0^{(k)}\big)_{2\leq k\leq r}$. This conjecture of Witten was proved in 2010 by Faber, Shadrin and Zvonkine \cite{FSZ10}, with tools from cohomological field theories.
 
\medskip

A natural question in the $r$-spin set-up is: apart from the C and D sides, can we complete the picture and draw connections with sides A and B as well, such as in the Witten's conjecture--Kontsevich's theorem case? So far the answer is only partially affirmative. Our main goal in this article is to fill in this picture, by establishing the connection between A and B, and completing the rest of the connections. Let us describe here the current state of the art:
\begin{itemize}
	\item First, with the same kind of argument as Dijkgraaf, Verlinde and Verlinde \cite{DVV91}, the Virasoro constraints for the $r$-KdV hierarchy are equivalent to the topological recursion (side B) applied to the the spectral curve:
	\[
	\mathcal{S}^{[r]} = \left(\Sigma=\overline{\mathbb{C}}\,\, ;\,\,\Sigma_0=\overline{\mathbb{C}} \,\, ;\,\, x(z)=z^r \,\,;\,\, y(z)=z  \,\, ;\,\, \omega_{0,2}(z_1,z_2) = \frac{dz_1\otimes dz_2}{(z_1-z_2)^2} \right)\;.
	\]
	The ramification point at $z=0$ not being simple, one needs to run \emph{higher} topological recursion \cite{BoEy13}, producing differentials $\omega_{g,n}^{[r]}$ that coincide with the ones defined from the intersection numbers of Witten's class. 
	\item Applying the analysis of a matrix model with external field made in \cite[Appendix D]{EO07inv} by one of the authors and Orantin, and using a theorem with Bouchard \cite{BoEy13} allowing to handle non-simple ramification points, one deduces that the differentials $\omega_{g,n}^{[r]}$ are actually the generalised resolvents of a matrix model (A-side).
This matrix model is a direct generalisation of Kontsevich's one, and the Feynman graphs involved in the expansions of the correlation functions are decorated ribbon graphs. However, the correlation functions that were proved to satisfy topological recursion using matrix model techniques are not the ones that correspond to the desired intersection theory of Witten's class (C-side). This is the main issue that we remedy in this work, by identifying the suitable correlation functions and proving that they also computed by topological recursion.
\end{itemize}
We get the following picture:
$$\includegraphics[width=\textwidth]{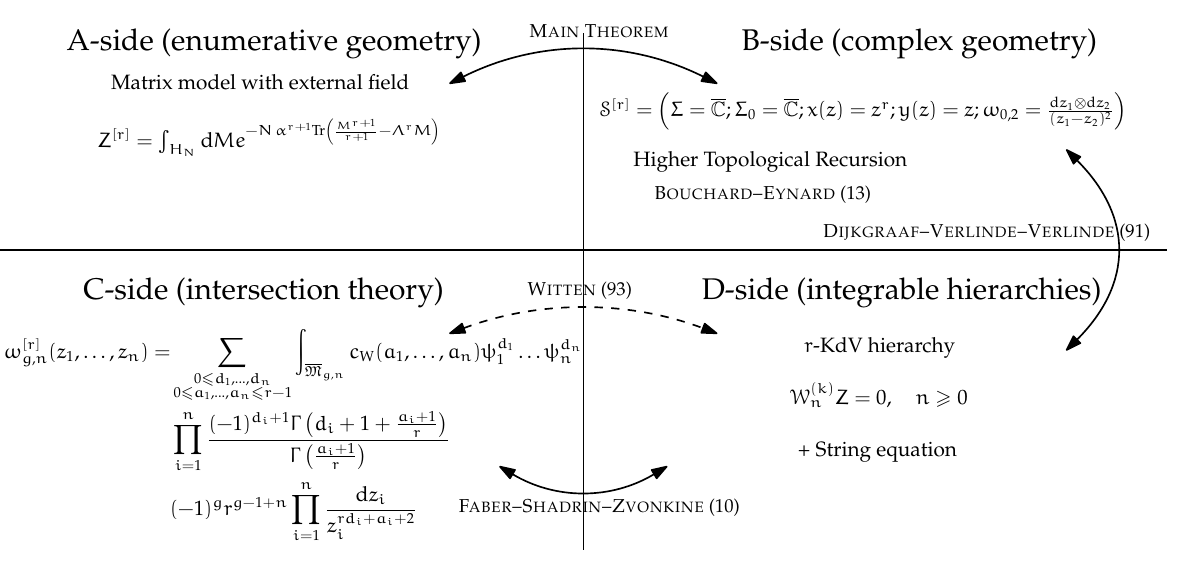} $$
The purpose of this article is to investigate in detail the link between the A-side for the $r$-spin, that is to say the enumeration of ribbon graphs corresponding to the generalisation of the Kontsevich's matrix model, and the B-side, \emph{ie} topological recursion. We will use combinatorial tools to prove the following result:

\vspace{0.3cm}
\textbf{Main Theorem.} The generating functions of generalised Kontsevich ribbon graphs are computed by topological recursion.
\vspace{0.3cm}

Kontsevich's theorem has more than a dozen of different proofs; however, we only know one proof of Witten's $r$-spin conjecture.
This is partly due to the fact that the definition of Witten's class is involved. Studying the enumerative geometry problem related to $r$-spin intersection numbers with purely combinatorial tools is therefore a first step towards a combinatorial description of Witten's class. Since matrix model techniques are not very insightful at a combinatorial level, we choose to restrict ourselves to combinatorial tools.
As an example, instead of deriving a recursive relation between generating functions \emph{via} a loop equation, we rather employ techniques \emph{\`a la} Tutte \cite{Tut68}.

\medskip

Apart from this main motivation, we expose here other aspects that justified the relevance of these objects and result, some of them \emph{a posteriori}:
\begin{itemize}
\item The initial data for the topological recursion of our Main Theorem consists of a very large class of spectral curves for which we give a unified combinatorial interpretation. Together with our method of proof, this fact implies a combinatorial interpretation of the so-called loop equations for this large family of spectral curves. 
\item As mentioned above, the generalised resolvents of the matrix model with external field were already proved to satisfy topological recursion for a certain spectral curve $(\Sigma,\Sigma_0,x,y,\omega_{0,2})$. Our result implies that certain diagonal correlation functions of the matrix model with external field satisfy topological recursion for that same spectral curve after inverting the roles of $x$ and $y$. This provides a combinatorial interpretation for this exchange transformation, which is a symplectic morphism of embedded curves preserving the symplectic form $dx\wedge dy$ (up to sign). The topological recursion applied to these two curves is conjectured to produce some intimately related output, in particular with identical $\omega_{g,0}$'s \cite{EO07inv,EO2MM,EOxy}.
\item Making use of the relation between the matrix model with external field and the $r$-KdV integrable hierarchy (plus string equation) established in \cite{AvM92}, we complete the connection between the A- and D-sides, which together with Faber--Shadrin--Zvonkine's theorem, provides a relation between our graphs (A-side) and the $r$-spin intersection numbers from the D-side. This identification can be seen as an ELSV-like formula, allowing the transfer of knowledge from combinatorics to the intersection theory of the moduli space of curves.
\end{itemize}


We provide some tools to further investigate the implications of topological recursion for the understanding of Witten's class and for the symplectic invariance for a large class of spectral curves. There already exists a conjectural combinatorial interpretation of the exchange transformation $x\leftrightarrow y$ in the context of the $1$-hermitian matrix model, whose generalised resolvents enumerate combinatorial maps and for which the topological recursion applied to the exchanged curve enumerates fully simple maps \cite{BG-F18}, which are defined as maps whose boundaries are required to be simple disjoint polygons. We explore the relation to this open problem in a future work.

\medskip

In the recent works \cite{BHW20,HW21}, the authors conjectured a blobbed topological recursion \cite{BSblob} for the Kontsevich matrix model with quartic potential and proved it for the genus $0$ sector by means of a remarkable global equation that describes the behaviour of the $\omega_{0,n}$ under the global involution of the spectral curve. We would like to investigate if there are interesting relations between our work and their program to establish a blobbed topological recursion, which is an extension of topological recursion in which the initial data is enriched by symmetric holomorphic forms in $n$ variables $(\phi_{g,n})_{2g-2+n>0}$ called \emph{blobs}.

\medskip

The paper is organised as follows: in Section \ref{sec:def:comb:model}, we first define the combinatorial model of generalised Kontsevich graphs as well as the corresponding generating functions, and we derive preliminary results on the same (Subsection \ref{sec:Kgraphs}); we then show in Subsection \ref{sec:tutte} that Tutte's equation is a recursive relation between generating functions; last, in Subsection \ref{sec:properties}, we show analytical properties of those functions using Tutte's equation. In Section \ref{sec:top:rec}, we harvest the fruits of previous results to prove the main theorem of this paper, \emph{i.e.}~that the generating functions do satisfy topological recursion. Last, Section \ref{sec:application} is dedicated to the application of the theorem to the computation of $r$-spin intersection numbers. This article is supplemented with two appendices, the first of which, Appendix \ref{Appendix}, illustrates how the deformation properties of the topological recursion procedure apply nicely in our context to construct a solution of the string equation, while the second one collects our notations.

\pagebreak

\section*{Acknowledgements}
This paper is partly a result of the ERC-SyG project, Recursive and Exact New Quantum Theory (ReNewQuantum) which received funding from the European Research Council (ERC) under the European Union’s Horizon 2020 research and innovation program under grant agreement No. 810573. The authors thank Vincent Bouchard for the initial discussions leading to this work. 

\medskip

R.B.~thanks Sanika Diwanji, the Institut de Physique Th\'eorique (IPhT) of Saclay, the Deutsches Elektronen-Synchrotron (DESY) of Hamburg and the Caisse des Allocations Familiales (CAF) of Paris for their support and hospitality when part of this work was achieved.  

\medskip

S.C.~was supported by the Max-Planck-Gesellschaft, and currently by the CAF of Paris.
S.C.~wants to thank Liselotte Charbonnier and Lucie Neirac for material support, the IPhT of Saclay for a research week, and L\'or\'ant Szegedy for useful discussions. 

\medskip

E.G.-F.~was supported by the public grant ``Jacques Hadamard'' as part of the Investissement d'avenir project, reference ANR-11-LABX-0056-LMH, LabEx LMH and currently receives funding from  the  European  Research Council  (ERC)  under  the  European  Union's Horizon  2020 research and  innovation  programme  (grant  agreement  No.~ERC-2016-STG 716083  ``CombiTop''). She is also grateful to the Institut des Hautes \'{E}tudes Scientifiques (IHES) for its hospitality. 

\medskip

Finally, the authors thank Reinier Kramer for raising his concerns on the convergence issues in the use of Riemann--Hurwitz formula in the proof of Theorem~\ref{thm:w01} that was initially provided, and Nitin Chidambaran for pointing out a hole in the proof of Theorem~\ref{thm:r:Airy} that was previously taken from the literature. 


\medskip

\section{Combinatorial aspects of generalised Kontsevich graphs}
\label{sec:def:comb:model}
This section is dedicated to the combinatorial setup of generalised Kontsevich graphs: the definitions of the sets of maps and of the generating functions, and preliminary results on those generating functions; Tutte's equation; last, analytic properties of the generating functions. 
\subsection{Definitions and preliminary results}
\label{sec:Kgraphs}
Throughout the article, $r\in\mathbb{Z}_{\geq 0},\,r\geq 2$; $N\in\mathbb{Z}_{\geq 0}$, and $\lambda = \{\lambda_1,\dots,\lambda_N\}\subset\mathbb{C}$.
\subsubsection{Sets of generalised Kontsevich graphs}
We define the sets of maps under study in this article. 
\bd\label{def:map}
A \emph{map} is a finite graph without isolated vertices embedded into an oriented compact surface. We require the complement of the graph to be a disjoint union of topological disks, which we call \emph{faces}. A \emph{connected} map is a map whose graph is connected. \\
For a map $G$, we denote by $\mathcal{V}(G)$, $\mathcal{E}(G)$ and $\mathcal{F}(G)$ respectively the vertices, edges and faces of $G$. The degree of a vertex $v$ is its valency, and is denoted $d_v$. 
\hfill $\star$
\ed
There are various models of combinatorial maps that can be studied, so we need to specify the constraints and decorations we want to impose. 
\bd[\emph{Constraints on the vertices}]\label{def:const:vertices}
We allow three types of vertices: black, white, and square vertices. A black vertex $v$ must have a degree $3 \leq d_v \leq r+1 $. White and square vertices can have any positive degree. We impose the following constraint on the maps: each face of the map can be adjacent to at most one white vertex. 
\hfill $\star$
\ed 
For some maps, we may add an additional constraint on the white vertices that we call the star constraint:
\bd[\emph{Star constraint}]\label{def:star:constraint}
A white vertex satisfies the \emph{star constraint} if the corners around it belong to pairwise distinct faces.
\hfill $\star$
\ed
A one-valent white vertex automatically satisfies the star constraint. A typical example of a white vertex satisfying the star constraint is shown on the l.h.s.~of Figure~\ref{fig:starconstraint2}, while on the r.h.s.~of Figure~\ref{fig:starconstraint2}, we show an example of a white vertex which does not satisfy the star constraint. \\
\begin{figure}[H]
\centering
\includegraphics[scale=1]{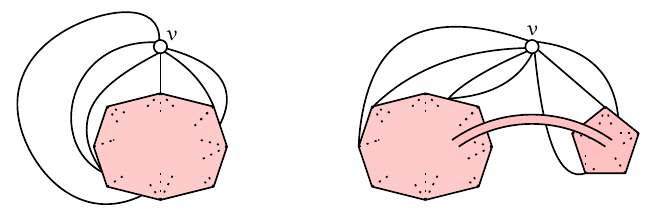}
\caption{On the l.h.s., an example of a white vertex satisfying the star constraint. On the r.h.s., an example of a white vertex \textbf{not} satisfying the star constraint: two corners around the white vertex belong to the external face.}
\label{fig:starconstraint2}
\end{figure}
We denote by $[a_1,\dots,a_k]$ a \emph{cyclically ordered set} -- also called CO-set here -- of $k$ elements. Recall that cyclically ordered sets of $k$ elements satisfy the equivalence relation:
\[
\forall i\in\{2,\dots,k\},\qquad [a_1,\dots,a_k]\sim [a_{i},\dots,a_k,a_1,\dots,a_{i-1}].
\]
We now turn to the four sets of generalised Kontsevich graphs that interest us. Actually, only the sets of Definitions \ref{def:cil:konts} and \ref{def:sq:cil:konts} are of primary interest to us, the other ones being introduced for combinatorial interest as auxiliary sets. In Definitions \ref{def:uncil:konts}--\ref{def:multicil:konts}, we let $g\geq 0$, $n\geq 1$, $k_1,\dots,k_n\geq 1$, $z_1,\dots,z_n\in\mathbb{C}$, and for all $1\leq i\leq n$, $1\leq j\leq k_i$, $z_{i,j}\in\mathbb{C}$. We denote by $S_i$ the CO-set $[z_{i,1},z_{i,2},\dots,z_{i,k_i}]$, and $\underline{k}=(k_1,\dots,k_n)$.
\bd[\emph{Unciliated Kontsevich graphs}]\label{def:uncil:konts}
$G\in\mathcal{F}_{g,n}^{[r]}(z_1,\dots,z_n)$ if:
\begin{itemize}
	\item $G$ is a connected map of genus $g$ having only black vertices;
	\item $G$ has $n$ marked faces labeled from 1 to $n$, and for $i\in\{1,\dots,n\}$, the $i^{\textup{th}}$ marked face is decorated with $z_i$;
	\item the unmarked faces receive decorations that belong to the finite set $\{\lambda_1,\dots, \lambda_N\}$.
\end{itemize}
A map $G$ satisfying those conditions is called an \emph{unciliated Kontsevich graph of type $(g,n)$}. We denote by $\# \Aut(G)$ the cardinal of the automorphism group of $G$.
\hfill $\star$
\ed
Note that if $N=0$, the map $G$ has only marked faces. 

\bd[\emph{Ciliated Kontsevich graphs}]\label{def:cil:konts}
$G\in\mathcal{W}_{g,n}^{[r]}(z_1,\dots,z_n)$ if:
\begin{itemize}
	\item $G$ is a connected map of genus $g$ having only black and white vertices;
	\item $G$ has $n$ white vertices of degree 1, labeled from 1 to $n$, and for $i\in\{1,\dots,n\}$, the face adjacent to the $i^{\textup{th}}$ white vertex is called the $i^{\textup{th}}$ marked face and is decorated with $z_i$;
	\item the unmarked faces receive decorations that belong to the finite set $\{\lambda_1,\dots, \lambda_N\}$.
\end{itemize}
$G\in\mathcal{W}_{g,n}^{[r]}(z_1,\dots,z_n)$ is a \emph{ciliated Kontsevich graph of type $(g,n)$}. Note that in this case, $\Aut G=\{\rm{Id}\}$. 
\hfill $\star$
\ed
The term \emph{ciliated} comes from the fact that each of the $n$ white vertices is only adjacent to one edge. This edge attached to a white vertex can be viewed as a cilium in the corresponding marked face.
\bd[\emph{Square ciliated Kontsevich graphs}]\label{def:sq:cil:konts}
$G\in\mathcal{U}_{g,n}^{[r]}(u;z_1,\dots,z_n)$ if:
\begin{itemize}
	\item $G$ is a connected map of genus $g$;
	\item $G$ has $n$ white vertices of degree 1, labeled from 1 to $n$, and for $i\in\{1,\dots,n\}$, the face adjacent to the $i^{\textup{th}}$ white vertex is called the $i^{\textup{th}}$ marked face and is decorated with $z_i$;
	\item $G$ has only one square vertex $v$ decorated with $u$, moreover the only edge adjacent to the first white vertex is also adjacent to $v$;
	\item the unmarked faces receive decorations that belong to the finite set $\{\lambda_1,\dots, \lambda_N\}$.
\end{itemize}
$G\in\mathcal{U}_{g,n}^{[r]}(u;z_1,\dots,z_n)$ is a \emph{square ciliated Kontsevich graph of type $(g,n)$}. Again, in this case, $\Aut G=\{\rm{Id}\}$. 

\hfill $\star$
\ed
\br
Notice that the square vertex in a square ciliated Kontsevich graph can have any valency: its valency is not restricted to the set $\{3,\dots,r+1\}$ as for black vertices. This is the difference between ciliated and square ciliated Kontsevich graphs. Later, we also assign different weights to square and black vertices. \hfill $\star$
\er
\bd[\emph{Multi-ciliated Kontsevich graphs}]\label{def:multicil:konts}
For $i\in\{1,\dots,n\}$, let $k_i\geq 1$ and $S_i=[z_{i,1},\dots,z_{i,k_i}]$. Let $\underline{k}=(k_1,\ldots,k_n)$. $G\in\mathcal{S}_{g,\underline{k}}^{[r]}(S_1,\dots,S_n)$ if:
\begin{itemize}
	\item $G$ is a connected map of genus $g$ without square vertex;
	\item $G$ has $n$ white vertices, labeled from 1 to $n$, satisfying the star constraint. For $i\in\{1,\dots,n\}$, the $i^{\textup{th}}$ white vertex has degree $k_i$, and is adjacent to faces that are decorated with $z_{i,1},\dots,z_{i,k_i}$ in clockwise order;
	\item the unmarked faces receive decorations that belong to the finite set $\{\lambda_1,\dots, \lambda_N\}$.
\end{itemize}
$G\in\mathcal{S}_{g,\underline{k}}^{[r]}(S_1,\dots,S_n)$ is a \emph{multi-ciliated Kontsevich graph of type $(g,n)$}. Again, $\Aut G=\{\rm{Id}\}$. 
\hfill $\star$
\ed
Note that from Definitions \ref{def:cil:konts} and \ref{def:multicil:konts}, 
\[
\mathcal{W}_{g,n}^{[r]}(z_1,\dots,z_n)=\mathcal{S}_{g,(1,\dots,1)}^{[r]}([z_1],\dots,[z_n]).
\]

\subsubsection{Counting generalised Kontsevich graphs}
The aim of this article is to enumerate the sets defined above. For given type $(g,n)$, the sets 
$$\mathcal{F}_{g,n}^{[r]}(z_1,\dots,z_n),\,\mathcal{W}_{g,n}^{[r]}(z_1,\dots,z_n),\,\mathcal{U}_{g,n}^{[r]}(u;z_1,\dots,z_n)\,\, \mathrm{and}\,\, \mathcal{S}_{g,\underline{k}}^{[r]}(S_1,\dots,S_n)$$
are discrete infinite sets. In order to give sense to enumeration of those sets, we first introduce the degree of a map, second we define local weights of a decorated map, third the discrete measure on the sets. In the end, we obtain generating functions of the sets.
\bd
We define the \emph{degree} of a map $G$:
\[
\deg G = (r+1)\left(\# \mathcal{E}(G)-\#\mathcal{V}(G)\right) = (r+1)\left(\#\mathcal{F}(G)-2+2g\right),
\]
where $g$ is the genus of $G$. We denote by 
$$\mathcal{F}_{g,n}^{[r],\delta}(z_1,\dots,z_n),\,\mathcal{W}_{g,n}^{[r],\delta}(z_1,\dots,z_n),\,\mathcal{U}_{g,n}^{[r],\delta}(u;z_1,\dots,z_n)\,\,\mathrm{and}\,\,\mathcal{S}_{g,\underline{k}}^{[r],\delta}(S_1,\dots,S_n)$$ 
the restrictions of 
$$\mathcal{F}_{g,n}^{[r]}(z_1,\dots,z_n),\,\mathcal{W}_{g,n}^{[r]}(z_1,\dots,z_n),\,\mathcal{U}_{g,n}^{[r]}(u;z_1,\dots,z_n)\,\,\mathrm{and}\,\,\mathcal{S}_{g,\underline{k}}^{[r]}(S_1,\dots,S_n)$$ 
to maps of degree $(r+1)\,\delta$, respectively.
\hfill $\star$
\ed 
The following lemma states that the degree of a map is relevant in order to endow the sets of maps with a grading:
\begin{lem}\label{lem:degree}
For a given degree $(r+1)\,\delta=(r+1)(2g-2+\#\mathcal{F})$ and given topology type $(g,n)$, the sets 
$$\mathcal{F}_{g,n}^{[r],\delta}(z_1,\dots,z_n),\,\mathcal{W}_{g,n}^{[r],\delta}(z_1,\dots,z_n),\,\mathcal{U}_{g,n}^{[r],\delta}(u;z_1,\dots,z_n)\,\, \mathrm{and}\,\, \mathcal{S}_{g,\underline{k}}^{[r],\delta}(S_1,\dots,S_n)$$ 
are finite.  
\hfill $\star$
\end{lem}

\begin{proof}
We present the argument for the set $\mathcal{S}_{g,\underline{k}}^{[r],\delta}(S_1,\dots,S_n)$; the other cases are treated similarly. Let $G\in \mathcal{S}_{g,\underline{k}}^{[r],\delta}(S_1,\dots,S_n)$. Let us prove that the number of vertices of $G$ is bounded by a constant depending only on $g,\,n,\,\delta,\, r$ and $k_1,\dots, k_n$. We denote by $n_d$ the number of $d$-valent vertices and $m=\max\{r+1,k_1,\dots,k_n\}$. Note that for $d>m$, $n_d=0$. What is more, since $G$ has $n$ white vertices, we have $n_1\leq n$ and $n_2\leq n$. The number of edges and the number of vertices of $G$ satisfy:
\[
2 \, \#\mathcal{E}(G)=\sum\limits_{d= 1}^{m}d\, n_d\; ; \qquad
\#\mathcal{V}(G)=\sum\limits_{d= 1}^{m}n_d.
\]
Besides, the number of vertices, edges and faces of $G$ satisfy the Euler relation:
\begin{equation*}
2-2g-\#\mathcal{F}(G)=\#\mathcal{V}(G)-\#\mathcal{E}(G).
\end{equation*}
Combining the three previous equations, we get:
\[
2-2g-\#\mathcal{F}(G)= \frac{n_1}{2}-\frac{1}{2}\sum\limits_{d=3}^{m} (d-2)n_d.
\]
From which we deduce that, for any $d\in\{3,\dots,m\}$:
\begin{equation*} 
n_d\leq \frac{2\delta}{d-2} +\frac{n}{d-2}\leq 2\delta + n.
\end{equation*}
In the end, the number of vertices is bounded by:
\[
\#\mathcal{V}(G)\leq \left(2\delta+n\right)\max\{r+1,k_1,\dots,k_n\},
\]
which depends only on $n$, $r$, $\delta$ and $k_1,\dots,k_n$. In other words, the graphs of degree $\delta$ have a bounded number of vertices, edges and faces. Since the decorations of the faces around marked vertices are fixed and the decorations of the unmarked ones have to be chosen in a finite set, the lemma is proved.
\end{proof}
We now turn to the local weights of a generalised Kontsevich graph. 
\bd\label{def:potential}
The \emph{potential} of the model is a polynomial $V\in\mathbb{C}[z]$ of degree $r+1$:
\begin{equation*}
V(z) = \sum_{j=1}^{r+1} \frac{v_j}{j}\,z^j.
\end{equation*}
We use the notation:
$$\Lambda_i = V'(\lambda_i)\;,\qquad i=1,\dots,N.$$  
\hfill $\star$
\ed
In the following definitions, the parameters $a_i$ belong to the set $\{\lambda_1,\dots,\lambda_N\}\cup\{z_1,\dots,z_n\}$.
\bd
For an edge of a map surrounded by faces (possibly the same) decorated with $a_1,\,a_2$, we define the \emph{propagator}:
\begin{equation*}
{\cal P}(a_1,a_2)=\frac{a_1-a_2}{V'(a_1)-V'(a_2)}.
\end{equation*}
\hfill $\star$
\ed
Observe that ${\cal P}(a_1,a_2)= {\cal P}(a_2,a_1)$ and ${\cal P}(a_1,a_1)=\lim_{a_2\rightarrow a_1}{\cal P}(a_1,a_2)=\frac{1}{V''(a_1)}$.
\bd\label{def:vertex}
A black vertex of degree $3\leq d\leq r+1$ surrounded by faces decorated with $a_1,\dots,a_d$ receives a weight $\mathcal{V}_d (a_1,\dots,a_d)$, which is a symmetric polynomial in its variables:
\begin{equation}\label{vertexWeight}
{\cal V}_d(a_{1},\dots,a_{d})
= \sum_{i=1}^d \frac{-\,V'(a_i)}{\prod_{j\neq i} (a_i-a_j)}.
\end{equation}
A white vertex receives weight 1.\\
A square vertex decorated by $u$ and surrounded by faces decorated with $a_1,\dots,a_d$ receives the weight:
\[
\prod\limits_{i=1}^{d}\frac{1}{u-a_i}.
\] 
\hfill $\star$
\ed

Although $1$-valent and $2$-valent black vertices are not allowed, we still sometimes make use of the vertex weights given by $\eqref{vertexWeight}$ extended to those cases, which read ${\cal V}_1(a_{1})=-V'(a_1)$ and
\begin{equation*}
{\cal V}_2(a_{1},a_2)=-\frac{V'(a_1)-V'(a_2)}{a_1-a_2}.
\end{equation*}
We have ${\cal P}(a_1,a_2) = - 1/{\cal V}_2(a_1,a_2)$.

\begin{lem}
Let $n$ be a non-negative integer and ${\rm schur}_{\lambda}$ the Schur polynomial associated to a partition $\lambda \vdash n$. For $d\geq 1$, we have
\begin{equation*}
{\cal V}_d(a_{1},\dots,a_{d}) = -\sum_{j\geq d} v_j \sum_{l_1+\dots+l_d = j-d} a_1^{l_1}\,a_2^{l_2} \dots a_d^{l_d} = -\sum_{j\geq d} v_j \,{\rm schur}_{j-d}(a_1,\dots,a_d),
\end{equation*}
where $v_j=0$ for $j>r+1$.
\end{lem}
\begin{proof}
Let $V_d(a_{1},\dots,a_{d})$ denote the Vandermonde determinant. For all $i=1,\ldots,d$, 
$$
V_d(a_1,\ldots,a_d)=\prod_{1\leq i < k \leq d}(a_i-a_k)=(-1)^{i-1}V_{d-1}(a_1,\ldots,\widehat{a_i},\ldots,a_d)\prod_{j\neq i}(a_i-a_j).
$$ 
We have
\begin{align*}
&{\cal V}_d(a_{1},\dots,a_{d})
 = \sum_{i=1}^d \frac{-\,V'(a_i)}{\prod_{j\neq i} (a_i-a_j)}=\frac{-1}{V_d(a_{1},\dots,a_{d})}\sum_{i=1}^d(-1)^{i-1}V_{d-1}(a_1,\ldots,\widehat{a_i},\ldots,a_d)V'(a_i)= \\ 
&=\frac{-1}{V_d(a_{1},\dots,a_{d})}\sum_{j=1}^{r+1}v_j\sum_{i=1}^d(-1)^{i-1}V_{d-1}(a_1,\ldots,\widehat{a_i},\ldots,a_d)a_i^{j-1}\\
&= \frac{-1}{V_d(a_{1},\dots,a_{d})}\sum_{j=1}^{r+1}v_j\det \begin{pmatrix} a_1^{j-1} & \cdots & a_d^{j-1}\\ a_1^{d-2} & \cdots & a_d^{d-2} \\ a_1^{d-3} & \cdots & a_d^{d-3}\\ \vdots & \ddots & \vdots \\ 1 & \cdots & 1 \end{pmatrix} = \frac{-1}{V_d(a_{1},\dots,a_{d})}\sum_{j=d}^{r+1}v_j\det \begin{pmatrix} a_1^{j-1} & \cdots & a_d^{j-1}\\ a_1^{d-2} & \cdots & a_d^{d-2} \\ a_1^{d-3} & \cdots & a_d^{d-3}\\ \vdots & \ddots & \vdots \\ 1 & \cdots & 1 \end{pmatrix}\\
&= -\sum_{j\geq d} v_j \,{\rm schur}_{j-d}(a_1,\dots,a_d)= -\sum_{j\geq d} v_j \sum_{l_1+\dots+l_d = j-d} a_1^{l_1}\,a_2^{l_2} \dots a_d^{l_d},
\end{align*}
where the next to last equality is by definition of Schur polynomials ${\rm schur}_{\lambda}$ and the last one is by their charaterisation as sums over semistandard Young tableaux of shape $\lambda$. Since in our case, $\lambda \vdash n$ has length $1$, $\lambda_1=j-d=n$, ${\rm schur}_n$ is the Schur polynomial of the representation $\underbrace{\square\!\square\!\square\!\square\!\square\!\square}_{n\,{\rm boxes}}$.
\end{proof}

It is now possible to associate a weight to a graph by multiplying the local weights:
\bd
The \emph{weight} of a map $G$ is given by:
\begin{equation*}
w(G) = 
\prod_{\substack{e\in \mathcal{E}(G)\\ e=(f_1,f_2)}} {\cal P}(a_{f_1},a_{f_2})
\prod_{\substack{v\in\mathcal{V}(G)\\ \mathrm{black}}}\,{\cal V}_{d_v}(\{a_f\}_{f\mapsto v})
\prod_{\substack{v\in\mathcal{V}(G)\\ \mathrm{square}}}\,\prod_{f\mapsto v} \frac{1}{u-a_f}\;,
\end{equation*}
where $f_1,\,f_2$ are the faces adjacent to $e$, $f\mapsto v$ accounts for the faces adjacent to the vertex $v$ and $a_f$ stands for the weight corresponding to the face $f$.  
\hfill $\star$
\ed
\begin{rem}
In the case $r=2$ and $V(z)=\frac{z^3}{3}$, the black vertices must have valency 3, and the weights are:
\[
\mathcal{P}(a_1,a_2) = \frac{a_1-a_2}{a_1^2-a_2^2} = \frac{1}{a_1+a_2} \;;\qquad \mathcal{V}_{3}(a,b,c) = -1.
\]
Since all the black vertices are of degree 3, an unciliated graph of genus $g$ with $n$ faces (marked and unmarked) has $4g-4+2n$ black vertices, which is even. Therefore, the product of the weights of the vertices is $1$, so in this case, the combinatorial model is equivalent to the model where the weights are:
\[
\mathcal{P}(a_1,a_2) = \frac{1}{a_1+a_2} \;;\qquad \mathcal{V}_{3}(a,b,c) = 1,
\]
and we recover the weights of the graphs introduced by Kontsevich \cite{Kon92} to prove Witten's conjecture \cite{Witt90}. This means that the combinatorial model that we consider is a generalisation of Kontsevich graphs, which we get when we specialise the potential to $V(z)=\frac{z^3}{3}$. 
\hfill $\star$
\end{rem}

The weights of maps define discrete measures on the sets of Kontsevich graphs. The goal of this paper is to compute the following generating functions that are counting functions with respect to those measures.

\bd[\emph{Generating functions}]\label{def:generating:functions}


\begin{eqnarray*}
F_{g,n}^{[r]}(z_1,\dots,z_n;\lambda;v_j;\alpha) &=&\sum\limits_{G\in\mathcal{F}_{g,n}^{[r]}(z_1,\dots,z_n)}\frac{w(G)}{\# \Aut G} \alpha^{-\deg G} \\
 &=& \sum\limits_{\delta \geq (2g+n-2)}\alpha^{-(r+1)\delta} \sum\limits_{G\in\mathcal{F}_{g,n}^{[r],\delta}(z_1,\dots,z_n)}\frac{w(G)}{\# \Aut G}, \\
W_{g,n}^{[r]}(z_1,\dots,z_n;\lambda;v_j;\alpha) &=&\sum\limits_{G\in\mathcal{W}_{g,n}^{[r]}(z_1,\dots,z_n)}\frac{w(G)}{\# \Aut G} \alpha^{-\deg G}\\
 &=& \sum\limits_{\delta\geq (2g+n-2)}\alpha^{-(r+1)\delta} \sum\limits_{G\in\mathcal{W}_{g,n}^{[r],\delta}(z_1,\dots,z_n)}\frac{w(G)}{\# \Aut G}, \\
U_{g,n}^{[r]}(u;z_1,\dots,z_n;\lambda;v_j;\alpha) &=&\sum\limits_{G\in\mathcal{U}_{g,n}^{[r]}(u;z_1,\dots,z_n)}\frac{w(G)}{\# \Aut G} \alpha^{-\deg G}\\
 &=& \sum\limits_{\delta\geq (2g+n-2)}\alpha^{-(r+1)\delta} \sum\limits_{G\in\mathcal{U}_{g,n}^{[r],\delta}(u;z_1,\dots,z_n)}\frac{w(G)}{\# \Aut G}, \\
S_{g;\underline{k}}^{[r]}(S_1,\dots,S_n;\lambda;v_j;\alpha) &=&\sum\limits_{G\in\mathcal{S}_{g;\underline{k}}^{[r]}(S_1,\dots,S_n)}\frac{w(G)}{\# \Aut G} \alpha^{-\deg G}\\
 &=& \sum\limits_{\delta\geq (2g+n-2)}\alpha^{-(r+1)\delta} \sum\limits_{G\in\mathcal{S}_{g;\underline{k}}^{[r],\delta}(S_1,\dots,S_n)}\frac{w(G)}{\# \Aut G}. 
\end{eqnarray*}

\vspace{-0.2cm}

\hfill $\star$
\ed
From Lemma~\ref{lem:degree}, those generating functions are well-defined formal Laurent series in $\alpha^{-1}$.We made the dependence on $\lambda,\,v_j$ and $\alpha$ explicit in the definition of the generating functions. However, since the notations can be heavy to handle, in the rest of the paper this dependence will be implicit, \emph{e.g.} $F_{g,n}^{[r]}(z_1,\dots,z_n)=F_{g,n}^{[r]}(z_1,\dots,z_n;\lambda;v_j;\alpha)$. The generating function $W^{[r]}_{0,1}$ (resp.~$W^{[r]}_{0,2}$) is called the \emph{disc} (resp.~the \emph{cylinder}) amplitude. \\
The $n$-differentials $\omega^{[r]}_{g,n}$, defined for $(g,n)\neq (0,1),\, (0,2)$ (see Definition \ref{def:differentials} later) as
\[
\omega^{[r]}_{g,n}(z_1,\dots,z_n)= W^{[r]}_{g,n}(z_1,\dots,z_n)\prod\limits_{i=1}^{n} V''(z_i)d z_i\;, 
\]
will be important objects in the following.

\br
The parameter $\alpha$ is redundant. Since the weight $w(G)$ is homogeneous of degree $(-\deg G)$, rescaling $\alpha$ by any complex number $c\in\mathbb C$ is equivalent to the rescalings $\lambda_i\to c\lambda_i$, $z_i\to c z_i$ and $v_j\to c^{r+1-j} v_j$. In terms of the formal $\alpha$-series, this reads:
\begin{eqnarray*}  
F^{[r]}_{g,n}(z_1,\dots,z_n;\lambda;v_j;c\alpha)&=&F^{[r]}_{g,n}(c z_1,\dots,c z_n;c\lambda;c^{r+1-j}v_j;\alpha),\cr
W^{[r]}_{g,n}(z_1,\dots,z_n;\lambda;v_j;c\alpha)&=&c^{(r+1)n}\,W^{[r]}_{g,n}(c z_1,\dots,c z_n;c\lambda;c^{r+1-j}v_j;\alpha),\cr
U^{[r],d}_{g,n}(u;z_1,\dots,z_n;\lambda;v_j;c\alpha)&=&c^{(r+1)(n+1)+d}U^{[r],d}_{g,n}(c u;c z_1,\dots,c z_{n};c\lambda;c^{r+1-j}v_j;\alpha),\cr
\text{and}\quad S^{[r]}_{g;\underline{k}}(S_1,\dots,S_n;\lambda;v_j;c\alpha)&=&c^{(r+1)n}\,S^{[r]}_{g;\underline{k}}(c S_1,\dots,c S_n;c\lambda;c^{r+1-j}v_j;\alpha),\cr
&&
\end{eqnarray*} 
for any complex number $c\in\mathbb C$, where $U^{[r],d}_{g,n}$ means that we restricted the defining sum of $U^{[r]}_{g,n}$ to maps with a degree $d$ square vertex.
\hfill $\star$
\er

\bd[\emph{Auxiliary generating function}]\label{def:auxiliary:functions}
For all $(g,n)$, we define the generating function $H^{[r]}_{g,n}$ from $U_{g,n}^{[r]}$ (hence the \emph{auxiliary} attribute) as:
\begin{equation*}
H_{g,n}^{[r]}(u;z_1,\dots,z_n) = V''(z_1)\, \left[V'(u)U^{[r]}_{g,n}(u;z_1,\dots,z_n) \right]_{+},
\end{equation*}
where $[f(u)]_{+}$ stands for the polynomial part of $f(u)$ when $f$ is expanded near $\infty$.
\hfill $\star$
\ed
Actually, those auxiliary generating functions have a combinatorial interpretation. $H_{g,n}^{[r]}(u;z_1,\dots,z_n)$ enumerates maps in the set $\mathcal{U}_{g,n}^{[r]}(u;z_1,\dots,z_n)$, where the local weights of the edges and the black and white vertices are unchanged. The only local weight that differs is the weight of the square vertex: a square vertex decorated by $u$ and surrounded by faces decorated with $a_1,\dots,a_d$ receives a weight
$$\left[V'(u)\prod\limits_{i=1}^{d}\frac{1}{u-a_i}\right]_+ = -\mathcal{V}_{d+1}(u,a_1,\dots,a_d), $$
so it is equivalent to replace the square vertex by a black one, and to insert a corner around this vertex, with label $u$, see Figure \ref{fig:combin:interp:H}.
\begin{figure}
\centering
\includegraphics[scale=1]{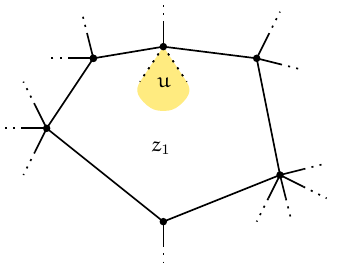}
\caption{Combinatorial interpretation of the graphs contributing to the generating function $H_{g,n}^{[r]}(u;z_1,\dots,z_n)$. The label $u$ is not involved in any propagator, but only in the weight of a vertex.}
\label{fig:combin:interp:H}
\end{figure}
\br
We do not include $\alpha^{-\mathrm{deg}\,G}$ in the weight $w(G)$ of a map. However, in the following, when the picture of a map $G$ is present in an equation, it stands for its weight along with its degree: $w(G) \alpha^{-\mathrm{deg}\,G}$.  \hfill $\star$
\er

\subsubsection{Preliminary results and combinatorial identities}

We now turn to preliminary results that will be useful for the following sections. First, it follows from Definition \ref{def:vertex} that:

\begin{lem}\label{lem:vertex} For $i, j, k, m\geq 1$, the following equalities on the weights of vertices and the propagators hold:
\beq\label{eqcutvertexk}
 \frac{{\cal V}_{m}(a_1,a_3,\dots,a_{m+1})-{\cal V}_{m}(a_2,a_3,\dots,a_{m+1})}{a_1-a_2}
 =  {\cal V}_{m+1}(a_1,a_2,a_3,\dots,a_{m+1}),
\eeq

\beq\label{eqcutpropag}
\frac{{\cal P}(a_i,a_k)-{\cal P}(a_j,a_k)}{a_i-a_j} = {\cal P}(a_i,a_k)\,{\cal V}_3(a_i,a_j,a_k)\,{\cal P}(a_j,a_k),
\eeq

\beq\label{eqdzVk}
\frac{\partial {\cal V}_m(a_{1},\dots,a_{m})}{\partial a_1} =  {\cal V}_{m+1}(a_1,a_{1},a_2,\dots,a_{m}),
\eeq

\beq\label{eqdzP}
\frac{\partial{\cal P}(a_i,a_k)}{\partial a_i}
= \,{\cal P}(a_i,a_k)\,{\cal V}_3(a_i,a_i,a_k)\,{\cal P}(a_i,a_k),
\eeq

and, in general
\beq\label{eqderVk}
{\cal V}_{m+\ell}(\overbrace{a_1,\dots,a_1}^{m+1\,{\rm times}}, a_2,\dots, a_{\ell})
= \frac{1}{m!}\,\frac{\partial^m}{\partial a_1^m}\,{\cal V}_{\ell}(a_1,a_2,\dots,a_{\ell}),\;\; \text{ for } \ell\geq 1.
\eeq
\hfill $\star$
\end{lem}
\begin{lem}\label{lem:tadpole} For any integer $k\geq 0$, we have
\begin{eqnarray*}
\sum_{m\geq 1} {\cal V}_{k+m}(a_1,\dots,a_k,\overbrace{z,\dots,z}^{m})\,\,\big(W^{[r]}_{0,1}(z)\big)^{m-1}
&=& {\cal V}_{k+1}\big(a_1,\dots,a_k,z+W^{[r]}_{0,1}(z)\big)\cr & &\quad+\quad \delta_{k,0}\Big(V'(z)-V'\big(z+W^{[r]}_{0,1}(z)\big)\Big).
\end{eqnarray*} 
\hfill $\star$
\end{lem}
Graphically, Lemma \ref{lem:tadpole} states that:
\[\includegraphics[scale=1]{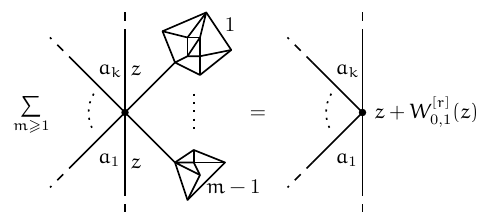}\]
This representation justifies the name \emph{tadpole resummation}: the tadpoles (numbered from $1$ to $m-1$ in the previous picture, and which represent disc amplitudes) can be summed up around a vertex.\\
\begin{proof}
Start by considering the case $k=0$. For $m\geq 1$, we have ${\cal V}_{m}(\overbrace{z,\dots,z}^{m}) = -\frac{V'^{(m-1)}(z)}{(m-1)!}$ from equation \eqref{eqderVk}, and thus
\begin{eqnarray*} 
\sum_{m\geq 1} {\cal V}_{m}(\overbrace{z,\dots,z}^{m})\,\,W^{[r]}_{0,1}(z)^{m-1}
&=& -\sum_{m\geq 1} \frac{V'^{(m)}(z)}{m!}W^{[r]}_{0,1}(z)^{m}
= - V'(z+W^{[r]}_{0,1}(z))
= \mathcal{V}_1 (z+W^{[r]}_{0,1}(z)).
\end{eqnarray*} 
In the case where $k\geq1$, we again apply formula \eqref{eqderVk}, and get:
\begin{eqnarray*} 
\sum_{m\geq1} {\cal V}_{k+m}(a_1,\dots,a_k,\overbrace{z,\dots,z}^{m})\,\,\big(W^{[r]}_{0,1}(z)\big)^{m-1}&=&\sum_{m\geq1}\frac{1}{(m-1)!}\frac{\partial^{m-1}}{\partial z^{m-1}}\mathcal{V}_{k+1}(a_1,\dots,a_k,z)W^{[r]}_{0,1}(z)^{m-1}\cr
&=& \mathcal{V}_{k+1}\Big(a_1,\dots,a_k,z+W^{[r]}_{0,1}(z)\Big).
\end{eqnarray*}
\end{proof}

\subsubsection{Relations between the generating functions}
We introduced four types of sets of generalised Kontsevich graphs, a priori disjoint, as well as the counting functions of those sets with respect to discrete measures. The aim of the present article is to compute those generating functions, so a natural question is: do relations between those generating functions exist? Stated in another way, do those generating functions carry the same amount of information? This section is dedicated to the answer of this question, synthetised here:
\begin{thm}\label{thm:relations}
The following relations hold:
\begin{eqnarray}\label{eq:uncil:to:cil}
W^{[r]}_{g,n}(z_1,\dots,z_n) &=& \frac{1}{V''(z_1)}\frac{\partial}{\partial z_1}\dots \frac{1}{V''(z_n)}\frac{\partial}{\partial z_n} F^{[r]}_{g,n}(z_1,\dots,z_n)\cr
&& +\delta_{g,0}\delta_{n,2}\left(\frac{1}{V''(z_1)V''(z_2)(z_1-z_2)^2}-\frac{1}{(V'(z_1)-V'(z_2))^2}\right) \cr
&& +\delta_{g,0}\delta_{n,1} \sum\limits_{j=1}^{N}\left(\frac{1}{V''(z_1)(z_1-\lambda_j)}-\frac{1}{(V'(z_1)-V'(\lambda_j))}\right).
\end{eqnarray}
For $(g,n)\neq(0,1)$:
\begin{equation}\label{eq:multicil:to:cil}
S^{[r]}_{g;\underline{k}}(S_1,\dots,S_n)=\sum\limits_{j_1=1}^{k_1}\dots\sum\limits_{j_n=1}^{k_n}\frac{W^{[r]}_{g,n}(z_{1,j_1},\dots,z_{n,j_n})}{\prod\limits_{m=1}^{n}\alpha^{(k_m-1) (r+1)}\prod\limits_{\substack{i_m=1 \\ i_m\neq j_m}}^{k_m}\left(V'(z_{m,j_m})-V'(z_{m,i_m})\right)}.
\end{equation}
Last:
\begin{equation}\label{eq:square:to:cil}
-\underset{u=\infty}{\Res}\,\,\dd u\,\, V'\left(u\right)\, U^{[r]}_{g,n}(u;z_1,\dots,z_n) = \delta_{g,0}\delta_{n,1}\frac{V'(z_1)}{V''(z_1)} \alpha^{r+1},
\end{equation}
\begin{equation}\label{eq:square:to:cil:2}
-\underset{u=\infty}{\Res}\,\,\dd u \,\, V'(u)\,(u-z_1)U^{[r]}_{g,n}(u;z_1,\dots,z_n) =\frac{V'(z_1)}{V''(z_1)}W^{[r]}_{g,n}(z_1,\dots,z_n)+ \delta_{g,0}\delta_{n,1}\left(\frac{N}{V''(z_1)}\right),
\end{equation}
and
\begin{equation}\label{eq:square:to:cil:3}
\underset{u\to \infty}{\mathrm{lim}}\left(u^2\,U^{[r]}_{g,n}(u;z_1,\dots,z_n)-\delta_{g,0}\delta_{n,1}\frac{\alpha^{r+1}}{V''(z_1)(u-z_1)}\right) = \frac{W^{[r]}_{g,n}(z_1,\dots,z_n)}{V''(z_1)}.
\end{equation}
\hfill $\star$
\end{thm}
From the theorem, we see that the knowledge of the $U^{[r]}_{g,n}$'s entails the knowledge of the $W^{[r]}_{g,n}$'s, of the $S^{[r]}_{g;\underline{k}}$'s and after integration, of the $F^{[r]}_{g,n}$'s. The proof relies on three lemmas:
\begin{lem}\label{lem:uncil:to:cil}
Generating functions of the unciliated Kontsevich graphs and ciliated Kontsevich graphs are related by the formula:
\begin{eqnarray*}
W^{[r]}_{g,n}(z_1,\dots,z_n) &=& \frac{1}{V''(z_1)}\frac{\partial}{\partial z_1}\dots \frac{1}{V''(z_n)}\frac{\partial}{\partial z_n} F^{[r]}_{g,n}(z_1,\dots,z_n)\cr
&& +\delta_{g,0}\delta_{n,2}\left(\frac{1}{V''(z_1)V''(z_2)(z_1-z_2)^2}-\frac{1}{(V'(z_1)-V'(z_2))^2}\right)\cr
&& +\delta_{g,0}\delta_{n,1} \sum\limits_{j=1}^{N}\left(\frac{1}{V''(z_1)(z_1-\lambda_j)}-\frac{1}{(V'(z_1)-V'(\lambda_j))}\right).
\end{eqnarray*}
\hfill $\star$
\end{lem}

\begin{lem}\label{lem:multicil:to:cil}
Let $\underline{k}'=(k_1-1,k_2,\dots,k_n)$ and $\underline{k}''=(1,k_2,\dots,k_n)$. Then:
\begin{eqnarray}\label{eq:multicil:to:cil:1}
S^{[r]}_{g;\underline{k}}(S_1,\dots,S_n)&=&\frac{1}{\alpha^{r+1}}\frac{S^{[r]}_{g;\underline{k}'}\left([z_{1,1},z_{1,3},\dots,z_{1,k_1}],S_2,\dots,S_n\right)-S^{[r]}_{g;\underline{k}'}\left([z_{1,2},\dots,z_{1,k_1}],S_2,\dots,S_n\right)}{V'(z_{1,1})-V'(z_{1,2})} \cr
& & + \delta_{g,0}\delta_{n,1}\delta_{k_1,2}\mathcal{P}(z_{1,1},z_{1,2}).
\end{eqnarray}
Applying this formula $k_1-1$ times, this yields:
\begin{equation}\label{eq:multicil:to:cil:2}
S^{[r]}_{g;\underline{k}}(S_1,\dots,S_n)=\frac{1}{\alpha^{(k_1-1)(r+1)}}\sum\limits_{j=1}^{k_1}\frac{S^{[r]}_{g;\underline{k}''}([z_{1,j}],S_2,\dots,S_n)+\delta_{g,0}\delta_{n,1}\alpha z_{1,j}}{\prod\limits_{\substack{i=1\\ i\neq j}}^{k_1}\left(V'(z_{1,j})-V'(z_{1,i})\right)}.
\end{equation}
\hfill $\star$
\end{lem}

\br
It is a remarkable feature of this model that a local property valid for the weight of a single vertex extends to a similar property at the macroscopic level, that is to say at the level of generating functions. Indeed, equation \eqref{eq:multicil:to:cil:1} of Lemma~\ref{lem:multicil:to:cil} is the macroscopic equivalent of the local equation \eqref{eqcutvertexk} of Lemma~\ref{lem:vertex}. \hfill $\star$
\er

\begin{lem}\label{lem:square:to:cil}
For $1\leq d \leq r+1$,
\[
\mathcal{V}_{d}(a_1,\dots,a_d)=\underset{u=\infty}{\Res}\,\dd u\, V'(u)\,\, \prod\limits_{i=1}^{d}\frac{1}{u-a_i}.
\]
\hfill $\star$
\end{lem}
\begin{lem}\label{lem:poly:part}
Concerning the auxiliary generating functions $H^{[r]}_{g,n}$, the following points hold:
\begin{itemize}
	\item $H^{[r]}_{0,1}(u;z_1)$ is a polynomial of degree $r-1$ in $u$ whose top coefficient is $v_{r+1}\alpha^{r+1}$. For $(g,n)\neq (0,1)$, $H^{[r]}_{g,n}(u;z_1,\dots,z_n)$ is a polynomial of degree $r-2$ in $u$.
	\item The coefficient of $u^{r-2}$ in the polynomial  $H^{[r]}_{g,n}(u;z_1,\dots,z_n)$ is:
	\begin{equation}\label{eq:top:coef:H}
	v_{r+1} W^{[r]}_{g,n}(z_1,\dots,z_n)+\delta_{g,0}\delta_{n,1}\alpha^{r+1}\left(v_{r+1}\, z_1 + v_r\right).
	\end{equation}
	\item We have the identity:
	\begin{equation}\label{eq:mult:by:u}
	V''(z_1)\left[u\, V'(u)\, U^{[r]}_{g,n}(u;z_1,\dots,z_n)\right]_{+} =  u\, H^{[r]}_{g,n}(u;z_1,\dots,z_n)+ \delta_{g,0}\delta_{n,1} \alpha^{r+1} V'(z_1).
	\end{equation}
\end{itemize}
\hfill $\star$
\end{lem}

We do not give the proof of Lemma \ref{lem:square:to:cil} since it comes from a simple direct computation. 
\begin{proof}[Proof of Lemma \ref{lem:uncil:to:cil}]

We first prove the lemma in the generic case $(g,n)\neq(0,2),\, (0,1)$. Consider $G\in\mathcal{F}^{[r]}_{g,n}(z_1,\dots,z_n)$. The contribution of the $i^{\textup{th}}$ marked unciliated face $f_i$ (carrying parameter $z_i$) to $\mathcal{W}(G)$ is given by:
\beq\label{eq:contribtoweight}
\prod_{\substack{v\in\mathcal{V}(G)\\ v\mapsto  f_i}} {\cal V}_{d_v}(z_i,\{a_{f}\}_{\substack{f\mapsto v\\ \, f\neq f_i}})
\prod_{\substack{f\in\mathcal{F}(G)\\ f\mapsto f_i} } {\cal P}(z_i,a_f),
\eeq
where:
\begin{itemize}
\item $v\in\mathcal{V}(G), v\mapsto  f_i$ stands for the (black) vertices that are adjacent to the $i^{\textup{th}}$ unciliated marked face $f_i$ and $f\mapsto v, f\neq f_i$ stands for the rest of the faces around $v$;
\item $f\in\mathcal{F}(G), f\mapsto f_i$ designates the faces adjacent to $f_i$.
\end{itemize}
The principle of the proof is to take the derivative of this expression with respect to $z_i$, and to remark that it is equivalent to adding a cilium attached to a white vertex into the $i^{\textup{th}}$ marked face $f_i$ in all the possible manners. Taking the derivative of \eqref{eq:contribtoweight} with respect to $z_i$ and using relations \eqref{eqdzVk} and \eqref{eqdzP} of Lemma~\ref{lem:vertex}, we get:
\begin{eqnarray*}
\sum_{\substack{v\in\mathcal{V}(G)\\ v\mapsto  f_i}} {\cal V}_{d_{v}+1}(z_i,z_i,\{a_{f}\}_{\substack{f\mapsto v\\ \, f\neq f_i}})\prod_{\substack{v'\in\mathcal{V}(G)\\ v'\mapsto  f_i, v'\neq v}} {\cal V}_{d_{v'}}(z_i,\{a_{f'}\}_{\substack{f'\mapsto v'\\ \, f'\neq f_i}})\,\, \prod_{\substack{f\in\mathcal{F}(G)\\ f\mapsto f_i}} {\cal P}(z_i,a_f) \cr
+ \sum_{\substack{f\in\mathcal{F}(G)\\ f\mapsto f_i}} {\cal V}_{3}(z_i,z_i,a_{f})\,{\cal P}(z_i,a_{f})^2\prod_{\substack{v\in\mathcal{V}(G)\\ v\mapsto  f_i}} {\cal V}_{d_v}(z_i,\{a_{f}\}_{f\mapsto v})\,\,\prod_{\substack{f'\in\mathcal{F}(G)\\ f'\mapsto f_i, f'\neq f}} {\cal P}(z_i,a_{f'}).
\end{eqnarray*}

The first sum corresponds to the cases where we add a cilium to the $i^{\textup{th}}$ marked face of $G$ and attach it to a vertex neighboring the marked face, whereas the second one corresponds to the cases where the cilium is attached to an edge surrounding the marked face. In the end, if we multiply this expression by $1/V''(z_i)$ (the weight of the propagator of the added cilium), we obtain all the possible insertions of a cilium (along with its white vertex) in the $i^{\textup{th}}$ marked face. For the degree of the map, inserting a cilium in both cases does not change the degree of the map. Note that unciliated maps may have automorphisms coming from rotating the marked faces and these are killed in the process of adding cilia to each marked face.\\ 
At the level of the generating functions, this means that:
\begin{equation*}
W^{[r]}_{g,n}\left(z_1,\dots,z_n\right) = \frac{1}{V''(z_1)}\frac{\partial}{\partial z_1}\dots\frac{1}{V''(z_n)}\frac{\partial}{\partial z_n}  F^{[r]}_{g,n}\left(z_1,\dots,z_n \right),
\end{equation*}
which proves the lemma in the generic case.\\

In the case $(g,n)=(0,2)$, the proof follows the same steps, except that two maps in $\mathcal{W}^{[r]}_{0,2}(z_1,z_2)$ are not obtained by adding cilia $z_1|z_1$ and $z_2|z_2$ to a map in $\mathcal{F}^{[r]}_{0,2}(z_1,z_2)$. Indeed, the following maps:\\
\[
\includegraphics[scale=1]{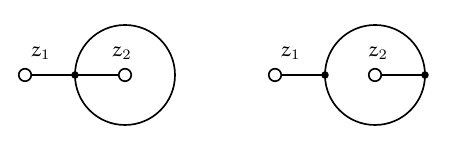}
\]
of degree 0 are obtained by adding cilia $z_1|z_1$ and $z_2|z_2$ to the map 
\[
\includegraphics[scale=1]{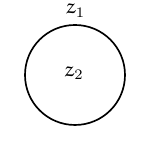}
\]
which is not in $\mathcal{F}^{[r]}_{0,2}(z_1,z_2)$. Therefore, we must add their weights to $W^{[r]}_{0,2}(z_1,z_2)$, which are worth:
\begin{eqnarray*}
\mathcal{P}(z_1,z_1)\mathcal{P}(z_2,z_2)\mathcal{P}(z_1,z_2)^{2}\,\mathcal{V}_{3}(z_1,z_1,z_2)\mathcal{V}_{3}(z_1,z_2,z_2)+\mathcal{P}(z_1,z_1)\mathcal{P}(z_2,z_2)\mathcal{P}(z_1,z_2)\mathcal{V}_{4}(z_1,z_1,z_2,z_2)&&\cr
=\frac{1}{V''(z_1)V''(z_2)(z_1-z_2)^2}-\frac{1}{(V'(z_1)-V'(z_2))^2}.&&
\end{eqnarray*}
This explains the additional terms in the second line of equation \eqref{eq:uncil:to:cil}. \\

In the same manner for $(g,n)=(0,1)$, the following maps:
\[
\includegraphics[scale=1]{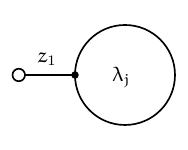}
\]
of degree 0, where $j$ runs over $\{1,\dots,N\}$, are obtained by adding the cilium $z_1|z_1$ to the maps
\[
\includegraphics[scale=1]{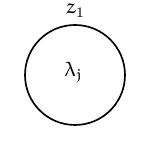}
\]
which are not in $\mathcal{F}^{[r]}_{0,1}(z_1)$. Therefore, as above, we must add their weights to $W^{[r]}_{0,1}(z_1)$, which are worth
$$\sum\limits_{j=1}^{N}\left(\frac{1}{V''(z_1) (z_1-\lambda_j)}- \frac{1}{V'(z_1)-V'(\lambda_j)}\right), $$
hence the additional terms in the third line of equation \eqref{eq:uncil:to:cil}.
\end{proof}

\begin{proof}[Proof of Lemma \ref{lem:multicil:to:cil}]
Since equations \eqref{eq:multicil:to:cil} and \eqref{eq:multicil:to:cil:2} are deduced recursively from equations \eqref{eq:multicil:to:cil:1}, it is sufficient to prove the latter for any $k_1\geq 2$. \\ 
\begin{figure}
    \centering
    \begin{minipage}[t]{0.48\textwidth}
		\centering
    \includegraphics[scale=1]{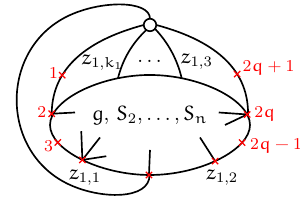}
    \end{minipage}
    \begin{minipage}[t]{0.48\textwidth}
		\centering
    \includegraphics[scale=1]{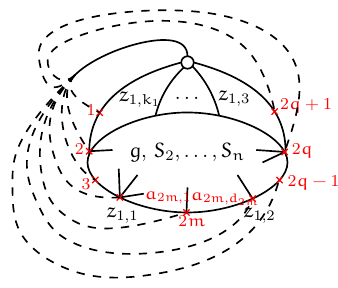}
    \end{minipage}
    \caption{On the l.h.s., a map with a multi-ciliated face with parameters $(z_{1,1},z_{1,2},\dots,z_{1,k_1})$. On the r.h.s., the cilium $z_{1,1}|z_{1,2}$ (along with its white vertex) can be inserted at any place along the boundary between $z_{1,k_1}|z_{1,1}$ and $z_{1,2}|z_{1,3}$, either on an edge or at a vertex. The possible sites of insertion are denoted by red crosses, and the labels neighboring a site $m$ are $a_{m,1},\dots,a_{m,d_m}$.}
    \label{fig:graphOkrec2}
\end{figure}
Consider a graph $G\in\mathcal{S}^{[r]}_{g;\underline{k}}(S_1,\dots,S_n)$ and suppose that $k_1\geq 3$ (the case $k_1=2$ is considered afterwards). The first white vertex carries the parameters $S_1=[z_{1,1},\dots,z_{1,k_1}]$.
Since $G$ satisfies the star constraint, after removing the edge $z_{1,1}|z_{1,2}$, the boundary of the face between cilia $z_{1,k_1}|z_{1,1}$ and $z_{1,2}|z_{1,3}$ is a sequence of vertices and propagators. This sequence is emphasised in Figure \ref{fig:graphOkrec2} by the red crosses, that we call \emph{sites}. We assign the labels $1,2,\dots,2 q+1$ to the sites, ordered clockwise from the white vertex. Sites labeled by odd integers are propagators, while the even sites are vertices, and we denote $a_{m,1},\dots,a_{m,d_m}$ the set of decorations of the faces adjacent to the site. For a propagator $d_{2m+1}=1$ ; the decorations satisfy the constraints
\[
a_{2m+1,1} = a_{2(m+1),1}\, ; \qquad a_{2m,d_m} = a_{2m+1,1}.
\] 
see r.h.s.~of Figure~\ref{fig:graphOkrec2}. Let us now insert the cilium $z_{1,1}|z_{1,2}$ into all possible sites along that boundary. It can be inserted either on an edge or on a vertex. See r.h.s.~of Figure \ref{fig:graphOkrec2}.
Let us assume that we insert it in position $m$. We define respectively the \emph{left} and \emph{right} functions:
\begin{equation*}
\ell_m(z) = \prod_{\substack{v\in\mathcal{V}(G)\\ v\leq m}} {\cal V}_{1+d_v}(z,a_{v,1},\dots,a_{v,d_v}) \prod_{\substack{e\in\mathcal{E}(G)\\ e\leq m}} {\cal P}(z,a_{e,1}),
\end{equation*}
\begin{equation*}
r_m(z) = \prod_{\substack{v\in\mathcal{V}(G)\\ v\geq m}} {\cal V}_{1+d_v}(z,a_{v,1},\dots,a_{v,d_v}) \prod_{\substack{e\in\mathcal{E}(G)\\ e\geq m}} {\cal P}(z,a_{e,1}).
\end{equation*}
The products are made on the vertices and edges among the sites we introduced. For the empty products we define
\begin{equation*}
r_{2q+2}(z) = 1
\quad , \quad 
\ell_{0}(z) = 1.
\end{equation*}
The contribution of the sum over all possible insertion of the $z_{1,1}|z_{1,2}$ cilium to the weight is:
\begin{eqnarray}  
\label{eq:insertion:1}
\frac{1}{\alpha^{r+1}}\mathcal{P}(z_{1,1},z_{1,2})\left[\sum_{m=1}^{q}\ell_{2m-1}(z_{1,1})\mathcal{V}_{2+d_{2m}}(z_{1,1},z_{1,2},a_{2m,1},\dots,a_{2m,d_{2m}})r_{2m+1}(z_{1,2})\right.\cr
\left. +\sum_{m=0}^{q-1}\ell_{2m}(z_{1,1})\mathcal{P}(z_{1,1},a_{2m+1,1})\mathcal{V}_3(z_{1,1},a_{2m+1,1},z_{1,2})\mathcal{P}(z_{1,2},a_{2m+1,1})r_{2m+2}(z_{1,2})\right].
\end{eqnarray}
Note that adding the edge increases the degree of the map by $(r+1)$: this is why, we multiplied by $\alpha^{-(r+1)}$. Those sums contribute to $S^{[r]}_{g,\underline{k}}(S_1,\dots,S_n)$.\\
If one removes the edge $z_{1,1}|z_{1,2}$ attached to the first white vertex, we obtain a map whose first white vertex carries parameters $[z_{1,1},z_{1,3},\dots,z_{1,k_1}]$ (l.h.s.~of  Figure  \ref{fig:graphOkrec3}) or $[z_{1,2},z_{1,3},\dots,z_{1,k_1}]$ (r.h.s.~of  Figure  \ref{fig:graphOkrec3}).
In the first case, the contribution to $S^{[r]}_{g;\underline{k}'}([z_{1,1},z_{1,3},\dots,z_{1,k_1}],S_2,\dots,S_n)$ of the part of the graph between $z_{1,k_1}$ and $z_{1,3}$ is $\ell_{2q+1}(z_{1,1})$. In the second case, the contribution to $S^{[r]}_{g;\underline{k}'}([z_{1,2},z_{1,3},\dots,z_{1,k_1}],S_2,\dots,S_n)$ of the part of the graph between $z_{1,k_1}$ and $z_{1,3}$ is $r_{1}(z_{1,2})$.\\
In the end, the contributions of the part of the graph between $z_{1,k_1}$ and $z_{1,3}$ to 
$$S^{[r]}_{g;\underline{k}'}([z_{1,1},z_{1,3},\dots,z_{1,k_1}],S_2,\dots,S_n)-S^{[r]}_{g;\underline{k}'}([z_{1,2},z_{1,3},\dots,z_{1,k_1}],S_2,\dots,S_n)$$
is
\begin{equation} 
\label{eq:insertion:2}
\ell_{2q+1}(z_{1,1})-r_{1}(z_{1,2}).
\end{equation}
\begin{figure}
    \centering
    \begin{minipage}[t]{0.48\textwidth}
		\centering
    \includegraphics[scale=1]{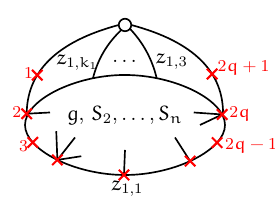}
    \end{minipage}
    \begin{minipage}[t]{0.48\textwidth}
		\centering
    \includegraphics[scale=1]{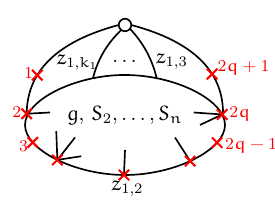}
    \end{minipage}
    \caption{The 2 possible maps after removing the edge $z_{1,1}|z_{1,2}$.}
    \label{fig:graphOkrec3}
\end{figure}
It remains to relate the contributions of equations \eqref{eq:insertion:1} and \eqref{eq:insertion:2}.\\
The contribution \eqref{eq:insertion:2} can be written as a telescopic sum:
\begin{equation*} 
\ell_{2q+1}(z_{1,1})-r_{1}(z_{1,2})=\sum_{m=1}^{2q+1}\ell_{m}(z_{1,1})r_{m+1}(z_{1,2})-\ell_{m-1}(z_{1,1})r_m(z_{1,2}).
\end{equation*} 
We distinguish two cases:
\begin{itemize}
\item if the site is a vertex, we have:
\begin{equation*} 
 \ell_{2m}(z_{1,1}) =\ell_{2m-1}(z_{1,1}) {\cal V}_{1+d_{2m}}(z_{1,1},a_{2m,1},\dots,a_{2m,d_{2m}}) 
 ,
\end{equation*} 
\begin{equation*} 
r_{2m}(z_{1,2}) =r_{2m+1}(z_{1,2}) {\cal V}_{1+d_{2m}}(z_{1,2},a_{2m,1},\dots,a_{2m,d_{2m}}).
\end{equation*}
Equation \eqref{eqcutvertexk} of Lemma \ref{lem:vertex} gives then
\begin{eqnarray*}
\frac{\ell_{2m}(z_{1,1})r_{2m+1}(z_{1,2})-\ell_{2m-1}(z_{1,1})r_{2m}(z_{1,2})}{z_{1,1}-z_{1,2}} &=& \ell_{2m-1}(z_{1,1}){\cal V}_{2+d_{2m}}(z_{1,1},z_{1,2},a_{2m,1},\dots,a_{2m,d_{2m}}) \cr
&& \quad\times\quad  r_{2m+1}(z_{1,2}).
\end{eqnarray*}
\item Similarly, if the site is an edge:
\begin{equation*}
 \ell_{2m+1}(z_{1,1}) =\ell_{2m}(z_{1,1}) {\cal P}(z_{1,1},a_{2m+1,1}) 
 , \quad 
r_{2m+1}(z_{1,2}) =r_{2m+2}(z_{1,2}) {\cal P}(z_{1,2},a_{2m+1,1}) ,
\end{equation*}
and again, due to equation \eqref{eqcutpropag} of Lemma \ref{lem:vertex}, we have
\begin{eqnarray*}
\frac{\ell_{2m+1}(z_{1,1})r_{2m+2}(z_{1,2})-\ell_{2m}(z_{1,1})r_{2m+1}(z_{1,2})}{z_{1,1}-z_{1,2}}
&=&\ell_{2m}(z_{1,1}){\cal P}(z_{1,1},a_{2m+1,1}){\cal V}_{3}(z_{1,1},z_{1,2},a_{2m+1,1})\cr
&& \quad\times\quad {\cal P}(z_{1,2},a_{2m+1,1}) r_{2m+2}(z_{1,2}).
\end{eqnarray*}
\end{itemize}
It follows that
\begin{eqnarray*}
\frac{\ell_{2q+1}(z_{1,1}) -r_{1}(z_{1,2}) }{z_{1,1}-z_{1,2}}
&=&  \sum_{m=1}^{q} \ell_{2m-1}(z_{1,1}){\cal V}_{2+d_{2m}}(z_{1,1},z_{1,2},a_{2m,1},\dots,a_{2m,d_{2m}}) r_{2m+1}(z_{1,2}) \cr
&+&  \sum_{m=0}^{q-1} \ell_{2m}(z_{1,1}){\cal P}(z_{1,1},a_{2m+1,1}){\cal V}_{3}(z_{1,1},z_{1,2},a_{2m+1,1}){\cal P}(z_{1,2},a_{2m+1,1})r_{2m+2}(z_{1,2}).\cr
&&
\end{eqnarray*}
This means that 
$$\frac{1}{\alpha^{r+1}}\frac{1}{V'(z_{1,1})-V'(z_{1,2})}\times\eqref{eq:insertion:2}=\eqref{eq:insertion:1},$$
which corresponds, at the level of generating functions, to equation \eqref{eq:multicil:to:cil:1} for $k_1\geq 3$.

\medskip

For $k_1=2$, the arguments are the same, except for the special case $g=0,\,n=1$. In this case, one has to consider also the map made of one propagator, the edge $z_{1,1}|z_{1,2}$, with both ends attached to a bivalent white vertex. It contributes to $S^{[r]}_{0,(2)}([z_{1,1},z_{1,2}])$. Its weight is simply $\mathcal{P}(z_{1,1},z_{1,2})$, and this graph is why the term with the delta function appears in equation \eqref{eq:multicil:to:cil:1}.
\end{proof}

\begin{proof}[Proof of equations \eqref{eq:square:to:cil}, \eqref{eq:square:to:cil:2} and \eqref{eq:square:to:cil:3}]
In the same way as in the previous proofs, we consider first the general case $(g,n)\neq(0,1)$, and then prove the formulas for $(g,n)=(0,1)$. For equation \eqref{eq:square:to:cil}, the result comes from the following compensation:
\[\includegraphics[scale=1]{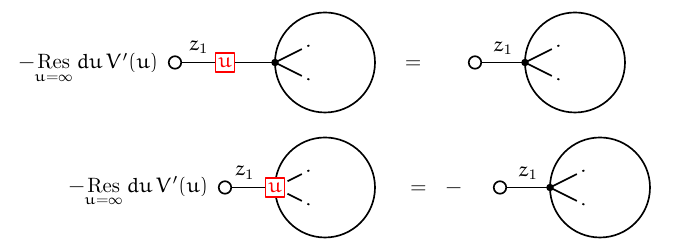}\]
where the equalities come respectively from the formulas for $1/{\cal P}(z_1,z_1)=V''(z_1)=-{\cal V}_2(z_1,z_1)$ and for $-{\cal V}_d(z_1,z_1,a_3,\ldots,a_d)$ from Lemma~\ref{lem:square:to:cil}.
In the case $(g,n)=(0,1)$, the following residue does not have any compensation:
$$-\underset{u=\infty}{\mathrm{Res}}\, \dd u \,V'(u)\includegraphics[scale=1]{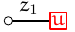} = \frac{V'(z_1)}{V''(z_1)}\alpha^{r+1}. $$
Therefore:
\[
-\underset{u=\infty}{\mathrm{Res}}\, \dd u \, V'(u) U^{[r]}_{g,n}(u,z_1,\dots,z_n)=\delta_{g,0}\delta_{n,1} \frac{V'(z_1)}{V''(z_1)}\alpha^{r+1}.
\]
Second, for equation \eqref{eq:square:to:cil:2}, we have the following compensation:
\[\includegraphics[scale=1]{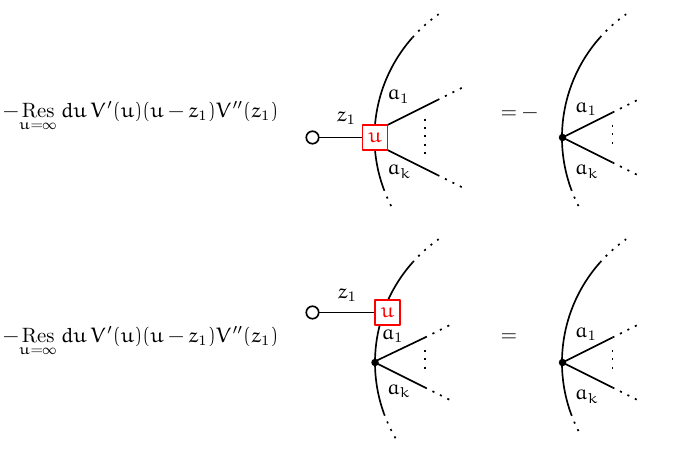} \]
However, the following residues do no get canceled by any other term:
\begin{equation}\label{degree2}
\includegraphics[scale=1]{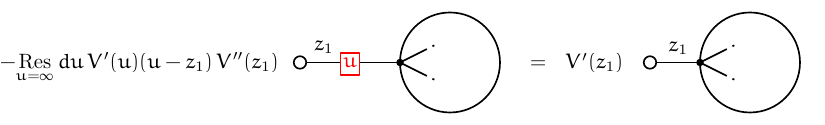}
\end{equation}
On the l.h.s., the map belongs to $\mathcal{U}^{[r]}_{g,n}(u;z_1,\dots,z_n)$, while on the r.h.s., the map belongs to $\mathcal{W}^{[r]}_{g,n}(z_1,\dots,z_n)$ and has the same degree. Hence:
\begin{eqnarray*}
-\underset{u=\infty}{\mathrm{Res}}\, \dd u \, V'(u) (u-z_1) V''(z_1) U^{[r]}_{g,n}(u;z_1,\dots,z_n) &=&\cr \sum\limits_{G\in \mathcal{U}^{[r]}_{g,n}(u;z_1,\dots,z_n)} -\underset{u=\infty}{\mathrm{Res}}\, \dd u \, V'(u) (u-z_1) V''(z_1) w(G) \alpha^{-\textup{deg} G} &=& \cr
 \sum\limits_{G\in \mathcal{W}^{[r]}_{g,n}(z_1,\dots,z_n)} V'(z_1) w(G) \alpha^{-\textup{deg} G} &=& V'(z_1) W^{[r]}_{g,n}(z_1,\dots,z_n),
\end{eqnarray*}
where the residue and the sum can be exchanged between the first and the second line since we consider formal power series in $\alpha$. This yields the result for $(g,n)\neq(0,1)$. For $(g,n)=(0,1)$, the compensations between maps of $\mathcal{U}^{[r]}_{0,1}(u;z_1)$ work in the same manner as in the generic case, except for the $N$ maps of this type (one per $\lambda_i$):
\[\includegraphics[scale=1]{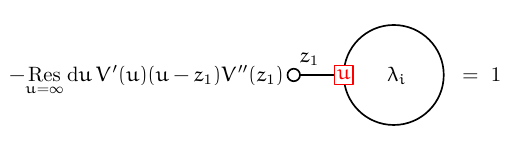}\]
The following term does not contribute:
\begin{equation}\label{degree1}\includegraphics[scale=1]{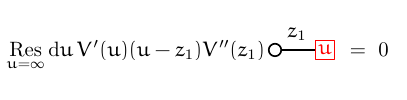} \end{equation}
Therefore:
\[
-\underset{u=\infty}{\mathrm{Res}}\,\dd u \, V'(u)\, (u-z_1) V''(z_1) U^{[r]}_{0,1}(u;z_1)=V'(z_1) W^{[r]}_{0,1}(z_1)+ N
\]
and the proof is complete.

Finally, for equation~\eqref{eq:square:to:cil:3}, we also start by computing the limit for $(g,n)\neq (0,1)$. In this case, for every $G\in\mathcal{U}^{[r]}_{g,n}(u;z_1,\dots,z_n)$, the degree of the square vertex is $\geq 2$. If it is $\geq 3$, it does not contribute to the limit. If it has degree $2$, as in~\eqref{degree2}, its contribution to the limit is exactly the same as the weight associated to the ciliated graph in $\mathcal{W}^{[r]}_{g,n}(z_1,\dots,z_n)$ obtained from $G$ by erasing the square vertex, and divided by $V''(z_1)$. This gives 
$$
\underset{u\to \infty}{\mathrm{lim}}u^2\,U^{[r]}_{g,n}(u;z_1,\dots,z_n)= \frac{W^{[r]}_{g,n}(z_1,\dots,z_n)}{V''(z_1)}, \text{ for } (g,n)\neq (0,1).
$$

For $(0,1)$, we can also have the graph whose square vertex has degree $1$, as in \eqref{degree1}, whose contribution to $\mathcal{U}^{[r]}_{0,1}(u;z_1)$ is $\frac{\alpha^{r+1}}{V''(z_1)(u-z_1)}$. This completes the proof:
\begin{equation}
\underset{u\to \infty}{\mathrm{lim}}\left(u^2\,U^{[r]}_{g,n}(u;z_1,\dots,z_n)-\delta_{g,0}\delta_{n,1}\frac{\alpha^{r+1}}{ V''(z_1) (u-z_1)}\right) = \frac{W^{[r]}_{g,n}(z_1,\dots,z_n)}{V''(z_1)}.
\end{equation}
\end{proof}

\begin{proof}[Proof of Lemma \ref{lem:poly:part}]
Let us prove the three points.\\
\textbf{First  and second points.} In order to get the degree of the polynomial $H^{[r]}_{g,n}$, one needs the smallest degree of $U^{[r]}_{g,n}$ when the latter is expanded in powers of $1/u$ when $u\to \infty$. For $(g,n)=(0,1)$, the smallest degree of $U^{[r]}_{0,1}(u;z_1)$ comes from the only map with a one-valent square vertex:
$$\includegraphics{./images/example_N_0_1.pdf} =\frac{\alpha^{r+1}}{V''(z_1)(u-z_1)} = \frac{\alpha^{r+1}}{V''(z_1)}\left(\frac{1}{u}+\frac{z_1}{u^2}+\mathcal{O}\left(\frac{1}{u^3}\right)\right). $$
The maps with a 2-valent square vertex give a contribution:
$$\frac{W^{[r]}_{0,1}(z_1)}{V''(z_1)(u-z_1)^2} = \frac{W^{[r]}_{0,1}(z_1)}{V''(z_1)}\left(\frac{1}{u^2}+\mathcal{O}\left(\frac{1}{u^3}\right)\right). $$
The polynomial $H^{[r]}_{0,1}(u;z_1)=V''(z_1)\left[V'(u) U^{[r]}_{0,1}(u;z_1)\right]_+$ is therefore given by:
\[
\begin{split}
H^{[r]}_{0,1}(u;z_1) &= \left[(v_{r+1}\,u^r + v_r\, u^{r-1} + \mathcal{O}(u^{r-2}) )\left(\frac{\alpha^{r+1}}{u}+ \frac{\alpha^{r+1} z_1 + W^{[r]}_{0,1}(z_1)}{u^2} + \mathcal{O}\left(\frac{1}{u^3}\right)\right)\right]_+ \\
&= v_{r+1} u^{r-1} + \left(v_{r+1} W^{[r]}_{0,1}(z_1) + \alpha^{r+1}(v_{r+1}\,z_1+ v_r)\right) u^{r-2} + \mathcal{O}\left(u^{r-3}\right), 
\end{split}
\]
which gives the first and second points for $(g,n)=(0,1)$. 
\medskip \\
For $(g,n)\neq (0,1)$, the smallest degree of $U^{[r]}_{g,n}(u;z_1,\dots,z_n)$ comes from the maps with a 2-valent vertex, which give a contribution:
$$\frac{W^{[r]}_{g,n}(z_1,\dots,z_n)}{V''(z_1)(u-z_1)^2} = \frac{W^{[r]}_{g,n}(z_1,\dots,z_n)}{V''(z_1)}\left(\frac{1}{u^2}+\mathcal{O}\left(\frac{1}{u^3}\right)\right),  $$
So the polynomial $H^{[r]}_{g,n}(u;z_1,\dots,z_n)$ is given by:
\[
\begin{split}
H^{[r]}_{g,n}(u;z_1,\dots,z_n)&= \left[(v_{r+1}\,u^r + \mathcal{O}(u^{r-1}) ) \left(\frac{W^{[r]}_{g,n}(z_1,\dots,z_n)}{u^2}+\mathcal{O}\left(\frac{1}{u^3}\right)\right)\right]_{+}\\
&= v_{r+1}\, W^{[r]}_{g,n}(z_1,\dots,z_n) u^{r-2} + \mathcal{O}(u^{r-3}),
\end{split}
\]
which proves the first and second points for generic $(g,n)$.
\medskip \\
\textbf{Third point.} Denoting by $[Q(u)]_{0}$ the constant coefficient of the polynomial $Q$ and by $I=\{z_2,\dots,z_n\}$, we have the identity:
\[
\begin{split}
V''(z_1) \left[u\, V'(u) U^{[r]}_{g,n}(u;z_1,I)\right]_{+} &=u\, V''(z_1) \left[V'(u) U^{[r]}_{g,n}(u;z_1,I)\right]_{+} + V''(z_1) \left[u\, V'(u) U^{[r]}_{g,n}(u;z_1,I)\right]_{0}\\
&= u\, H^{[r]}_{g,n}(u;z_1,I) +V''(z_1) \left[u\, V'(u) U^{[r]}_{g,n}(u;z_1,I)\right]_{0}.
\end{split}
\]
The last term is deduced from equation \eqref{eq:square:to:cil}:
$$V''(z_1) \left[u\, V'(u) U^{[r]}_{g,n}(u;z_1,I)\right]_{0} = - V''(z_1) \underset{u=\infty}{\mathrm{Res}} \dd u \, V'(u) U^{[r]}_{g,n}(u;z_1,I) = \delta_{g,0}\delta_{n,1}\alpha^{r+1}\, V'(z_1).$$
We hence obtain equation \eqref{eq:mult:by:u}.
\end{proof}

\subsection{Tutte's equation}\label{sec:tutte}
To get a recursive equation on the generating functions $U^{[r]}_{g,n}$ and $W^{[r]}_{g,n}$, we adopt a strategy similar to the one employed by Tutte \cite{Tut68}: starting with a graph $G\in\mathcal{U}^{[r]}_{g,n}(u;z_1,\dots,z_n)$, we erase the first marked white vertex and its unique adjacent edge; this edge was attached to the square vertex of $G$, we introduce a bivalent white vertex on the following edge around the square vertex in the clockwise direction. From those Tutte's equations, we deduce recursive relations involving $H^{[r]}_{g,n}$ and $W^{[r]}_{g,n}$. The interest of relations involving $H^{[r]}_{g,n}$ rather than $U^{[r]}_{g,n}$ is that they are polynomials in $u$. This feature plays an important role when we derive the solutions for the generating functions. 

\subsubsection{Case of the disc}
There are three cases to consider:
\begin{enumerate}
	\item In the first case, the following edge is adjacent to a face decorated with $\lambda_j$, $j\in\{1,\dots,N\}$, see the following figure:
	\[\includegraphics[scale=1]{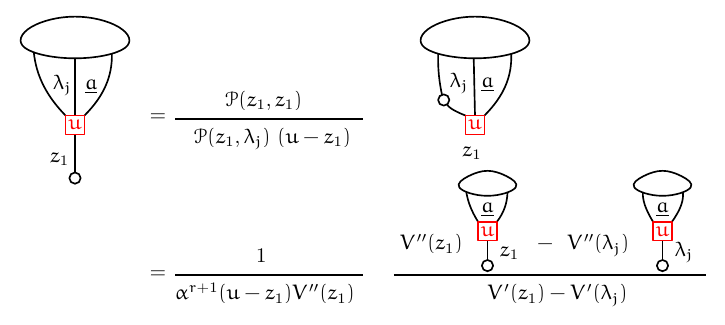} \]
	The equality between the first and the second line of the picture comes from a similar argument as the one used in the proof of Lemma \ref{lem:multicil:to:cil} (note that the map on the r.h.s.~satisfies the star constraint so the lemma can be applied). Those terms give a contribution:
	\begin{equation}\label{eq:tutte:disc:contrib:1}
	\frac{\alpha^{-(r+1)}}{u-z_1}\frac{1}{V''(z_1)}\sum\limits_{j=1}^{N} \frac{V''(z_1)U^{[r]}_{0,1}(u;z_1)-V''(\lambda_j)U^{[r]}_{0,1}(u;\lambda_j)}{V'(z_1)-V'(\lambda_j)}.
	\end{equation}
	\item In the second case, the following edge is adjacent to the first marked face, see the following figure:
	$$\includegraphics[scale=1]{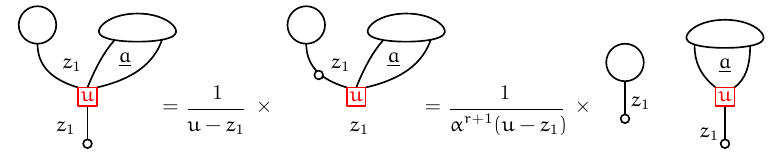} $$
	Those terms give a contribution:
	\begin{equation}\label{eq:tutte:disc:contrib:2}
	\frac{\alpha^{-(r+1)}}{u- z_1}W^{[r]}_{0,1}(z_1) U^{[r]}_{0,1}(u;z_1). 
	\end{equation}
	\item The third case is specific to the topology $(g,n)=(0,1)$. It is the case where there is no following edge around the square vertex:
	\[\includegraphics[scale=1]{./images/example_N_0_1.pdf} =  \frac{\alpha^{r+1}}{u-z_1}\frac{1}{V''(z_1)}\] 
	This map, of degree $-(r+1)$, has the weight:
	\begin{equation}\label{eq:tutte:disc:contrib:3}
	\frac{\alpha^{r+1}}{u-z_1}\frac{1}{V''(z_1)}.
	\end{equation}
\end{enumerate}
In the end, Tutte's equation for the disc is given by 
$$U^{[r]}_{0,1}(u;z_1)=\eqref{eq:tutte:disc:contrib:1}+\eqref{eq:tutte:disc:contrib:2}+\eqref{eq:tutte:disc:contrib:3},$$ 
which we can rewrite as: 
\begin{eqnarray}\label{eq:tutte:disc:1}
U^{[r]}_{0,1}(u;z_1)\left[(u-z_1)\alpha^{r+1}-\left(W^{[r]}_{0,1}(z_1)+\sum\limits_{j=1}^{N}\frac{1}{V'(z_1)-V'(\lambda_j)}\right)\right]&=&-\sum\limits_{j=1}^{N}\frac{V''(\lambda_j)}{V''(z_1)}\frac{U^{[r]}_{0,1}(u;\lambda_j)}{V'(z_1)-V'(\lambda_j)}\cr
&& +\frac{\alpha^{2(r+1)}}{V''(z_1)}.
\end{eqnarray}
From this equation and using equation \eqref{eq:mult:by:u} of Lemma \ref{lem:poly:part}, we obtain the Tutte equation for the disc, in terms of $H^{[r]}_{0,1}$ and $W^{[r]}_{0,1}$:
\begin{eqnarray}\label{eq:tutte:disc:H}
\left[u-z_1-\alpha^{-(r+1)} W^{[r]}_{0,1}(z_1)-\sum\limits_{j=1}^{N}\frac{\alpha^{-(r+1)}}{V'(z_1)-V'(\lambda_j)}\right] \, H^{[r]}_{0,1}(u;z_1)&=&-\alpha^{-(r+1)}\sum\limits_{j=1}^{N}\frac{H^{[r]}_{0,1}(u;\lambda_j)}{V'(z_1)-V'(\lambda_j)}\cr
&& +\alpha^{(r+1)}\left(V'(u)-V'(z_1)\right).
\end{eqnarray}

\subsubsection{Generic case}
For $(g,n)\neq (0,1)$, there are four cases to consider. Denote $I=\{z_2,\dots,z_n\}$, $I_j=I\backslash\{z_j\}$, $J\sqcup J'=I$, $h+h'=g$. 
\begin{enumerate}
	\item In the first case, the following edge is adjacent to a face decorated with $\lambda_j$, $j\in\{1,\dots,N\}$, see the following figure:
	\[\includegraphics[scale=1]{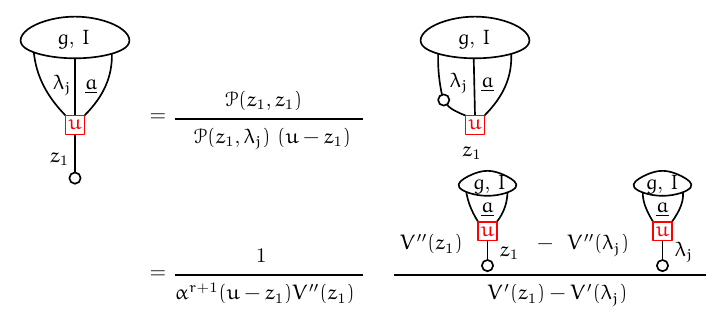} \]
	Like in the case of the disc, the equality between the first and the second line of the picture comes from a similar argument as the one used in the proof of Lemma \ref{lem:multicil:to:cil} (note that the map on the r.h.s.~satisfies the star constraint so the lemma can be applied). Those terms give a contribution:
	\begin{equation}\label{eq:tutte:generic:contrib:1}
	\frac{\alpha^{-(r+1)}}{u-z_1}\frac{1}{V''(z_1)}\sum\limits_{j=1}^{N} \frac{V''(z_1)U^{[r]}_{g,n}(u;z_1,I)-V''(\lambda_j)U^{[r]}_{g,n}(u;\lambda_j,I)}{V'(z_1)-V'(\lambda_j)}.
	\end{equation}
	\item In the second case, the following edge is adjacent to a face decorated with $z_m$, $m\in\{2,\dots,n\}$, see the following figure:
	$$\includegraphics[scale=1]{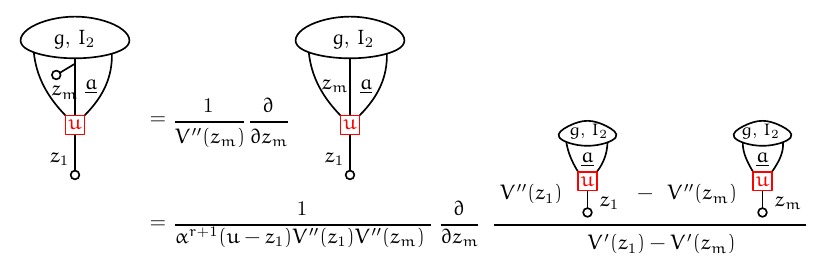} $$
	To get the first equality, we use the arguments of the proof of Lemma \ref{lem:uncil:to:cil} in order to remove the $m^{\textup{th}}$ white vertex. Then, we recover the situation of the previous case, replacing $\lambda_j$ by $z_m$, and we show that those terms contribute as:
	\begin{equation}\label{eq:tutte:generic:contrib:2}
	\frac{\alpha^{-(r+1)}}{u-z_1}\frac{1}{V''(z_1)}\sum\limits_{m=2}^{n} \frac{1}{V''(z_m)}\frac{\partial}{\partial z_m}\frac{V''(z_1)U^{[r]}_{g,n-1}(u;z_1,I_m)-V''(z_m)U^{[r]}_{g,n-1}(u;z_m,I_m)}{V'(z_1)-V'(z_m)}.
	\end{equation}
	\item In the third case, the following edge is adjacent to the first marked face, see the following figure:
	\[\includegraphics[scale=1]{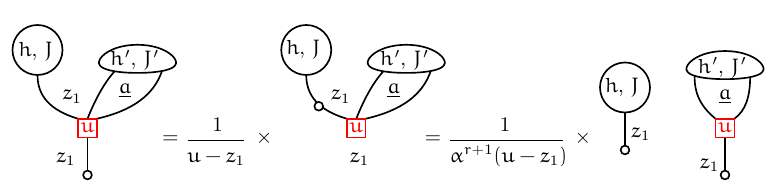} \]
	Those terms give a contribution:
	\begin{equation}\label{eq:tutte:generic:contrib:3}
	\frac{\alpha^{-(r+1)}}{u-z_1}\sum\limits_{\substack{h+h'=g\\J\sqcup J'= I}}W^{[r]}_{h,1+\#J}(z_1,J) U^{[r]}_{h',1+\#J'}(u;z_1,J'). 
	\end{equation}
	\item In the fourth case, the following edge is also adjacent to the first marked face, in the following configuration:
	\[\includegraphics[scale=1]{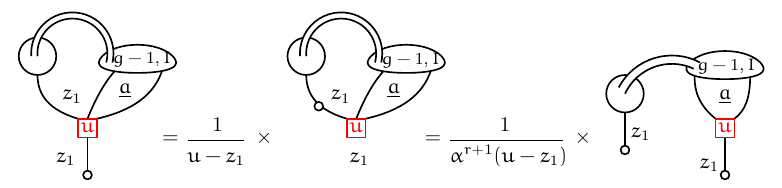} \]
	Those terms give a contribution:
	\begin{equation}\label{eq:tutte:generic:contrib:4}
	\frac{\alpha^{-(r+1)}}{u-z_1}U^{[r]}_{g-1,n+1}(u;z_1,z_1,I).
	\end{equation}
\end{enumerate}
In the end, Tutte's equation in the generic case is given by 
$$U^{[r]}_{g,n}(u;z_1,I)=\eqref{eq:tutte:generic:contrib:1}+\eqref{eq:tutte:generic:contrib:2}+\eqref{eq:tutte:generic:contrib:3}+\eqref{eq:tutte:generic:contrib:4},$$
which can be recast as:
\begin{eqnarray}\label{eq:tutte:generic:1}
\left[(u-z_1)\alpha^{r+1}-\left(W^{[r]}_{0,1}(z_1)+\sum\limits_{j=1}^{N}\frac{1}{V'(z_1)-V'(\lambda_j)}\right) \right]U^{[r]}_{g,n}(u;z_1,I)&=\cr
-\quad\sum\limits_{j=1}^{N}\frac{V''(\lambda_j)U^{[r]}_{g,n}(u;\lambda_j,I)}{V''(z_1)(V'(z_1)-V'(\lambda_j))}&\cr
+ \quad\sum\limits_{m=2}^{n}\frac{1}{V''(z_1)V''(z_m)}\frac{\partial}{\partial z_m}\frac{V''(z_1)U^{[r]}_{g,n-1}(u;z_1,I_m)-V''(z_m)U^{[r]}_{g,n-1}(u;z_m,I_m)}{V'(z_1)-V'(z_m)}&\cr
+\sum\limits_{\substack{h+h'=g\\ J\sqcup J'=I}}^{'}W^{[r]}_{h,1+\#J}(z_1,J)U^{[r]}_{h',1+\#J'}(u;z_1,J')&\cr
+\quad U^{[r]}_{g-1,n+1}(u;z_1,z_1,I)&\!\!\!\!\!\!\!\!\!\!\!\!\!\!,
\end{eqnarray}
where in the sum $\sum^{'}$, we exclude the term $h=0,\,J=\emptyset$. In the same manner as for the disc, we use the definition of $H^{[r]}_{g,n}$ and \eqref{eq:mult:by:u} of Lemma \ref{lem:poly:part} to get the Tutte equation in the generic case in terms of $H^{[r]}_{g,n}$ and $W^{[r]}_{g,n}$:
\begin{eqnarray}\label{eq:tutte:generic:H}
\left[u-z_1-\alpha^{-(r+1)}W^{[r]}_{0,1}(z_1) - \sum\limits_{j=1}^{N}\frac{\alpha^{-(r+1)}}{V'(z_1)-V'(\lambda_j)} \right]H^{[r]}_{g,n}(u;z_1,I)&=\cr
\alpha^{-(r+1)} H^{[r]}_{0,1}(u;z_1) W^{[r]}_{g,n}(z_1,I)-\alpha^{-(r+1)}\sum\limits_{j=1}^{N}\frac{H^{[r]}_{g,n}(u;\lambda_j,I)}{V'(z_1)-V'(\lambda_j)}&\cr
+\quad\sum\limits_{m=2}^{n}\frac{\alpha^{-(r+1)}}{V''(z_m)}\frac{\partial}{\partial z_m}\frac{H^{[r]}_{g,n-1}(u;z_1,I_m)-H^{[r]}_{g,n-1}(u;z_m,I_m)}{V'(z_1)-V'(z_m)}&\cr
+\quad\alpha^{-(r+1)}\sum\limits_{\substack{h+h'=g\\ J\sqcup J'=I}}^{'}W^{[r]}_{h,1+\#J}(z_1,J)H^{[r]}_{h',1+\#J'}(u;z_1,J')&\cr
+\quad\alpha^{-(r+1)}H^{[r]}_{g-1,n+1}(u;z_1,z_1,I)&\!\!\!\!\!\!\!\!\!\!\!\!\!\!.
\end{eqnarray}
where now the sum $\sum^{'}$ excludes the terms $h=0,\,J=\emptyset$ and $h'=0,\,J'=\emptyset$.

\begin{ex}
Consider the case $N=0$. Applying equations \eqref{eq:tutte:disc:1} and \eqref{eq:square:to:cil:2}, we get for the disc:
\[
W^{[r]}_{0,1}(z_1)=0  ; \qquad U^{[r]}_{0,1}(u;z_1) = \frac{1}{u-z_1}\frac{\alpha^{r+1}}{V''(z_1)}.
\]
This is consistent with the direct computation of those generating functions from the enumeration of maps (no map contributes to $W^{[r]}_{0,1}$ in this case, and only one contributes to $U^{[r]}_{0,1}$):
$$\includegraphics{./images/example_N_0_1.pdf} $$
For the cylinder, \emph{i.e.} $(g,n)=(0,2)$, we use equation \eqref{eq:tutte:generic:1} and equation \eqref{eq:square:to:cil:2} to get:
\[
\begin{split}
W^{[r]}_{0,2}(z_1,z_2) &= \frac{-1}{(V'(z_1)-V'(z_2))^2}+\frac{1}{V''(z_1)V''(z_2)(z_1-z_2)^2} \\
U^{[r]}_{0,2}(u;z_1,z_2) &= \frac{1}{V''(z_1)}\left(\frac{1}{u- z_1}\right)^{2}\left[\frac{-1}{(V'(z_1)-V'(z_2))^2}+\frac{1}{V''(z_1)V''(z_2)(z_1-z_2)^2}\right] \\
&+ \frac{1}{V''(z_1)}\frac{1}{u- z_1}\frac{1}{u - z_2}\left[\frac{1}{u- z_1}\frac{z_1-z_2}{(V'(z_1)-V'(z_2))^2}-\frac{1}{u- z_2}\frac{1}{V''(z_2)(V'(z_1)-V'(z_2))}\right].
\end{split}
\]
Again, those formulas are consistent with the direct computation from the maps: 2 maps contribute to $W^{[r]}_{0,2}$ in this case (see Figure \ref{fig:example:N:0:3}), and 4 maps contribute to $U^{[r]}_{0,2}$ (Figure \ref{fig:example:N:0:2}). \hfill $\star$
\begin{figure}
\centering
\includegraphics[scale=1]{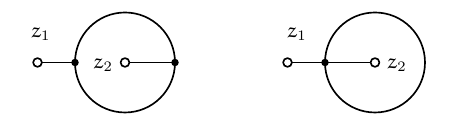}
\caption{The 2 maps belonging to $\mathcal{W}^{[r]}_{0,2}(z_1,z_2)$ in the case $N=0$.}
\label{fig:example:N:0:3}
\end{figure}
\begin{figure}
\centering
\includegraphics[scale=1]{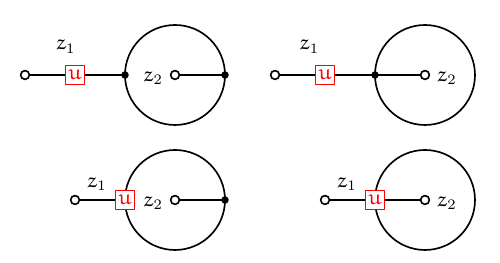}
\caption{The 4 maps belonging to $\mathcal{U}^{[r]}_{0,2}(u;z_1,z_2)$ in the case $N=0$.}
\label{fig:example:N:0:2}
\end{figure}
\end{ex}

\subsection{Properties of the generating functions}\label{sec:properties}
$V'(z)$ is a polynomial of degree $r$, so for a given $z_0\in\mathbb{C}$, the equation $V'(z)=V'(z_0)$ has $r$ solutions denoted $z_0,z_0^{(1)}, \dots,z_0^{(r-1)}$. We label them so that around infinity, 
$$z_0^{(k)} \underset{z_0\to \infty}{=}e^{\frac{2\pi \mathrm{i}\, k}{r}}z \left(1+\mathcal{O}\left(\frac{1}{z}\right)\right).$$
Also, for convenience when carrying out summations over the solutions, we set $z^{(0)}=z$. From Tutte's equations \eqref{eq:tutte:disc:1} and \eqref{eq:tutte:generic:1}, we show the following analytical properties for the generating series:
\begin{thm}\label{thm:pole:lambda}
Let $(g,n)\neq (0,1)$. The generating function $W^{[r]}_{g,n}(z_1,\dots,z_n)$ has no pole at $z_1=\lambda_j$ nor at $z_1=\lambda_j^{(k)}$, $\forall j\in\{1,\dots,N\},\, \forall k\in \{1,\dots,r-1\}$. \\
For all $j\in\{1,\dots,N\}$, $W^{[r]}_{0,1}(z_1)$ has no pole at $z_1=\lambda_j$. However, when $z_1\to \lambda_j^{(k)}$:
\[
W^{[r]}_{0,1}(z_1) \underset{z_1\to \lambda_{j}^{(k)}}{=} \frac{-1}{V'(z_1)-V'(\lambda_j)} + \mathcal{O}(1).
\]
\hfill $\star$
\end{thm}

\begin{thm}\label{thm:pole:z}
Let $(g,n)\neq (0,2)$. The generating function $W^{[r]}_{g,n}(z_1,\dots,z_n)$ has no pole at $z_1=z_m$ nor at $z_1=z_m^{(k)}$, $\forall m\in\{2,\dots,n\},\, \forall k\in \{1,\dots,r-1\}$. \\
$W^{[r]}_{0,2}(z_1,z_2)$ has no pole at $z_1=z_2$. When $z_1\to z_2^{(k)}$:
\[
W^{[r]}_{0,2}(z_1,z_2) \underset{z_1\to z_2^{(k)}}{=} \frac{-1}{(V'(z_1)-V'(z_2))^2} + \mathcal{O}(1).
\]
\hfill $\star$
\end{thm}

\br\label{rem:pole:01}
It is remarkable that the contributions to the pole of $W^{[r]}_{0,1}$ from maps of degree greater than $r+1$ vanish. Actually, only the map 
\[\includegraphics[scale=1]{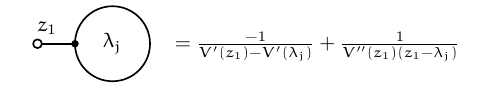} \]
contributes to the pole at $z_1=\lambda_j^{(k)}$. \hfill $\star$
\er

\br
It is also remarkable that all the contributions to the pole of $W^{[r]}_{0,2}$ from maps of degree greater than $r+1$ vanish. Actually, only the 2 maps of Figure \ref{fig:example:N:0:3} contribute to the pole. The sum of their weights is:
$$\frac{-1}{(V'(z_1)-V'(z_2))^2}+\frac{1}{V''(z_1)V''(z_2)(z_1-z_2)^2}.$$ \hfill $\star$
\er

\begin{proof}[Proof of Theorem \ref{thm:pole:lambda}]
First, notice that, since the weight of a map shows no pole at $z_1\to\lambda_j$, $j\in\{1,\dots,N\}$, the same is true for the generating functions. However, when an edge $z_1|\lambda_j$ is present, it creates poles at $z_1\to \lambda_j^{(k)}$, so those poles might be present at the level of the generating functions. We prove the theorem by recursion on $2g+n$. \medskip \\
\textbf{Initialisation: case of the disk.} First, we compute explicitly the smallest degrees of $H^{[r]}_{0,1}$ and $W^{[r]}_{0,1}$. The degree $-(r+1)$ part is:
\begin{equation}\label{eq:degree:-1:H:W}
\begin{split}
H^{[r],-1}_{0,1}(u;z_1) &= \left[\frac{V'(u)}{u-z_1}\right]_{+} = \frac{V'(u)-V'(z_1)}{u-z_1}, \\
W^{[r],-1}_{0,1}(z_1) &= 0.
\end{split}
\end{equation}
The degree $0$ part is:
\begin{equation}\label{eq:degree:zero:H:W}
\begin{split}
H^{[r],0}_{0,1}(u;z_1) &= \sum\limits_{j=1}^{N}\left(\frac{\mathcal{V}_3(u,z_1,\lambda_j)}{V'(z_1)-V'(\lambda_j)}-\frac{\mathcal{V}_3(u,z_1,z_1)}{V''(z_1)(z_1-\lambda_j)}\right), \\
W^{[r],0}_{0,1}(z_1) &= \sum\limits_{j=1}^{N}\left(\frac{-1}{V'(z_1)-V'(\lambda_j)}+\frac{1}{V''(z_1)(z_1-\lambda_j)} \right).
\end{split}
\end{equation}
Let us prove by recursion on the degree that, for any $\delta\geq 0$, $H^{[r],\delta}_{0,1}(u;z_1)$ has at most a pole of order 1 when $z_1\to \lambda_j^{(k)}$ \emph{i.e.}:
$$ (V'(z_1)-V'(\lambda_j)) H^{[r],\delta}_{0,1}(u;z_1) \underset{z_1\to \lambda_j^{(k)}}{=} \mathcal{O}(1), $$
and that for $\delta\geq 1$, $W^{[r],\delta}_{0,1}(z_1)$ is analytic when $z_1\to \lambda_j^{(k)}$, \emph{i.e.}: 
$$ W^{[r],\delta}_{0,1}(u;z_1) \underset{z_1\to \lambda_j^{(k)}}{=} \mathcal{O}(1). $$
We prove those behaviors jointly by induction on $\delta$, using Tutte's equation for the disk (eq. \eqref{eq:tutte:disc:H}). The base case for $H^{[r],\delta}_{0,1}$ is given in equation \eqref{eq:degree:zero:H:W}. Let us now take $\delta\geq 0$ and suppose that for all $\delta' \leq \delta$, the scaling properties on $H^{[r],\delta'}_{0,1}$ and $W^{[r],\delta'}_{0,1}$ are true. By Tutte's equation for the disk (eq. \eqref{eq:tutte:disc:H}) 
\begin{eqnarray*}
(u-z_1) H^{[r],\delta+1}_{0,1}(u;z_1)&=&\sum\limits_{j=1}^{N}\frac{H^{[r],\delta}_{0,1}(u;z_1)-H^{[r],\delta}_{0,1}(u;\lambda_j)}{V'(z_1)-V'(\lambda_j)} +W^{[r],0}_{0,1}(z_1)H^{[r],\delta}_{0,1}(u;z_1)\\
&& + H^{[r],-1}_{0,1}(u;z_1) W^{[r],\delta+1}_{0,1}(z_1) + \sum\limits_{\delta' = 1}^{\delta} W_{0,1}^{[r],\delta'}(z_1) H^{[r],\delta-\delta'}_{0,1}(u;z_1).
\end{eqnarray*}
(Note that if $\delta =0$, the last sum is empty.) Taking into account the formula \eqref{eq:degree:zero:H:W} for $W^{[r],0}_{0,1}(z_1)$ and \eqref{eq:degree:-1:H:W} for $H^{[r],-1}_{0,1}(u;z_1)$, we get:
\begin{equation}\label{eq:tutte:disc:degree}
\begin{split}
(u-z_1) H^{[r],\delta+1}_{0,1}(u;z_1)= &\sum\limits_{j=1}^{N}\left(\frac{H^{[r],\delta}_{0,1}(u;z_1)}{V''(z_1)(z_1-\lambda_j)}-\frac{H^{[r],\delta}_{0,1}(u;\lambda_j)}{V'(z_1)-V'(\lambda_j)}\right)\\
& + \frac{V'(u)-V'(z_1)}{u-z_1} W^{[r],\delta+1}_{0,1}(z_1) + \sum\limits_{\delta' = 1}^{\delta} W_{0,1}^{[r],\delta'}(z_1) H^{[r],\delta-\delta'}_{0,1}(u;z_1).
\end{split}
\end{equation}
Setting $u=z_1$, we obtain:
\begin{equation}\label{eq:tutte:disc:degree:w}
V''(z_1)W^{[r],\delta+1}_{0,1}(z_1)=\sum\limits_{j=1}^{N}\left(\frac{H^{[r],\delta}_{0,1}(z_1;\lambda_j)}{V'(z_1)-V'(\lambda_j)}-\frac{H^{[r],\delta}_{0,1}(z_1;z_1)}{V''(z_1)(z_1-\lambda_j)}\right)-\sum\limits_{\delta' = 1}^{\delta} W_{0,1}^{[r],\delta'}(z_1) H^{[r],\delta-\delta'}_{0,1}(z_1;z_1). 
\end{equation}
By induction hypothesis, the terms in the r.h.s.~have poles of order at most 1 as $z_1\to \lambda_j^{(k)}$, so $W^{[r],\delta+1}_{0,1}(z_1)$ has a pole of order at most 1 at those points. Now, from equation \eqref{eq:uncil:to:cil} of Lemma \ref{lem:uncil:to:cil}:
\begin{equation}\label{eq:uncil:to:cil:degree}
W^{[r],\delta+1}_{0,1}(z_1) = \frac{1}{V''(z_1)} \frac{\partial}{\partial z_1} F^{[r],\delta+1}_{0,1}(z_1). 
\end{equation}
(for $W^{[r],0}_{0,1}$, this relation does not stand.) Since $W^{[r],\delta+1}_{0,1}(z_1)$ is the partial derivative w.r.t. $z_1$ of a rational function in $z_1$, it cannot have only poles of order 1. Thus we get that $W^{[r],\delta+1}_{0,1}(z_1)$ is regular at $z_1\to \lambda_j^{(k)}$:
$$W^{[r],\delta+1}_{0,1}(z_1) \underset{z_1\to\lambda_j^{(k)}}{=}\mathcal{O}(1).$$
Plugging this behavior into the r.h.s.~of equation \eqref{eq:tutte:disc:degree} and using the induction hypothesis, we deduce that $H^{[r],\delta+1}_{0,1}(u;z_1)$ can have poles of order at most 1 as $z_1\to \lambda_j^{(k)}$:
$$(V'(z_1)-V'(\lambda_j))H^{[r],\delta+1}_{0,1}(u;z_1) \underset{z_1\to\lambda_j^{(k)}}{=} \mathcal{O}(1).$$
This ends the recursive step on $\delta$.
\medskip \\
We see that the only case for which $W^{[r],\delta}_{0,1}(z_1)$ has a pole at $z_1\to\lambda_j^{(k)}$ is for $\delta=0$, and it is the case for which equation \eqref{eq:uncil:to:cil:degree} does not hold. In the end, we proved that:
$$W^{[r]}_{0,1}(z_1)\underset{z_1\to\lambda_j^{(k)}}{=} \frac{-1}{V'(z_1)-V'(\lambda_j)}+\mathcal{O}(1).$$
\textbf{Generic case.} Let $g\geq 0,\, n\geq 1$, such that $2g+n \geq 2$. Suppose, by recursion hypothesis, that for all $(h,m)$ such that $2h+m < 2g+n$,
$$ W^{[r]}_{h,m}(z_1,I)\underset{z_1\to\lambda_j^{(k)}}{=} \frac{-\delta_{h,0}\delta_{m,1}}{V'(z_1)-V'(\lambda_j)}+\mathcal{O}(1).$$ 
The proof for the case $(g,n)$ follows the same lines as for the initialisation step, that is to say we prove by induction on $\delta$ that, for all $\delta\geq 2g+n-2$,
\[
\begin{split}
(V'(z_1)-V'(\lambda_j)) H^{[r],\delta}_{g,n}(u;z_1,I) &\underset{z_1\to\lambda_j^{(k)}}{=} \,\,\mathcal{O}(1),\cr
W^{[r],\delta}_{g,n}(u;z_1,I) &\underset{z_1\to\lambda_j^{(k)}}{=}\,\, \mathcal{O}(1).
\end{split}
\]
The initialisation step is done for $\delta=2g-2+n$ (indeed, for $\delta<2g-2+n$, the sets $\mathcal{U}^{[r],\delta}_{g,n}(u;z_1,\dots,z_n)$ and $\mathcal{W}^{[r],\delta}_{g,n}(z_1,\dots,z_n)$ are empty). The maps contributing to the generating functions $H^{[r],2g-2+n}_{g,n}(u;z_1,\dots,z_n)$ and $W^{[r],2g-2+n}_{g,n}(z_1,\dots,z_n)$ have only $n$ faces, whose labels are $z_1,\dots,z_n$. Therefore, there is no map containing an edge $z_1|\lambda_j$, which entails that the property is true for $\delta=2g-2+n$. \medskip \\
For the recursive step, the proof follows exactly the argument used for $(g,n)=(0,1)$, except that now we need to use generic Tutte's equation \eqref{eq:tutte:generic:H} instead of Tutte's equation for the disk, and that we use the induction hypothesis for all the terms on the r.h.s.~of Tutte's equation such that $2h+m <2g+n$ or $\delta'<\delta$. \medskip \\
Hence, by recursion on $2g+n$ and $\delta$, Theorem \ref{thm:pole:lambda} is true.\medskip \\
\textbf{Summary of the ideas.} If we sum-up the argument of the proof, the fact that $W^{[r],\delta}_{g,n}(z_1,I)$ has no pole at $z_1\to \lambda_j^{(k)}$ lies on the fact that:
\begin{itemize}
	\item first we show \emph{via} Tutte's equation that $W^{[r],\delta}_{g,n}(z_1,\dots,z_n)$ has a pole of degree at most 1 at $z_1\to\lambda_j^{(k)}$;
	\item for $(g,n,\delta)\neq (0,1,0)$, $W^{[r],\delta}_{g,n}(z_1,\dots,z_n)$ is expressed as a derivative w.r.t. $z_1$, which prevents $W^{[r],\delta}_{g,n}(z_1,\dots,z_n)$ from having a pole of degree 1 at $z_1\to\lambda_j^{(k)}$.
\end{itemize} 
There is one exception that comes from:
$$W^{[r],0}_{0,1}(z_1) =\sum\limits_{j=1}^N \left(\frac{-1}{V'(z_1)-V'(\lambda_j)} +\frac{1}{V''(z_1)(z_1-\lambda_j)}\right), $$
which entails that $W^{[r],0}_{0,1}(z_1)$ is not a derivative w.r.t. $z_1$. In that case, a pole of degree 1 is allowed. 
\end{proof}

\begin{proof}[Proof of Theorem \ref{thm:pole:z}]
The proof follows exactly the same lines as for the behavior close to $\lambda_j^{(k)}$. This is why we do not give the full details, but rather the ideas behind. 
\begin{itemize}
	\item first we show, \emph{via} Tutte's equation, that $W^{[r],\delta}_{g,n}(z_1,\dots,z_n)$ has a pole of degree at most 2 at $z_1\to z_2^{(k)}$;
	\item for $(g,n,\delta)\neq (0,2,0)$, $W^{[r],\delta}_{g,n}(z_1,\dots,z_n)$ is expressed as a derivative w.r.t. $z_1$ and $z_2$, which prevents $W^{[r],\delta}_{g,n}(z_1,\dots,z_n)$ from having a pole of degree less or equal to 2 at $z_1\to z_2^{(k)}$.
\end{itemize}
There is one exception that comes from:
$$W^{[r],0}_{0,2}(z_1,z_2) = \frac{-1}{\left(V'(z_1)-V'(z_2)\right)^2} +\frac{1}{V''(z_1)V''(z_2)(z_1-z_2)^2}\;,$$
which entails that $W^{[r],0}_{0,2}(z_1,z_2)$ is not a derivative w.r.t. $z_1$ and $z_2$. In that case, a pole of degree 2 is allowed.
\end{proof}

\section{Topological recursion for generalised Kontsevich graphs}\label{sec:top:rec}
\bd
The parameters $v_1,\dots,v_{r+1}$ and $\lambda_1,\dots,\lambda_N$ are said \emph{generic} if they satisfy the two following constraints:
\begin{enumerate}
	\item the solutions of $V''(z)=0$ are simple roots ;
	\item for all $i\in\{1,\dots,N\}$, $V''(\lambda_i)\neq 0$, and for all $i\neq j$, $V'(\lambda_i)\neq V'(\lambda_j)$.
\end{enumerate}
\ed
For now, we consider generic parameters. In the following, we shall see that the first condition is equivalent to considering spectral curves with \emph{simple ramification points}, and we shall relax this constraint later in order to extend to more general profiles at the ramification points.

\subsection{Spectral curve}\label{sec:spectral:curve}

\bd\label{def:x:y}
We set:
\begin{eqnarray}\label{eq:x:y}
x(z)&=& V'(z),\cr
y(z)&=& z + \alpha^{-(r+1)}W^{[r]}_{0,1}(z)+\alpha^{-(r+1)}\sum\limits_{j=1}^{N}\frac{1}{V'(z)-V'(\lambda_j)}.
\end{eqnarray}
\ed

\begin{thm}\label{thm:w01}
There exists a polynomial $Q$ of degree $r$, such that if $\zeta$ is the implicit function defined by 
\[Q(\zeta) = x(z)\qquad {\rm s.t.} \qquad \zeta\underset{z\to \infty}{=}z +\mathcal{O}(1) \, ,\]
then the function $y$ can be written as
\begin{equation}\label{eq:w01}
\boxed{
y(\zeta)= \zeta +\alpha^{-(r+1)}\sum\limits_{j=1}^N \frac{1}{Q'(\xi_j)(\zeta-\xi_j)}\;,
}
\end{equation}
where $\xi_i$ is defined by $Q(\xi_i)=V'(\lambda_i)$. Moreover, the polynomial $Q$ is completely determined by:
\begin{equation}\label{eq:Q}
V'(y(\zeta))-Q(\zeta)\underset{\zeta\to\infty}{=} \mathcal{O}\left(\frac{1}{\zeta}\right), 
\end{equation}
and it is a formal power series in $\alpha^{-(r+1)}$. The deck transformations $\zeta^{(k)}$ are also formal power series in $\alpha^{-(r+1)}$.
\hfill $\star$
\end{thm}

This theorem states that there exists a re-parametrisation of $y$ such that it has an explicit formula, and it gives a way to compute this re-parametrisation. In the following (for example Theorem \ref{thm:w02}), we privilege the new parameters as the variables of the functions: for instance, $\omega^{[r]}_{0,2}$ is viewed as a differential in the variables $\zeta_1,\,\zeta_2$ rather than the variables $z_1,\,z_2$. 

\br
From its definition in equation \eqref{eq:x:y}, $y$ is a formal power series in $\alpha^{-(r+1)}$. Equations \eqref{eq:w01} and \eqref{eq:Q} allow to express $y$ as a rational function of $\zeta$, and $\zeta$ is a solution of an algebraic equation: the series $y$ and $\zeta$ are then algebraic functions of $\alpha^{-(r+1)}$, so they are convergent in some disc. After the proof of Theorem \ref{thm:w01}, we consider that $\alpha$ belongs to the disc of convergence of $y$ and $\zeta$; in particular $\zeta$ is considered as a complex variable. \hfill $\star$
\er

\bd\label{def:P01}
We define the following polynomial in $y$ of degree $r$ whose coefficients are rational functions of $x$:
\begin{equation}\label{eq:def:P:0:1}
P^{[r]}_{0,1}(x,y) = \alpha^{r+1}\left(V'(y)-x\right) - \alpha^{-(r+1)}\sum\limits_{j=1}^{N} \frac{H^{[r]}_{0,1}(y,\lambda_j)}{x-V'(\lambda_j)}.
\end{equation}
\hfill $\star$
\ed

\begin{lem}\label{lem:P01:H01}
The polynomials $P^{[r]}_{0,1}(x,u)$ and $H^{[r]}_{0,1}(u;z)$ are given by:
\begin{equation}\label{eq:P01:H01}
\begin{split}
P^{[r]}_{0,1}(x(z),y) &= \alpha^{r+1} v_{r+1} \prod\limits_{k=0}^{r-1} \left(y- y\left(z^{(k)}\right) \right), \\
H^{[r]}_{0,1}(u;z) &= \alpha^{r+1} v_{r+1} \prod\limits_{k=1}^{r-1} \left(u- y\left(z^{(k)}\right) \right).
\end{split}
\end{equation}
Also:
\begin{equation}\label{eq:deriv:P}
H^{[r]}_{0,1}(y(z);z) = \left.\frac{\partial}{\partial y} P^{[r]}_{0,1} (x(z),y)\right|_{y=y(z)}.
\end{equation}
\hfill $\star$
\end{lem}

\begin{thm}\label{thm:w02}
The bi-differential $W^{[r]}_{0,2} (\zeta_1,\zeta_2) \dd x(\zeta_1) \dd x(\zeta_2)$ satisfies:
\begin{equation}\label{eq:w02}
\boxed{
W^{[r]}_{0,2} (\zeta_1,\zeta_2) \dd x(\zeta_1) \dd x(\zeta_2) = \frac{\dd\zeta_1 \, \dd \zeta_2}{(\zeta_1-\zeta_2)^2}-\frac{\dd Q(\zeta_1) \, \dd Q(\zeta_2)}{\left(Q(\zeta_1)-Q(\zeta_2)\right)^2}\;.
}
\end{equation}
\hfill $\star$
\end{thm}

From those results, we now define the spectral curve of the model, and doing so, we also define the differentials $\omega^{[r]}_{0,1}(\zeta)$ and $\omega^{[r]}_{0,2}(\zeta_1,\zeta_2)$.

\bd\label{def:spectral:curve}
The \emph{spectral curve} for the combinatorics of generalised Kontsevich graphs is given by $\mathcal{S}=(\mathbb{P}^1,x,y,\omega^{[r]}_{0,2})$, where $x\,\colon\, \mathbb{P}^1\to \mathbb{P}^1$, $y\,\colon\, \mathbb{P}^1\to \mathbb{P}^1$, and the bi-differential $\omega^{[r]}_{0,2}$ are defined by:
\begin{equation*}
\boxed{
\begin{cases}
x(\zeta)= Q(\zeta)\,, \\
y(\zeta)= \zeta +\alpha^{-(r+1)}\sum\limits_{i=1}^N \frac{1}{Q'(\xi_i)(\zeta-\xi_i)}\,, \\
\omega^{[r]}_{0,1}(\zeta)=\alpha^{r+1}y(\zeta)\dd x(\zeta)\,, \\
\omega^{[r]}_{0,2}(\zeta_1,\zeta_2) = \frac{\dd\zeta_1 \, \dd\zeta_2}{(\zeta_1-\zeta_2)^2}\,.
\end{cases}
}
\end{equation*}
\ed


\br 
Let us emphasise the fact that the generating functions we are considering satisfy very universal properties. For instance, the cylinder amplitude $W^{[r]}_{0,2} (\zeta_1,\zeta_2) \dd x(\zeta_1) \dd x(\zeta_2)$, from Theorem~\ref{thm:w02}, is the Bergman kernel on the spectral curve minus the Bergman kernel on the base curve. \hfill $\star$
\er

\begin{proof}[Proof of Lemma \ref{lem:P01:H01}]
From its definition (equation \eqref{eq:def:P:0:1}), $P^{[r]}_{0,1}(x(z),u)$ is a polynomial of degree $r$ in $u$, whose top coefficient is $\alpha^{r+1} v_{r+1}$. From Tutte's equation for the disc \eqref{eq:tutte:disc:H}, we have 
$$(u-y(z))H^{[r]}_{0,1}(u;z) = P^{[r]}_{0,1}(x(z),u).$$
Since $P^{[r]}_{0,1}(x(z),u)$ is a rational function of $x(z)$, $P^{[r]}_{0,1}\left(x\left(z^{(k)}\right),u\right)=P^{[r]}_{0,1}(x(z),u)$, which means that, for all $k\in\{0,\dots,r-1\}$:
$$P^{[r]}_{0,1}(x(z),u) = (u-y(z^{(k)}))H^{[r]}_{0,1}(u;z^{(k)}).$$
Therefore, evaluating $u$ at $y\left(z^{(k)}\right)$, we obtain that $P^{[r]}_{0,1}(x(z),u)$ has $r$ roots located at $u=y\left(z^{(k)}\right)$. Hence:
$$P^{[r]}_{0,1}(u;z) = \alpha^{r+1} v_{r+1} \prod\limits_{k=0}^{r}\left(u-y\left(z^{(k)}\right)\right),$$
from which we deduce formulas of equation \eqref{eq:P01:H01}. \medskip \\
For equation \eqref{eq:deriv:P}, we start again from Tutte equation:
$$H^{[r]}_{0,1}(u;z) = \frac{P^{[r]}_{0,1}(x(z),u)}{u-y(z)} = \frac{P^{[r]}_{0,1}(x(z),u)-P^{[r]}_{0,1}(x(z),y(z))}{u-y(z)}. $$
The second equality stands because $P^{[r]}_{0,1}(x(z),y(z))=0$. Letting $u$ tend to $y(z)$, we obtain equation~\eqref{eq:deriv:P}.
\end{proof}

\begin{proof}[Proof of Theorem \ref{thm:w01}]
\textbf{Notations for the proof:} the coefficient of $\alpha^{-d(r+1)}$ in a formal power series $f$ of $\alpha^{-(r+1)}$ is denoted $[f]_d$. For the formal power series $H^{[r]}_{0,1}(u;z)$ and $W^{[r]}_{0,1}(z)$, those are still denoted by $H^{[r],d}_{0,1}(u;z)$ and $W^{[r],d}_{0,1}(z)$ respectively. \medskip \\
Let us define a polynomial $Q$ of degree $r$ by the implicit relation:
\begin{equation}\label{eq:def:Q}
\begin{cases}
Q(z)\underset{z\to\infty}{=}V'\bigg(z+\frac{1}{\alpha^{r+1}}\sum\limits_{j=1}^{N}\frac{1}{Q'(\xi_j)(z-\xi_j)}\bigg), \\
Q(\xi_j)=V'(\lambda_j).
\end{cases}
\end{equation}
We also define the function $f$ as:
\begin{equation}\label{eq:def:f}
f(z) = z+\frac{1}{\alpha^{r+1}}\sum\limits_{j=1}^{N}\frac{1}{Q'(\xi_j)(z-\xi_j)}.
\end{equation}
The proof is made in three steps.
\begin{enumerate}
	\item There is a unique solution $(x(z),y(z))$ to Tutte's equation \eqref{eq:tutte:disc:H} as formal power series in $\alpha^{-(r+1)}$ such that $x(z)=V'(z)$, and the coefficients of $\alpha^{0}$ and $\alpha^{-(r+1)}$ of $y$ are:
	\[	[y(z)]_0 = z\,,\qquad [y(z)]_1 = \sum\limits_{j=1}^{N}\frac{1}{V''(z)(z-\lambda_j)}\, .\]
	\item Setting $[Q(z)]_0=V'(z)$, and $[\xi_{j}]_0=\lambda_j$, there are unique formal power series $Q\in\mathbb{C}[[\alpha^{-(r+1)}]][z]$, $\xi_j\in\mathbb{C}[[\alpha^{-(r+1)}]]$ satisfying equation \eqref{eq:def:Q}. One defines the formal power series $\zeta$ as:
	\begin{equation}\label{eq:def:zeta}
	\begin{cases}
	Q(\zeta)=V'(z), \\
	[\zeta]_0 = z.
	\end{cases}
	\end{equation}
	Gathering $Q$, $f$ and $\zeta$ defined in equations \eqref{eq:def:Q}, \eqref{eq:def:f} and \eqref{eq:def:zeta}, $Q(\zeta)$ and $f(\zeta)$ are formal power series in $\alpha^{-(r+1)}$.
	\item The pair $(Q(\zeta),f(\zeta))$ satisfies Tutte's equation \eqref{eq:tutte:disc:H}, $Q(\zeta)=V'(z)$, and:
	\[ \left[f(\zeta)\right]_0 =z \,,\qquad \left[f(\zeta)\right]_1= \sum\limits_{j=1}^{N}\frac{1}{V''(z)(z-\lambda_j)}.
	\]
	This means that $(x(z),y(z))=(Q(\zeta),f(\eta))$.
\end{enumerate}

\textbf{Step 1.}

There is a unique formal Laurent series $H^{[r]}_{0,1}(u;z)$ in $\alpha^{-(r+1)}$ which satisfies Tutte's equation \eqref{eq:tutte:disc:H}, with the initial conditions:
\[H^{[r],-1}_{0,1}(u;z)  =  \frac{V'(u)-V'(z)}{u-z},\qquad W^{[r],0}_{0,1}(z) = \sum\limits_{j=1}^{N}\left(\frac{-1}{V'(z)-V'(\lambda_j)}+\frac{1}{V''(z)(z-\lambda_j)}\right).
\]
Indeed, from the knowledge of $H^{[r],d'}_{0,1}(u;z)$ and $W^{[r],d''}_{0,1}(z)$ for $d'\leq d$, $d''\leq d+1$, equation \eqref{eq:tutte:disc:degree} derived in the proof of Theorem \ref{thm:pole:lambda} allows to get $H^{[r],d+1}_{0,1}(u;z)$, and then equation \eqref{eq:tutte:disc:degree:w} allows to get $W^{[r],d+2}_{0,1}(z)$. Therefore, one can deduce recursively all the formal Laurent series $H^{[r]}_{0,1}(u;z)$ from Tutte's equation and the initial conditions. \medskip\\

By equation \eqref{eq:P01:H01}, $H^{[r]}_{0,1}(u;z)$ is completely determined by $(x(z),y(z))$:
\[H^{[r]}_{0,1}(u;z)=\frac{P^{[r]}_{0,1}(x(z),u)-P^{[r]}_{0,1}(x(z),y(z))}{u-y(z)} ,\]
where $x(z)=V'(z)$ and $y(z)=z+\frac{1}{\alpha^{r+1}}\bigg(W^{[r]}_{0,1}(z)+\sum\limits_{j=1}^{N}\frac{1}{V'(z)-V'(\lambda_j)}\bigg)$ by their definition, equation \eqref{eq:x:y}. Notice that the initial conditions for $y$:
\[ \left[y(z)\right]_0=z,\qquad \left[y(z)\right]_1=\sum\limits_{j=1}^{N}\frac{1}{V''(z)(z-\lambda_j)}\]
induce the initial conditions for $H^{[r],-1}_{0,1}(u;z)$ and $W^{[r],0}(z)$. Indeed, by equation \eqref{eq:P01:H01}:
\[
\begin{split}
H^{[r],-1}_{0,1}(u;z) &= v_{r+1}\left[\prod\limits_{k=1}^{r-1}(u-y(z^{(k)})\right]_0 = v_{r+1}\prod\limits_{k=1}^{r-1}(u-\left[y(z^{(k)}\right]_0)\\
&=v_{r+1}\prod\limits_{k=1}^{r-1}(u-z^{(k)})=\frac{V'(u)-V'(z)}{u-z}\,;
\end{split}
\]
and from the definition of $y(z)$ \eqref{eq:x:y}:
\[W^{[r],0}_{0,1}(z)= \left[y(z)\right]_1 - \sum\limits_{j=1}^{N}\frac{1}{V'(z)-V'(\lambda_j)} = \sum\limits_{j=1}^{N}\left(\frac{-1}{V'(z)-V'(\lambda_j)}+\frac{1}{V''(z)(z-\lambda_j)}\right). \] 

In the end, there are unique formal power series $x(z),\,y(z)$ in $\alpha^{-(r+1)}$ which satisfy Tutte's equation, and:
\[
x(z)=V'(z),\qquad  \left[y(z)\right]_0=z,\qquad \left[y(z)\right]_1=\sum\limits_{j=1}^{N}\frac{1}{V''(z)(z-\lambda_j)}.
\]

\br 
Combinatorially, the statement of the first step is equivalent to claim that you can construct all the maps by adding edges successively. \hfill $\star$
\er

\textbf{Step 2.}

Equation \eqref{eq:def:Q} allows to compute $[Q(z)]_d$ and $\left[\xi_j\right]_d$ recursively on $d$ from the knowledge of $[Q(z)]_0=V'(z)$, and $[\xi_{j}]_0=\lambda_j$. Indeed, suppose one knows $[Q(z)]_{d'}$ and $[\xi_j]_{d'}$ for all $d'\leq d$. One has then access to $[Q'(\xi_j)]_{d'}$ for all $d'\leq d$. Therefore one can deduce:
\[\left[V'\left(z+\frac{1}{\alpha^{r+1}}\sum\limits_{j=1}^{N}\frac{1}{Q'(\xi_j)(z-\xi_j)}\right) \right]_{d+1}.\]  
Taking the positive powers in $z$ as $z\to\infty$ in the former formula, one can compute $[Q(z)]_{d+1}$. \medskip \\

In turn, the knowledge of $[Q(z)]_{d'}$ and $[\xi_j]_{d''}$ for all $d'\leq d+1$, $d''\leq d$ allows to compute $[\xi_j]_{d+1}$: since $\xi_j$ is subject to $Q(\xi_j)=V'(\lambda_j)$, we need to impose $[Q(\xi_j)]_{d+1}=0$. \medskip \\

Therefore, $Q$ and $\xi_j$ are well-defined formal power series in $\alpha^{-(r+1)}$. In the same manner, $\zeta$ as defined in equation \eqref{eq:def:zeta}, is a well-defined formal power series in $\alpha^{-(r+1)}$. Therefore, $Q(\zeta)$ and $f(\zeta)$ are well defined power series in $\alpha^{-(r+1)}$. \medskip \\

The deck transformations $\zeta^{(k)}$ are defined as 
\[
Q(\zeta^{(k)}) = Q(\zeta),\qquad [\zeta^{(k)}]_0\underset{z\to\infty}{=} e^{\frac{2\mathrm{i}\pi k}{r}} z+\mathcal{O}(1)\,.
\]

\textbf{Step 3.}

We prove that $Q(\zeta),\,f(\zeta)$ satisfy Tutte's equation:
\begin{equation}\label{eq:Tutte:Q:f}
v_{r+1}\prod\limits_{k=0}^{r-1} \left(u-f(\zeta^{(k)})\right) - V'(u)+ Q(\zeta) + v_{r+1}\alpha^{-(r+1)}\sum\limits_{j=1}^{N}\frac{\prod\limits_{k=1}^{r-1} \left(u-f(\xi_j^{(k)})\right)}{Q(\zeta)-Q(\xi_j)} =0\,,
\end{equation}
where we use the expressions given in equation \eqref{eq:P01:H01} as inspiration. Note that, in this expression, $\zeta^{(k)}$ -- resp. $\xi_j^{(k)}$ -- is defined as: $Q(\zeta^{(k)})=Q(\zeta)$ -- resp. $Q(\xi_j^{(k)})=Q(\xi_j)$ -- (remember that $z^{(k)}$ is defined as $V'(z)=V'(z^{(k)})$.

To prove that \eqref{eq:Tutte:Q:f} holds, we show that the l.h.s.~-- which is a formal power series in $\alpha^{-(r+1)}$ whose coefficients are rational fractions in $z$ -- has no pole and vanishes at infinity as a series in $\alpha^{-(r+1)}$ depending on $z$: for each power of $\alpha$, we have a rational function of $z$ without pole and vanishing at infinity, so it must vanish everywhere.
\begin{itemize}
	\item Behavior at $z \to \lambda_{j_0}^{(\ell)}$ (or $z^{(\ell')}\to\lambda_{j_0}$): there is no pole at this point because
	\[
	\begin{split}
	\prod\limits_{k=0}^{r-1} \left(u-f(\zeta^{(k)})\right)&\underset{z^{(\ell')}\to \lambda_{j_0} }{=} -\alpha^{-(r+1)}\frac{1}{Q'(\xi_{j_0})}\prod\limits_{k=1}^{r-1} \left(u-f(\xi_{j_0}^{(k)})\right)+\mathcal{O}(1), \\
	\sum\limits_{j=1}^{N} \frac{\prod\limits_{k=1}^{r-1} \left(u-f(\xi_j^{(k)})\right)}{Q(\zeta)-Q(\xi_j)} &\underset{z^{(\ell')}\to \lambda_{j_0} }{=} \frac{1}{Q'(\xi_{j_0})}\prod\limits_{k=1}^{r-1} \left(u-f(\xi_{j_0}^{(k)})\right)+\mathcal{O}(1).
	\end{split}
	\]
	The simple poles compensate, so there is no pole at $z^{(\ell')}\to\lambda_{j_0}$.
	\item Behavior at the ramification points $V''(z)=0$. Since $Q(\zeta)=V'(z)$, we can write the l.h.s.~of equation \eqref{eq:Tutte:Q:f} in the following way:
	\[v_{r+1}\prod\limits_{k=0}^{r-1} \left(u-f(\zeta^{(k)})\right) - V'(u)+V'(z) + v_{r+1}\alpha^{-(r+1)}\sum\limits_{j=1}^{N}\frac{\prod\limits_{k=1}^{r-1} \left(u-f(\xi_j^{(k)})\right)}{V'(z)-V'(\lambda_j)} =0\,.
 \]
	The only term which can have a pole at $V''(z)=0$ is the first one, since	the formal power series $\zeta$ has poles at  $V''(z)=0$. However, since $f(\zeta^{(k)})$ can be interpreted as the formal power series $f(z)$ whose coefficients are rational functions, evaluated at the formal power series $\zeta^{(k)}$, and since the product runs over $k=0,\dots,r-1$ the first term can be expressed as a function of $Q(\zeta)=V'(z)$. Therefore it has no pole at $V''(z)$.
	\item Behavior at $z \to \infty$. By the definition of $Q$, $f$ and $\zeta$, equations \eqref{eq:def:Q}, \eqref{eq:def:f}, \eqref{eq:def:zeta}, we have:
	\[f(\zeta^{(k)}) \underset{z\to\infty}{=}  \left(f(\zeta)\right)^{(k)}+\mathcal{O}(1/z^r)\,. \]
Therefore, near infinity:
	\[
	\begin{split}
	v_{r+1}\prod\limits_{k=0}^{r-1} \left(u-f(\zeta^{(k)})\right)&\underset{z\to\infty}{=} V'(u)-V'(f(\zeta)) +\mathcal{O}(1/z) \\
	&\underset{z\to\infty}{=} V'(u)-Q(\zeta) +\mathcal{O}(1/z)\,,
	\end{split}
	\]
	so we see that there is no pole at infinity, and that the l.h.s.~of Tutte's equation \eqref{eq:Tutte:Q:f} vanishes at $z\to \infty$.
\end{itemize}
Therefore eq. \eqref{eq:Tutte:Q:f} vanishes everywhere: $Q(\zeta),\, f(\zeta)$ satisfy Tutte's equation. \medskip\\
Last, we check that the initial conditions are satisfied. First, by the definition of $\zeta$, equation \eqref{eq:def:zeta}, we already have $Q(\zeta)=V'(z)=x(z)$. Second, we compute $[f(\zeta)]_0$ and $[f(\zeta)]_1$:
\[\begin{split}
[f(\zeta)]_0 &= [\zeta]_0 =[y(z)]_0\,,\\
[f(\zeta)]_1 &= [\zeta]_1 + \sum\limits_{j=1}^{N}\frac{1}{[Q'(\xi_j)]_0([\zeta]_0-[\xi_j]_0])} =[\zeta]_1 + \sum\limits_{j=1}^{N}\frac{1}{V''(\lambda_j)(z-\lambda_j)}\,.
\end{split}
\]
The coefficient $[\zeta]_1$ is determined by $[Q(\zeta)]_1=0$. In order to compute this constraint, we need to compute $[Q(z)]_1$ first, following the procedure of step 2. We obtain:
\[
[Q(z)]_1 =\sum\limits_{j=1}^{N} \frac{1}{V''(\lambda_j)} \frac{V''(z)-V''(\lambda_j)}{z-\lambda_j}\,.
\]
The constraint $[Q(\zeta)]_1=0$ then yields:
\[
[\zeta]_1 = -\frac{1}{V''(z)}\sum\limits_{j=1}^{N} \frac{1}{V''(\lambda_j)} \frac{V''(z)-V''(\lambda_j)}{z-\lambda_j}\,.
\]
Therefore, we obtain:
\[
[f(\zeta)]_1 = \sum\limits_{j=1}^{N}\frac{1}{V''(z)(z-\lambda_j)}=[y(z)]_1\,.
\]
This ends the proof. 
\end{proof}


\begin{proof}[Proof of Theorem \ref{thm:w02}]
We start from the generic Tutte's equation \eqref{eq:tutte:generic:H}, applied with $(g,n)=(0,2)$:
\[
\begin{split}
(u-y(z_1))H^{[r]}_{0,2}(u;z_1,z_2) =& \alpha^{-(r+1)} \left[ \sum\limits_{j=1}^{N}\frac{H^{[r]}_{0,2}(u;\lambda_j,z_2)}{x(z_1)-V'(\lambda_j)}-\frac{1}{V''(z_2)}\frac{\partial}{\partial z_2} \frac{H^{[r]}_{0,1}(u;z_2)}{x(z_1)-x(z_2)}\right. \\
& \left. +H^{[r]}_{0,1}(u;z_1)\left(W^{[r]}_{0,2}(z_1,z_2) + \frac{1}{(x(z_1)-x(z_2))^2}\right)\right].
\end{split}
\]
Since $H^{[r]}_{0,1}$ and $H^{[r]}_{0,2}$ are polynomials in $u$, we can evaluate them at finite values of $u$. Setting $u=y(z_1)$, we obtain:
\[
H^{[r]}_{0,1}(y(z_1);z_1)\left(W^{[r]}_{0,2}(z_1,z_2) + \frac{1}{(x(z_1)-x(z_2))^2}\right) = \sum\limits_{j=1}^{N}\frac{H^{[r]}_{0,2}(y(z_1);\lambda_j,z_2)}{x(z_1)-V'(\lambda_j)}-\frac{1}{V''(z_2)}\frac{\partial}{\partial z_2} \frac{H^{[r]}_{0,1}(y(z_1);z_2)}{x(z_1)-x(z_2)}.
\]
As a function of $z_1$, the r.h.s.~is a rational function of $x(z_1)$ and $y(z_1)$, and by equation \eqref{eq:deriv:P} of Lemma \ref{lem:P01:H01}, $H^{[r]}_{0,1}(y(z_1);z_1)$ is also a rational function of $x(z_1)$ and $y(z_1)$. Therefore, as a function of $z_1$, $W^{[r]}_{0,2}(z_1,z_2) + \frac{1}{(x(z_1)-x(z_2))^2}$ is a rational function of $x(z_1) = Q(\zeta_1)$ and $y(z_1) = y(\zeta_1)$:
\[
W^{[r]}_{0,2}(\zeta_1,\zeta_2) + \frac{1}{(x(\zeta_1)-x(\zeta_2))^2} =  \frac{\sum\limits_{j=1}^{N}\frac{H^{[r]}_{0,2}(y(z_1);\lambda_j,z_2)}{x(z_1)-V'(\lambda_j)}-\frac{1}{V''(z_2)}\frac{\partial}{\partial z_2} \frac{H^{[r]}_{0,1}(y(z_1);z_2)}{x(z_1)-x(z_2)}}{\left.\frac{\partial}{\partial y} P^{[r]}_{0,1}(x(\zeta_1),y)\right|_{y=y(\zeta_1)}}.
\] 
\begin{itemize}
	\item From Theorems \ref{thm:pole:lambda} and \ref{thm:pole:z}, we know that, except at the ramification points of the spectral curve, the only pole of the l.h.s.~is at $\zeta_1 = \zeta_2$, with the behavior:
$$W^{[r]}_{0,2}(\zeta_1,\zeta_2) + \frac{1}{(x(\zeta_1)-x(\zeta_2))^2} \underset{\zeta_1\to \zeta_2}{=} \frac{1}{Q'(\zeta_1)Q'(\zeta_2)(\zeta_1-\zeta_2)^2} + \mathcal{O}(1).  $$
	\item At the ramification points of the spectral curve -- at $Q'(\zeta_1)=0$ -- the numerator of the r.h.s.~has no pole, while the denominator has simple zeros. This means that the l.h.s.~has simple poles at the ramification points. 
	\item Last, when $\zeta_1$ is close to infinity, the l.h.s.~behaves like $\mathcal{O}(1/\zeta_1^r)$.
\end{itemize}
From those 3 points, we deduce that 
$$\left(W^{[r]}_{0,2}(\zeta_1,\zeta_2) + \frac{1}{(Q(\zeta_1)-Q(\zeta_2))^2} - \frac{1}{Q'(\zeta_1)Q'(\zeta_2)(\zeta_1-\zeta_2)^2} \right) Q'(\zeta_1) \dd \zeta_1. $$
is a differential in $\zeta_1$ without pole, which vanishes at infinity. It is thus identically zero:
$$W^{[r]}_{0,2}(\zeta_1,\zeta_2) = \frac{-1}{(Q(\zeta_1)-Q(\zeta_2))^2} + \frac{1}{Q'(\zeta_1)Q'(\zeta_2)(\zeta_1-\zeta_2)^2}.$$ 
\end{proof}

From Theorems \ref{thm:w01} and \ref{thm:w02}, we deduce the following results:
\begin{lem}\label{lem:linear:loop:disc:cylinder}
\begin{equation}\label{eq:linear:loop:disc}
\sum\limits_{k=0}^{r-1} \omega^{[r]}_{0,1}(\zeta^{(k)}) =  \bigg(-\frac{v_r\alpha^{r+1}}{v_{r+1}}+\sum\limits_{j=1}^{N} \frac{1}{Q(\zeta)-Q(\xi_j)} \bigg) \dd x(\zeta),
\end{equation}
\begin{equation}\label{eq:linear:loop:cylinder}
\sum\limits_{k=0}^{r-1} \omega^{[r]}_{0,2}(\zeta_1^{(k)},\zeta_2) = \frac{Q'(\zeta_1)Q'(\zeta_2)\dd \zeta_1 \dd \zeta_2}{(Q(\zeta_1)-Q(\zeta_2))^2}.
\end{equation}
\hfill $\star$
\end{lem}
\begin{proof}
Equation \eqref{eq:linear:loop:disc} can be rewritten as
\[
\sum\limits_{k=0}^{r-1} y(\zeta^{(k)}) =-\frac{v_r}{v_{r+1}} +\sum\limits_{j=1}^N \frac{\alpha^{-(r+1)}}{Q(\zeta)-Q(\xi_j)}.
\]
It holds because:
\begin{itemize}
	\item we have the identity
	\[
	\sum\limits_{k=0}^{r-1} \frac{1}{Q'(\xi_j)(\zeta^{(k)}-\xi_j)} = \frac{1}{Q(\zeta)-Q(\xi_j)}.
	\]
	\item from the definition of $Q(\zeta)=\sum\limits_{k=0}^{r} Q_k \zeta^k$ by equation \eqref{eq:Q}, $Q_{r}=v_{r+1} $ and $Q_{r-1}=v_r$, so the sum of the roots of $Q(\zeta)=V'(z)$ is $-v_r/v_{r+1}$, that is to say
	\[
	\sum\limits_{k=0}^{r-1} \zeta^{(k)} =-\frac{v_r}{v_{r+1}}.
	\]
\end{itemize}
Equation \eqref{eq:linear:loop:cylinder} follows from the identity
\[
\sum\limits_{k=0}^{r-1}\frac{1}{\zeta_1^{(k)}-\zeta_2}=\frac{Q'(\zeta_2)}{Q(\zeta_1)-Q(\zeta_2)}\;,
\]
which one differentiates w.r.t. $\zeta_1$ and $\zeta_2$ to get the result. 
\end{proof}

\subsection{Finding equivalent expressions for the generating functions}\label{sec:check:P:H}
In this subsection, we introduce the generating functions $P^{[r]}_{g,n}(u;\zeta_1,I) $, $\check{H}^{[r]}_{g,n}(u;\zeta_1;I) $ and $ \check{P}^{[r]}_{g,n}(u;\zeta_1;I)$ in Definitions \ref{def:Pgn} and \ref{def:check:P:H}. The subsection is dedicated to the proof of Lemma \ref{lem:Pgn:Hgn}. This lemma will be useful to prove the main theorem of the combinatorial part of this paper (Theorem \ref{thm:top:rec}), in the next subsection. 

\bd\label{def:Pgn}
Let $(g,n)\neq (0,1)$ and denote by $I$, $I_m$ respectively the sets $\{\zeta_2,\dots,\zeta_n\}$ and $I\backslash \{\zeta_m\}$. Then we define:
\begin{equation}\label{eq:def:Pgn}
P^{[r]}_{g,n}(u;\zeta_1,I)= -\alpha^{-(r+1)} \left[\sum\limits_{j=1}^{N}\frac{H^{[r]}_{g,n}(u;\xi_j,I)}{x(\zeta_1)-Q(\xi_j)}+\sum\limits_{m=2}^{n}\frac{1}{Q'(\zeta_m)}\frac{\partial}{\partial \zeta_m}\left(\frac{H^{[r]}_{g,n-1}(u;\zeta_m,I_m)}{x(\zeta_1)-x(\zeta_m)}\right)\right].
\end{equation}
For $(g,n)=(0,1)$, $P^{[r]}_{0,1}$ was defined in equation \eqref{eq:def:P:0:1} in Definition \ref{def:P01}.
\hfill $\star$
\ed
From its definition, it is clear that $P^{[r]}_{g,n}(u;\zeta_1,I)$ is a polynomial in $u$ of degree $r-1$ for $(g,n)=(0,2)$, and $r-2$ otherwise; it is a rational function of $x(\zeta_1)$. In particular, $P^{[r]}_{g,n}(u;\zeta_1,I)=P^{[r]}_{g,n}(u;\zeta_1^{(k)},I)$. \medskip \\

In order to define the generating functions $\check{H}^{[r]}_{g,n}(u;\zeta_1;I) $ and $ \check{P}^{[r]}_{g,n}(u;\zeta_1;I)$, we need to introduce some notations, following Section 3 of \cite{BouEy17}.
\bd\label{def:tilde:E}
Let $g\geq 0 ,\,n\geq 1$, $k\geq 1$. Let $\underline{t}=\{t_1,\dots,t_k\}$ and $I=\{\zeta_2,\dots,\zeta_n\}$. We define:
\[
\begin{split}
\widetilde{W}^{[r]}_{0,1}(\zeta) &= \alpha^{r+1}\, y(\zeta)\,, \\
\widetilde{W}^{[r]}_{0,2}(\zeta_1,\zeta_2) &= W^{[r]}_{0,2} (\zeta_1,\zeta_2) + \frac{1}{(x(\zeta_1)-x(\zeta_2))^2}\,, \\
\widetilde{W}^{[r]}_{g,n}(\zeta_1,I) &= W^{[r]}_{g,n}(\zeta_1,I), \qquad \text{for } 2g-2+n>0\,,
\end{split}
\]
and
\[
\mathcal{E}^{(k)}W^{[r]}_{g,n}(\underline{t};I) =\sum\limits_{\mu \in \mathcal{S}(\underline{t})} \sum\limits_{J_1\sqcup \dots \sqcup J_{\ell(\mu)} = I} \sum\limits_{\substack{h_1+\dots+h_{\ell(\mu)} = \\ g+\ell(\mu)-k}}\left(\prod\limits_{i=1}^{\ell(\mu)}\widetilde{W}^{[r]}_{h_i,\#\mu_i+\#J_i}(\mu_i,J_i)\right),
\]
where $\mathcal{S}(\underline{t})$ is the set of partition of $\underline{t}$; $\ell(\mu)$ is the number of subsets in the partition $\mu$. We set
$$\mathcal{E}^{(0)} W^{[r]}_{g,n}(I) =\delta_{g,0}\delta_{n,1}. $$
\hfill $\star$
\ed
\begin{ex}
\[
\begin{split}
\mathcal{E}^{(1)}W^{[r]}_{g,n}(\{\zeta_1\};I)&= \widetilde{W}^{[r]}_{g,n}(\zeta_1,I). \\
\mathcal{E}^{(2)}W^{[r]}_{g,n}(\{\zeta_1,\zeta'_1\};I) &= \widetilde{W}^{[r]}_{g-1,n+1}(\zeta_1,\zeta'_1,I)+\sum\limits_{\substack{h+h'=g\\ J\sqcup J' =I}} \widetilde{W}^{[r]}_{h,1+\#J}(\zeta_1,J) \widetilde{W}^{[r]}_{h',1+\#J'}(\zeta'_1,J')\\
&= W^{[r]}_{g-1,n+1}(\zeta_1,\zeta'_1,I)+\sum\limits_{\substack{h+h'=g\\ J\sqcup J' =I}}^{'} W^{[r]}_{h,1+\#J}(\zeta_1,J) W^{[r]}_{h',1+\#J'}(\zeta'_1,J')\\
&+ \sum\limits_{m=2}^{n}\left(\frac{W^{[r]}_{g,n-1}(\zeta_1,I_m)}{(x(\zeta_1)-x(\zeta_m))^2}+\frac{W^{[r]}_{g,n-1}(\zeta'_1,I_m)}{(x(\zeta'_1)-x(\zeta_m))^2}\right) + y(\zeta_1) W^{[r]}_{g,n}(\zeta'_{1},I)+y(\zeta'_1) W^{[r]}_{g,n}(\zeta_{1},I) \\
& +\frac{\delta_{g,1}\delta_{n,1}}{(x(\zeta_1)-x(\zeta'_1))^2}+\frac{\delta_{g,0}\delta_{n,2}y(\zeta_1)}{(x(\zeta'_1)-x(\zeta_2))^2}+\frac{\delta_{g,0}\delta_{n,2}y(\zeta'_1)}{(x(\zeta_1)-x(\zeta_2))^2}\,.
\end{split}
\]
\end{ex}

\textbf{Some notations.} For $k\in\{0,\dots,r-1\}$, we denote by $\tau(\zeta)$ and $\tau_k(\zeta)$ respectively the sets $\{\zeta,\zeta^{(1)},\dots,\zeta^{(r-1)}\}$ and $\tau(\zeta)\backslash \{\zeta^{(k)}\}$. Also, given two sets $U,\,V$, we say that $U\subset_k V$ if $U$ is a subset of $V$ of cardinal $k$.  
\bd\label{def:check:P:H}
We define the generating functions $\check{H}^{[r]}_{g,n}(u;\zeta_1;I) $ and $ \check{P}^{[r]}_{g,n}(u;\zeta_1;I)$:
\begin{equation}\label{eq:def:check:P:H}
\begin{split}
\check{H}^{[r]}_{g,n}(u;\zeta_1,I) &= v_{r+1}\sum\limits_{k=0}^{r-1} (-1)^k u^{r-1-k}\alpha^{-(k-1)(r+1)} \sum\limits_{\underline{t} \subset_k \tau_0(\zeta_1)} \mathcal{E}^{(k)} W^{[r]}_{g,n}(\underline{t};I),\\
\check{P}^{[r]}_{g,n}(u;\zeta_1,I) &= v_{r+1}\sum\limits_{k=0}^{r} (-1)^k u^{r-k}\alpha^{-(k-1)(r+1)} \sum\limits_{\underline{t} \subset_k \tau(\zeta_1)} \mathcal{E}^{(k)} W^{[r]}_{g,n}(\underline{t};I).
\end{split}
\end{equation}
\hfill $\star$
\ed
Both are polynomials in $u$; $\check{H}^{[r]}_{0,1}(u;\zeta_1,I)$ and $\check{P}^{[r]}_{0,1}(u;\zeta_1,I)$ are of respective degrees $r-1 $ and $r$; for $(g,n)\neq(0,1)$, $\check{P}^{[r]}_{g,n}(u;\zeta_1,I)$ and $\check{P}^{[r]}_{g,n}(u;\zeta_1,I)$ are of respective degrees $r-2$ and $r-1$ \emph{a priori} (we will see later that the top coefficient of $\check{P}^{[r]}_{g,n}(u;\zeta_1,I)$ vanishes for $(g,n)\neq (0,1), (0,2)$, so it is actually of degree $r-2$). We see that $\check{P}^{[r]}_{g,n}(u;\zeta_1,I)$ is invariant under the change $\zeta_1 \to \zeta_1^{(k)}$, so it is actually a function of $x(\zeta_1)$.
\begin{ex}
From the definition:
\[
\begin{split}
\check{H}^{[r]}_{0,1}(u;\zeta_1) & = \alpha^{r+1}v_{r+1}\left(u^{r-1} + \sum\limits_{k=1}^{r-1}(-1)^k u^{r-k} \sum\limits_{1\leq i_1<\dots<i_k\leq r-1} \prod\limits_{\ell=1}^{k}y\left(\zeta_1^{(i_\ell)}\right) \right)\\
&= \alpha^{r+1}v_{r+1} \prod\limits_{k=1}^{r-1}\left(u-y\left(\zeta_1^{(k)}\right)\right).
\end{split}
\]
In the same manner:
$$\check{P}^{[r]}_{0,1}(u;\zeta_1) = \alpha^{r+1}v_{r+1} \prod\limits_{k=0}^{r-1}\left(u-y\left(\zeta_1^{(k)}\right)\right).$$
\hfill $\star$
\end{ex}
From this example and Lemma \ref{lem:P01:H01}, we remark that 
$$H^{[r]}_{0,1}(u;\zeta_1)=\check{H}^{[r]}_{0,1}(u;\zeta_1)\qquad \mathrm{and}\qquad P^{[r]}_{0,1}(x(\zeta_1),u)=\check{P}^{[r]}_{0,1}(u;\zeta_1). $$
Actually, this is a general feature, as we will see in Lemma \ref{lem:Pgn:Hgn}. \medskip \\
Let us derive two relations; the first one between the $H^{[r]}_{g,n}$'s, the $W^{[r]}_{g,n}$'s and the $P^{[r]}_{g,n}$'s; the second one between the $\check{H}^{[r]}_{g,n}$'s, the $W^{[r]}_{g,n}$'s and the $\check{P}^{[r]}_{g,n}$'s. From Tutte's equation \eqref{eq:tutte:generic:H}, we get for $(g,n)\neq (0,1)$:
\begin{equation}\label{eq:rel:P:H}
\begin{split}
P^{[r]}_{g,n}(u;\zeta_1,I) =& -\alpha^{-(r+1)}\Bigg[H^{[r]}_{0,1}(u;\zeta_1)\widetilde{W}^{[r]}_{g,n}(\zeta_1,I)+H^{[r]}_{g-1,n+1}(u;\zeta_1,\zeta_1,I)\\
&+\sum\limits_{\substack{h+h'=g \\J\sqcup J'=I}}^{'}\widetilde{W}^{[r]}_{h,1+\#J}(\zeta_1,J) H^{[r]}_{h',1+\#J'}(u;\zeta_1,J')\Bigg] \\
& +(u-y(\zeta_1)) H^{[r]}_{g,n}(u;\zeta_1,I),
\end{split}
\end{equation}
where we exclude the cases $(h,J)=(0,\emptyset)$ and $(h',J') = (0,\emptyset)$ in the primed sum. \medskip \\
For the other generating functions, we start from equation \eqref{eq:def:check:P:H} with $(g,n)\neq (0,1)$:
\[
\begin{split}
\check{P}^{[r]}_{g,n}(u;\zeta_1,I) = &v_{r+1}\sum\limits_{k=1}^{r}(-1)^k u^{r-k} \alpha^{-(k-1)(r+1)}\sum\limits_{\underline{t} \subset_k \tau(\zeta_1)} \mathcal{E}^{(k)} W^{[r]}_{g,n}(\underline{t};I) \\
=& v_{r+1}\sum\limits_{k=0}^{r-1}(-1)^{k+1} u^{r-1-k} \alpha^{-k(r+1)}\Bigg( \sum\limits_{\underline{t} \subset_{k+1} \tau_0(\zeta_1)} \mathcal{E}^{(k+1)} W^{[r]}_{g,n}(\underline{t};I) \\
&+ \sum\limits_{\underline{t} \subset_{k} \tau_0(\zeta_1)} \mathcal{E}^{(k+1)} W^{[r]}_{g,n}(\{\zeta_1\}\cup\underline{t};I) \Bigg)\\
=& u\,\check{H}^{[r]}_{g,n}(u;\zeta_1,I) + v_{r+1}\sum\limits_{k=0}^{r-1}(-1)^{k+1} u^{r-1-k}\alpha^{-k(r+1)} \sum\limits_{\underline{t} \subset_{k} \tau_0(\zeta_1)} \mathcal{E}^{(k+1)} W^{[r]}_{g,n}(\{\zeta_1\}\cup\underline{t};I).
\end{split}
\]
In the first line, the summation over $k$ begins at $k=1$ because $\mathcal{E}^{(0)}W^{[r]}_{g,n}(I)=0$ for $(g,n)\neq (0,1)$. To get the second line, we decompose the sum over $\underline{t} \subset_{k+1}\tau(\zeta_1)$ into a sum over $\underline{t} \subset_{k+1}\tau_0(\zeta_1)$ and a sum over $\{\zeta_1\}\cup\underline{t} \subset_{k+1}\tau(\zeta_1)$. In order to simplify the last line, we use Lemma 3.17 of \cite{BouEy17}:
\[
\mathcal{E}^{(k+1)}W^{[r]}_{g,n}(\{\zeta_1\}\cup\underline{t};I) = \mathcal{E}^{(k)}W^{[r]}_{g-1,n+1}(\underline{t};\{\zeta_1\}\cup I) +\sum\limits_{\substack{h+h'= g \\ J\sqcup J' =I }}\widetilde{W}^{[r]}_{h,1+\#J}(\zeta_1,J) \mathcal{E}^{(k)}W^{[r]}_{h',1+\# J'}(\underline{t};J').
\]
Plugging this result in our computations, we obtain:
\[
\check{P}^{[r]}_{g,n}(u;\zeta_1,I) = u\,\check{H}^{[r]}_{g,n}(u;\zeta_1,I) - \alpha^{-(r+1)}\sum\limits_{\substack{h+h'= g \\ J\sqcup J' =I }}\widetilde{W}^{[r]}_{h,1+\#J}(\zeta_1,J) \check{H}^{[r]}_{h',1+\# J'}(u;\zeta_1;J'),
\]
\emph{i.e.}
\begin{equation}\label{eq:rel:check:P:H}
\begin{split}
\check{P}^{[r]}_{g,n}(u;\zeta_1,I) =& -\alpha^{-(r+1)}\Bigg[\check{H}^{[r]}_{0,1}(u;\zeta_1)\widetilde{W}^{[r]}_{g,n}(\zeta_1,I)+\check{H}^{[r]}_{g-1,n+1}(u;\zeta_1,\zeta_1,I)\\
&+\sum\limits_{\substack{h+h'=g \\J\sqcup J'=I}}^{'}\widetilde{W}^{[r]}_{h,1+\#J}(\zeta_1,J) \check{H}^{[r]}_{h',1+\#J'}(u;\zeta_1,J')\Bigg] \\
& +(u-y(\zeta_1)) \check{H}^{[r]}_{g,n}(u;\zeta_1,I),
\end{split}
\end{equation}
where we exclude the cases $(h,J)=(0,\emptyset)$ and $(h',J') = (0,\emptyset)$ in the primed sum. \medskip \\
Comparing equations \eqref{eq:rel:P:H} and \eqref{eq:rel:check:P:H}, we immediately see that $\check{P}^{[r]}_{g,n}(u;\zeta_1,I)$ and $\check{H}^{[r]}_{g,n}(u;\zeta_1,I)$ satisfy the same equation as $P^{[r]}_{g,n}(u;\zeta_1,I)$ and $H^{[r]}_{g,n}(u;\zeta_1,I)$. \medskip \\
\pagebreak

We can now turn to the main lemma of this subsection:
\begin{lem}\label{lem:Pgn:Hgn}
For all $g,n$, we have:
$$H^{[r]}_{g,n}(u;\zeta_1,I)=\check{H}^{[r]}_{g,n}(u;\zeta_1,I)\qquad \mathrm{and}\qquad P^{[r]}_{g,n}(u;\zeta_1,I)=\check{P}^{[r]}_{g,n}(u;\zeta_1,I). $$
\hfill $\star$
\end{lem}

\begin{proof}
We prove the statement by induction on $2g+n$. The initial case -- $(g,n)=(0,1)$ -- is the object of Lemma \ref{lem:P01:H01}. For the induction step, let $g\geq 0,\, n\geq 1$ such that $2g+n \geq 2 $ and suppose that for all $g'\geq 0,\, n'\geq 1$ such that $2g'+n' <2g+n$, the lemma is true. Starting from equation \eqref{eq:rel:check:P:H} and by the induction hypothesis:
\[
\begin{split}
\check{P}^{[r]}_{g,n}(u;\zeta_1,I) =& -\alpha^{-(r+1)}\Bigg[H^{[r]}_{0,1}(u;\zeta_1)\widetilde{W}^{[r]}_{g,n}(\zeta_1,I)+H^{[r]}_{g-1,n+1}(u;\zeta_1,\zeta_1,I)\\
&+\sum\limits_{\substack{h+h'=g \\J\sqcup J'=I}}^{'}\widetilde{W}^{[r]}_{h,1+\#J}(\zeta_1,J) H^{[r]}_{h',1+\#J'}(u;\zeta_1,J')\Bigg] \\
& +(u-y(\zeta_1)) \check{H}^{[r]}_{g,n}(u;\zeta_1,I).
\end{split}
\]
Using equation \eqref{eq:rel:P:H} in order to replace the terms of the first line on the r.h.s., we obtain 
\[
\check{P}^{[r]}_{g,n}(u;\zeta_1,I)-P^{[r]}_{g,n}(u;\zeta_1,I)=(u-y(\zeta_1))\left(\check{H}^{[r]}_{g,n}(u;\zeta_1,I)-H^{[r]}_{g,n}(u;\zeta_1,I)\right).
\]
We have seen that actually, $P^{[r]}_{g,n}(u;\zeta_1,I)$ and $\check{P}^{[r]}_{g,n}(u;\zeta_1,I)$ are polynomials in $u$ of degree $r-1$ and are functions of $x(\zeta_1)$. Therefore, for all $k\in\{0,\dots,r-1\}$:
\[
\check{P}^{[r]}_{g,n}(u;\zeta_1,I)-P^{[r]}_{g,n}(u;\zeta_1,I)=\left(u-y\left(\zeta_1^{(k)}\right)\right)\left(\check{H}^{[r]}_{g,n}\left(u;\zeta_1^{(k)},I\right)-H^{[r]}_{g,n}\left(u;\zeta_1^{(k)},I\right)\right).
\]
Specialising $u$ at $y\left(\zeta_1^{(k)}\right)$ for $k\in\{0,\dots,r-1\}$, we get:
$$\check{P}^{[r]}_{g,n}\left(y\left(\zeta_1^{(k)}\right);\zeta_1,I\right)=P^{[r]}_{g,n}\left(y\left(\zeta_1^{(k)}\right);\zeta_1,I\right). $$
This means that for all $\zeta_1$, $P^{[r]}_{g,n}(u;\zeta_1,I)$ and $\check{P}^{[r]}_{g,n}(u;\zeta_1,I)$ are polynomials in $u$ of degree $r-1$ that coincide on $r$ distinct values of $u$ ($r-1$ distinct values if $\zeta_1$ is a ramification point), so they are equal:
$$\check{P}^{[r]}_{g,n}(u;\zeta_1,I)=P^{[r]}_{g,n}(u;\zeta_1,I). $$
It entails that we also have
$$\check{H}^{[r]}_{g,n}(u;\zeta_1,I)=H^{[r]}_{g,n}(u;\zeta_1,I). $$
Therefore, the lemma is true by induction.
\end{proof}

\subsection{Topological recursion}\label{subsec:top:rec}
The spectral curve of our model has, for generic parameters, $r-1$ ramification points $b_1,\dots,b_{r-1}$: $Q'(b_i)=0$. We label those ramification points in such a way that $b_k$ is the point at which the $k^{\textup{th}}$ sheet meets the physical sheet, \emph{i.e.} $b_k=b_k^{(k)}$. To each ramification point $b_k$ corresponds a recursion kernel: 
\bd[\emph{Recursion kernel}]\label{def:recursion:kernel}
The recursion kernel corresponding to the ramification point $b_k$ is given by:
$$K_{b_k} (\zeta_1,\zeta) = \frac{\int_{\zeta'=\zeta^{(k)}}^{\zeta} \omega^{[r]}_{0,2}(\zeta_1,\zeta')}{2\left(\omega^{[r]}_{0,1}(\zeta)-\omega^{[r]}_{0,1}\left(\zeta^{(k)}\right)\right)}. $$
\hfill $\star$
\ed

\bd[\emph{The differentials}]\label{def:differentials}
For $g\geq 0$, $n\geq 1$, we define the $n$-differential:
\begin{equation}\label{eq:def:differentials}
\omega^{[r]}_{g,n}(\zeta_1,\dots,\zeta_n) = \widetilde{W}^{[r]}_{g,n}(\zeta_1,\dots,\zeta_n) \prod\limits_{i=1}^{n}\dd x(\zeta_i).
\end{equation}
\hfill $\star$
\ed

\begin{thm}\label{thm:top:rec}
For generic parameters $v_1,\dots,v_{r+1}$, $\lambda_1,\dots,\lambda_N$, the differentials $\omega^{[r]}_{g,n}$'s satisfy Topological recursion with the spectral curve of Definition \ref{def:spectral:curve}:
\begin{equation}\label{eq:top:rec}
\omega^{[r]}_{g,n}(\zeta_1,I) = \sum\limits_{k=1}^{r-1}\underset{\zeta=b_k}{\mathrm{Res}} K_{b_k}(\zeta_1,\zeta)\Bigg(\omega^{[r]}_{g-1,n+1}\left(\zeta,\zeta^{(k)},I\right)+\sum\limits_{\substack{h+h'=g\\ J\sqcup J' =I}}^{'}\omega^{[r]}_{h,1+\#J}(\zeta,J)\omega^{[r]}_{h',1+\#J'}\left(\zeta^{(k)},J'\right)\Bigg),
\end{equation}
where we exclude the $(0,1)$ topologies in the primed sum. 
\hfill $\star$
\end{thm}
\br
The theorem is proved only for generic parameters here. We refer the reader to the next section of the paper for the non-generic case and more general branching profile of the spectral curve, which requires the use of Bouchard--Eynard Theorem \cite{BoEy13}. \hfill $\star$
\er
\br
The $n$-differentials $\omega^{[r]}_{g,n}(\zeta_1,\dots,\zeta_n)$'s are meromorphic $n$-forms in terms of the variables $\zeta_i\in\mathbb{P}^1$, $i\in\{1,\dots,n\}$. The generating series $W^{[r]}_{g,n}(z_1,\dots,z_n)$ can be recovered from $\omega^{[r]}_{g,n}(\zeta_1,\dots,\zeta_n)$ by expansion around $z_i\to \infty$, $i\in\{1,\dots,n\}$. \hfill $\star$
\er

The remaining of this section is dedicated to the proof of Theorem \ref{thm:top:rec}.

\subsubsection{Linear and quadratic loop equations}
\begin{lem}[\emph{Linear loop equation}]\label{lem:linear:loop}
Let $I=\{\zeta_2,\dots,\zeta_n\}$.
\begin{equation}\label{eq:linear:loop}
\sum\limits_{k=0}^{r-1}\omega^{[r]}_{g,n}\left(\zeta_1^{(k)},\zeta_2,\dots,\zeta_n\right) = \delta_{g,0}\delta_{n,1}\Bigg(-\frac{v_r\alpha^{r+1}}{v_{r+1}}+\sum\limits_{j=1}^{N}\frac{1}{x(\zeta_1)-x(\xi_j)}\Bigg)\dd x(\zeta_1) +\delta_{g,0}\delta_{n,2}\frac{\dd x(\zeta_1)\,\dd x(\zeta_2)}{\left(x(\zeta_1)-x(\zeta_2)\right)^2}.
\end{equation}
\hfill $\star$
\end{lem}

\begin{lem}[\emph{Quadratic loop equation}]\label{lem:quadratic:loop}
Let $I=\{\zeta_2,\dots,\zeta_n\}$.
\[
\sum\limits_{k=0}^{r-1}\Bigg[\omega^{[r]}_{g-1,n+1}\left(\zeta_1^{(k)},\zeta_{1}^{(k)},I\right)+\sum\limits_{\substack{h+h'=g\\J\sqcup J'=I}} \omega^{[r]}_{h,1+\# J}\left(\zeta_{1}^{(k)},J\right) \omega^{[r]}_{h',1+\# J'}\left(\zeta_{1}^{(k)},J'\right) \Bigg]
\]
is a differential in $x(\zeta_1)$ without pole at the ramification points.
\hfill $\star$
\end{lem}
We prove those statements from the equality between $H^{[r]}_{g,n}(u;\zeta_1,I)$ and $\check{H}^{[r]}_{g,n}(u;\zeta_1,I)$ of Lemma \ref{lem:Pgn:Hgn}. The linear loop equation is obtained by the equality between the coefficients of $u^{r-2}$ of $H^{[r]}_{g,n}(u;\zeta_1,I)$ and $\check{H}^{[r]}_{g,n}(u;\zeta_1,I)$; while the quadratic loop equation is deduced from the equality between the coefficients of $u^{r-3}$.
\begin{proof}[Proof of Lemma \ref{lem:linear:loop}]
The $(g,n)=(0,1)$ and $(g,n)=(0,2)$ cases have been addressed in Lemma \ref{lem:linear:loop:disc:cylinder}. Suppose that $(g,n)$ is different from those cases. From equation \eqref{eq:top:coef:H} of Lemma \ref{lem:poly:part}, the coefficient of $u^{r-2}$ in $H^{[r]}_{g,n}(u;\zeta_1,I) $ is 
$$v_{r+1}\, W^{[r]}_{g,n}(\zeta_1,I). $$
From equation \eqref{eq:def:check:P:H} of Definition \ref{def:check:P:H}, the coefficient of $u^{r-2}$ in $\check{H}^{[r]}_{g,n}(u;\zeta_1,I)$ is 
$$-v_{r+1}\sum\limits_{k=1}^{r-1} \mathcal{E}^{(1)}W^{[r]}_{g,n}\left(\zeta_1^{(k)};I\right)= -v_{r+1} \sum\limits_{k=1}^{r-1} W^{[r]}_{g,n}\left(\zeta^{(k)}_1,I\right)$$
(here $\widetilde{W}^{[r]}_{g,n}=W^{[r]}_{g,n}$ because we exclude the disc and cylinder topologies).
Since $H^{[r]}_{g,n}(u;\zeta_1,I)=\check{H}^{[r]}_{g,n}(u;\zeta_1,I)$ by Lemma \ref{lem:Pgn:Hgn}, we get:
$$ \sum\limits_{k=0}^{r-1}W^{[r]}_{g,n}\left(\zeta_1^{(k)},I\right)=0,$$
which is equivalent to the linear loop equation. 
\end{proof}
\br
Actually, this lemma implies that the coefficient of $u^{r-1}$ in $\check{P}^{[r]}_{g,n}(u;\zeta_1,I)$ vanishes when $(g,n)\neq (0,1),\,(0,2)$, so the latter is actually a polynomial of degree $r-2$ in $u$.   \hfill $\star$
\er
\begin{proof}[Proof of Lemma \ref{lem:quadratic:loop}]
Let $I=\{\zeta_2,\dots,\zeta_n\}$, we prove the quadratic loop equation for $2g+n-2 \geq 2$ (the proof for the cases $(g,n)=(1,1)$ and $(g,n)=(0,3)$ follows the same line with more special cases). To do so, we compare the coefficients of $u^{r-3}$ in $H^{[r]}_{g,n}(u;\zeta_1,I)$ and $\check{H}^{[r]}_{g,n}(u;\zeta_1,I)$. \medskip \\
\textbf{Coefficient in $H^{[r]}_{g,n}(u;\zeta_1,I)$.} Using Tutte's equation \eqref{eq:tutte:generic:H} and the $u^{r-2}$ coefficient of $H^{[r]}_{g,n}(u;\zeta_1,I)$ (equation \eqref{eq:top:coef:H} of Lemma \ref{lem:poly:part}), the coefficient of $u^{r-3}$ in $H^{[r]}_{g,n}(u;\zeta_1,I)$ is:
\[
\begin{split}
\left[H^{[r]}_{g,n}(u;\zeta_1,I)\right]_{r-3}=& (v_{r}+2\, z_1 v_{r+1}) W^{[r]}_{g,n}(z_1,I) + v_{r+1}\alpha^{-(r+1)}\Bigg[\sum\limits_{j=1}^N \frac{W^{[r]}_{g,n}(z_1,I)-W^{[r]}_{g,n}(\lambda_j,I)}{V'(z_1)-V'(\lambda_j)}  \\
&+\sum\limits_{m=2}^{n}\frac{1}{V''(z_m)}\frac{\partial}{\partial z_m} \frac{W^{[r]}_{g,n-1}(z_1,I_m)- W^{[r]}_{g,n-1}(z_m,I_m) }{V'(z_1)-V'(z_m)} + W^{[r]}_{g-1,n+1}(z_1,z_1,I) \\
& + \sum\limits_{\substack{h+h'=g\\J\sqcup J'=I}}W^{[r]}_{h,1+\# J}(z_1,J) W^{[r]}_{h',1+\# J'}(z_1,J')\Bigg].
\end{split}
\]
That we can rewrite as (in the cases $(g,n)=(1,1),\, (0,3)$, other terms would appear):
\begin{equation}\label{eq:coef:r-3:H}
\begin{split}
\left[H^{[r]}_{g,n}(u;\zeta_1,I)\right]_{r-3}=&v_{r+1}\alpha^{-(r+1)} \Bigg[\widetilde{W}^{[r]}_{g-1,n+1}(\zeta_1,\zeta_1,I) + \sum\limits_{\substack{h+h'=g \\ J\sqcup J'=I}} \widetilde{W}^{[r]}_{h,1+\#J}(\zeta_1,J)\widetilde{W}^{[r]}_{h',1+\#J'}(\zeta_1,J')  \\
& +\alpha^{r+1}\Bigg(\frac{v_r}{v_{r+1}} -\sum\limits_{j=1}^N\frac{\alpha^{-(r+1)}}{x(\zeta_1)-x(\xi_j)}\Bigg)\widetilde{W}^{[r]}_{g,n}(\zeta_1,I) -\sum\limits_{m=2}^{n}\frac{\widetilde{W}^{[r]}_{g,n-1}(\zeta_1,I_m)}{\left(x(\zeta_1)-x(\zeta_m)\right)^2} \\
&  -\sum\limits_{j=1}^{N}\frac{\widetilde{W}^{[r]}_{g,n}(\xi_j,I)}{x(\zeta_1)-x(\xi_j)} - \sum\limits_{m=2}^{n}\frac{1}{Q'(\zeta_m)}\frac{\partial}{\partial \zeta_m} \frac{\widetilde{W}^{[r]}_{g,n-1}(\zeta_m,I_m)}{x(\zeta_1)-x(\zeta_m)}\Bigg].
\end{split}
\end{equation}
\medskip \\
\textbf{Coefficient in $\check{H}^{[r]}_{g,n}(u;\zeta_1,I)$.} From the definition of $\check{H}^{[r]}_{g,n}(u;\zeta_1,I)$ in equation \eqref{eq:def:check:P:H}, we get:
\[
\begin{split}
\left[\check{H}^{[r]}_{g,n}(u;\zeta_1,I)\right]_{r-3}=&\\
 \sum\limits_{1\leq k< \ell\leq r-1}v_{r+1}\alpha^{-(r+1)}\Bigg[\widetilde{W}^{[r]}_{g-1,n+1}\left(\zeta_1^{(k)},\zeta_1^{(\ell)},I\right)+\sum\limits_{\substack{h+h'=g \\ J\sqcup J' =I}}\widetilde{W}^{[r]}_{h,1+\#J}\left(\zeta_1^{(k)},J\right)\widetilde{W}^{[r]}_{h',1+\#J'}\left(\zeta_1^{(\ell)},J'\right)\Bigg].&
\end{split}
\]
We can write the r.h.s.~as a sum over $k,\, \ell$: $\sum\limits_{1 \leq k < \ell \leq r-1}f\left(\zeta_1^{(k)},\zeta_1^{(\ell)},I\right)$ where $f$ is symmetric: $f\left(\zeta_1^{(k)},\zeta_1^{(\ell)},I\right)= f\left(\zeta_1^{(\ell)},\zeta_1^{(k)},I\right)$. We can thus rewrite this sum as:
$$\sum\limits_{1 \leq k < \ell \leq r-1}f\left(\zeta_1^{(k)},\zeta_1^{(\ell)},I\right)=\frac{1}{2}\sum\limits_{1 \leq k\neq \ell \leq r-1}f\left(\zeta_1^{(k)},\zeta_1^{(\ell)},I\right) = \frac{1}{2}\left[\sum\limits_{k,\, \ell =1}^{r-1}f\left(\zeta_1^{(k)},\zeta_1^{(\ell)},I\right)-\sum\limits_{k=1}^{r-1}f\left(\zeta_1^{(k)},\zeta_1^{(k)},I\right)\right].$$
Then, applying the linear loop equation \eqref{eq:linear:loop}, we can simplify the sum over $k$ and $\ell$, and in the end, the coefficient of $u^{r-3}$ in $\check{H}^{[r]}_{g,n}(u;\zeta_1,I)$ is (in the cases $(g,n)=(1,1),\, (0,3)$, other terms would appear):
\begin{equation}\label{eq:coef:r-3:check:H}
\begin{split}
\left[\check{H}^{[r]}_{g,n}(u;\zeta_1,I)\right]_{r-3}=& \frac{v_{r+1}\alpha^{-(r+1)}}{2} \Bigg[\widetilde{W}^{[r]}_{g-1,n+1}(\zeta_1,\zeta_1,I)+\sum\limits_{\substack{h+h'=g \\ J\sqcup J' =I}} \widetilde{W}^{[r]}_{h,1+\#J}(\zeta_1,J)\widetilde{W}^{[r]}_{h',1+\#J'}(\zeta_1,J')  \\
& -\sum\limits_{k=1}^{r-1}\Bigg(\widetilde{W}^{[r]}_{g-1,n+1}\left(\zeta_1^{(k)},\zeta_1^{(k)},I\right)+\sum\limits_{\substack{h+h'=g \\ J\sqcup J' =I}} \widetilde{W}^{[r]}_{h,1+\#J}\left(\zeta_1^{(k)},J\right) \widetilde{W}^{[r]}_{h',1+\#J'}\left(\zeta_1^{(k)},J'\right) \Bigg) \\
&+2\alpha^{r+1}\Bigg(\frac{v_r}{v_{r+1}}-\sum\limits_{j=1}^N \frac{\alpha^{-(r+1)}}{x(\zeta_1)-x(\xi_j)}\Bigg) \widetilde{W}^{[r]}_{g,n}(\zeta_1,I) -2\sum\limits_{m=2}^{n}\frac{\widetilde{W}^{[r]}_{g,n-1}(\zeta_1,I_m) }{\left(x(\zeta_1)-x(\zeta_m)\right)^2}\Bigg]. \\
\end{split}
\end{equation}
By Lemma \ref{lem:Pgn:Hgn}, the r.h.s.~of equations \eqref{eq:coef:r-3:H} and \eqref{eq:coef:r-3:check:H} are equal, so we get:
\[
\begin{split}
\sum\limits_{k=0}^{r-1} \Bigg(\widetilde{W}^{[r]}_{g-1,n+1}\left(\zeta_1^{(k)},\zeta_1^{(k)},I\right)+\sum\limits_{\substack{h+h'=g \\ J\sqcup J' =I}} \widetilde{W}^{[r]}_{h,1+\#J}\left(\zeta_1^{(k)},J\right) \widetilde{W}^{[r]}_{h',1+\#J'}\left(\zeta_1^{(k)},J'\right) \Bigg) &= \\
2 \sum\limits_{j=1}^{N}\frac{\widetilde{W}^{[r]}_{g,n}(\xi_j,I)}{x(\zeta_1)-x(\xi_j)} +2 \sum\limits_{m=2}^{n}\frac{1}{Q'(\zeta_m)}\frac{\partial}{\partial \zeta_m} \frac{\widetilde{W}^{[r]}_{g,n-1}(\zeta_m,I_m)}{x(\zeta_1)-x(\zeta_m)}. &
\end{split}
\]
The r.h.s.~is a rational function of $x(\zeta_1)$ without pole at the ramification points, so we proved the quadratic loop equation.
\end{proof}

\subsubsection{From the loop equations to topological recursion}
From the linear and quadratic loop equations (Lemmas \ref{lem:linear:loop} and \ref{lem:quadratic:loop}) and from the pole structures of the generating functions (Theorems \ref{thm:pole:lambda} and \ref{thm:pole:z}), we are able to deduce Theorem \ref{thm:top:rec}. We reproduce the steps given in \cite{EO09} for completeness (see Sections 5.1.5 and 5.2.5 there for instance). 

\begin{proof}[Proof of Theorem \ref{thm:top:rec}]
Let 
\begin{equation}\label{eq:one:differential}
\dd S_{\zeta_1,\zeta_2}(\zeta) =\int_{\zeta'=\zeta_2}^{\zeta_1} \omega^{[r]}_{0,2}(\zeta,\zeta')
\end{equation}
be the differential in $\zeta\in\mathbb{P}^1$ whose only poles are simple at $\zeta=\zeta_1$ and $\zeta=\zeta_2$, with residues $+1$ and $-1$ respectively. Let us choose $o\in\mathbb{P}^1$ as an arbitrary base point; $(g,n)\neq (0,1),\, (0,2)$ and set $I=\{\zeta_2,\dots,\zeta_n\}$ we then have Cauchy formula:
\begin{equation}\label{eq:proof:tr}
\omega^{[r]}_{g,n}(\zeta,I)=\underset{\zeta'\to\zeta}{\Res}\dd S_{\zeta',o}(\zeta) \omega^{[r]}_{g,n}(\zeta',I).
\end{equation}
From Theorems \ref{thm:pole:lambda} and \ref{thm:pole:z}, we know that the differential form $\omega^{[r]}_{g,n}(\zeta,\zeta_2,\dots,\zeta_n)$ has poles only at the ramification points $\zeta\to b_k$, $k=1,\dots,r-1$. Therefore, from Riemann bilinear identity:
\[
\underset{\zeta'\to \zeta}{\Res}\dd S_{\zeta',o}(\zeta)\omega^{[r]}_{g,n}(\zeta',I) +\sum\limits_{k=1}^{r-1} \underset{\zeta'\to b_k}{\Res} \dd S_{\zeta',o}(\zeta)\omega^{[r]}_{g,n}(\zeta',I) =0
\]
(for a non-rational spectral curve, the r.h.s.~is a sum over a symplectic basis of cycles). We can therefore move the integration contour of the residue in equation \eqref{eq:proof:tr} around the ramification points:
\begin{equation}\label{eq:proof:tr:2}
\omega^{[r]}_{g,n}(\zeta,I)=-\sum\limits_{k=1}^{r-1} \underset{\zeta'\to b_k}{\Res} \dd S_{\zeta',o}(\zeta)\omega^{[r]}_{g,n}(\zeta',I).
\end{equation}
Let us focus on a given ramification point $b_k$. At this, ramification point, $\omega^{[r]}_{g,n}(\zeta,I)$ and $\omega^{[r]}_{g,n}(\zeta^{(k)},I)$ have a pole at $\zeta=b_k$, and for $i\neq 0,\,k$, $\omega^{[r]}_{g,n}(\zeta^{(i)},I)$ has no pole at $\zeta=b_k$. Moreover, from the linear loop equation (Lemma \ref{lem:linear:loop}), we know that 
\begin{equation}\label{eq:proof:tr:3}
\omega^{[r]}_{g,n}(\zeta^{(k)},I)= -\omega^{[r]}_{g,n}(\zeta,I)+\,\mathrm{regular},
\end{equation}
where, in previous equation and in the following one, ``+ regular" means that the other terms are regular at $\zeta=b_k$. \\
Using the quadratic loop equation (Lemma \ref{lem:quadratic:loop}) and equation \eqref{eq:proof:tr:3} near $b_k$, we have:
\begin{eqnarray*}
2(y(\zeta)-y(\zeta^{(k)}))\dd x(\zeta) \, \omega^{[r]}_{g,n}(\zeta,I) &=& \omega^{[r]}_{g-1,n+1}(\zeta,\zeta^{(k)},I)+\sum\limits_{\substack{h+h' =g \\ J\sqcup J'= I}}^{'}\omega^{[r]}_{h,1+\# J}(\zeta,J)\omega^{[r]}_{h',1+\#J'}(\zeta^{(k)},J')\\
&& +\, \mathrm{regular}.
\end{eqnarray*}
We insert this kind of term for each ramification point into equation \eqref{eq:proof:tr:2} and we obtain the formula of topological recursion noticing that:
$$K_{b_k}(\zeta_1,\zeta) = \frac{\dd S_{\zeta,\zeta^{(k)}}(\zeta_1)}{2(y(\zeta)-y(\zeta^{(k)}))\dd x(\zeta)}\,. 
$$
\end{proof}

\subsection{Spectral curves with non-simple ramification points: higher Topological Recursion}\label{sec:higher:tr}
So far, the arguments used for the proof of Theorem~\ref{thm:top:rec} are valid for spectral curves with simple ramification points: this condition corresponds to the first point of the definition of \emph{generic} parameters denoted $v_1,\dots, v_{r+1}$, introduced at the very beginning of this section.

\medskip

Some cases of interest, that are treated in the last part of this paper, require relaxing this condition and being able to consider a spectral curve with one ramification point of order $r-1$ (equivalently, the base curve has a branchpoint of multiplicity $r$). More precisely, consider the potential $V'(z)=z^r$; its only ramification point is at $z=0$. To this potential corresponds a combinatorial model of generalised Kontsevich graphs and generating functions $W^{[r],0}_{g,n}$ perfectly valid (we did not require the parameters to be generic when we exposed the model). We may ask two questions: 
\begin{enumerate}
	\item Do the $\omega^{[r],0}_{g,n}$'s satisfy a recursive relation similar to topological recursion?
	\item If it is the case, consider the one-parameter family of potentials: $V'_{\epsilon}(z)=z^r-r\epsilon^{r-1}z$, having $r-1$ simple ramification points for $\epsilon\neq 0$, approaching the potential $V'(z)=z^r$. For $\epsilon\neq 0$, Theorem~\ref{thm:top:rec} holds, and we can compute $\omega^{[r],\epsilon}_{g,n}$ and $W^{[r],\epsilon}_{g,n}$ via topological recursion. Do we have $\underset{\epsilon\to 0}{\lim}\omega^{[r],\epsilon}_{g,n}(\zeta_1,I)=\omega^{[r],0}_{g,n}(\zeta_1,I)$?
\end{enumerate}
The answer to both questions is positive: the $\omega^{[r],0}_{g,n}$'s satisfy \emph{higher topological recursion}, a generalisation of topological recursion introduced in \cite{BHLMR14} and whose treatment was systematised in \cite{BoEy13}; moreover they coincide with the limits $\underset{\epsilon\to 0}{\lim}\omega^{[r],\epsilon}_{g,n}$. Actually, the paper by Bouchard and one of the authors \cite{BoEy13} addresses explicitly those two questions, and both answers are shown in a same movement: let us unfold what has just been stated, following the definitions and results of \cite{BoEy13}. 

\medskip

\bd\label{def:higher:recursion}
Let $k\geq 1$, $\underline{t}=\{t_1,\dots,t_k\}$, $I=\{\zeta_1,\dots,\zeta_n\}$ and $o\in \mathbb{P}^1$. The \emph{generalised recursion kernel} $K_k(\zeta_0;\underline{t})$ is:
\[K_k(\zeta_0;\underline{t}) =  -\frac{\int_{\zeta=o}^{t_1}\omega^{[r]}_{0,2}(\zeta_0,\zeta) }{\prod\limits_{j=1}^{k}\left(y(t_1)-y(t_j)\right)\dd x(t_1)}.\]
We also define:
\[
\mathcal{R}^{(k)}W^{[r]}_{g,n}(\underline{t};I) = \sum\limits_{\mu\in\mathcal{S}(\underline{t})}\sum\limits_{\substack{J_1\sqcup\dots\sqcup J_{\ell(\mu)}=I\\ h_1+\dots + h_{\ell(\mu)} = \\ g+\ell(\mu)-k}}^{'} \Big(\prod\limits_{i=1}^{\ell(\mu)}\widetilde{W}^{[r]}_{h_i,\#\mu_i+\#J_i}(\mu_i,J_i)\Big),
\]
where $\mathcal{S}(\underline{t})$ is the set of partitions of $\underline{t}$; $\ell(\mu)$ is the length of the partition $\mu$, and the primed sum means that we exclude the cases where $(h_i,\#\mu_i+\#J_i)=(0,1)$ for some $i$.
\hfill $\star$
\ed
$\mathcal{R}^{(k)}W^{[r]}_{g,n}$ differs from $\mathcal{E}^{(k)}W^{[r]}_{g,n}$ (Definition \ref{def:tilde:E}) only by the primed sum.\medskip \\
For $\epsilon \neq 0$, the differentials $\omega^{[r],\epsilon}_{g,n}(\zeta_1,I)$ satisfy topological recursion, so by Theorem 5 of \cite{BoEy13}, they satisfy the so-called \emph{global topological recursion}:
\begin{equation}\label{eq:global:tr}
W^{[r],\epsilon}_{g,n}(\zeta_1,I)=\sum\limits_{k=1}^{r-1}\frac{1}{2\pi \mathrm{i}} \oint_{\Gamma_k(t)} \Bigg(\sum\limits_{\substack{\underline{t}\subset \{t^{(1)},\dots,t^{(r-1)}\} \\ \underline{t}\neq \emptyset}} K_{\#\underline{t}+1}(\zeta_1;t,\underline{t}) \mathcal{R}^{(\#\underline{t}+1)}W^{[r],\epsilon}_{g,n}(t,\underline{t};I) \Bigg),
\end{equation}
where $\Gamma_k(t)$ is the contour around the $k^{\textup{th}}$ ramification point. \\
With this result under our belt, we can deduce that in the limit $\epsilon\to 0$, the correlation functions satisfy the same global recursion, which, expressed locally, is:

\begin{thm}\label{thm:r:Airy}
The correlation functions $W^{[r],0}_{g,n}(\zeta_1,I)$ satisfy the local recursion:
\[W^{[r],0}_{g,n}(\zeta_1,I)= \frac{1}{2\pi\mathrm{i}}\oint_{\Gamma(0)}\Bigg(\sum\limits_{\substack{\underline{t}\subset \{ t\,e^{\frac{2\mathrm{i}\pi}{r-1}},\dots,t\,e^{\frac{2(r-2)\mathrm{i}\pi}{r-1}},t\} \\ \underline{t}\neq\emptyset}} K_{\#\underline{t}+1}(\zeta_1;t,\underline{t})\mathcal{R}^{(\#\underline{t}+1)}W^{[r],0}_{g,n}(t,\underline{t};I)\Bigg),\]
where the contour integral in $t$ runs over the contour $\Gamma(0)$ around 0.  
\hfill $\star$
\end{thm}
We had initially taken this result for granted, invoking the argument exposed in Section 3.5 of \cite{BoEy13} dedicated to such limits, until Nitin Chidambaram pointed out a hole in this argument to us: although the result holds in this case, the argument is not correct. We briefly present the problem with the original argument of \cite{BoEy13} here. In our case, the problem can be overcome by applying \cite[Proposition~5.2]{CCGG22}. We are deeply grateful to Nitin Chidambaram for raising his concerns about this limit. Moreover, in a recently released work \cite{BBCKS+} of G.~Borot, V.~Bouchard, N.~Chidambaram, R.~Kramer and S.~Shadrin, the authors study limits in topological recursion for a wider class of spectral curves, broadly generalising the argument presented in \cite[Proposition~5.2]{CCGG22}.

\medskip

\textbf{Problem in the argument of Section 3.5 of \cite{BoEy13}.} It uses the formulation of topological recursion via its global expression, that is equation \eqref{eq:global:tr}, whose asset is that the integrand is the same for all the contour integrals around the ramification points. In the original argument, for small enough $\epsilon$, one fixes a contour $\Gamma$ surrounding all the ramification points, and then one takes the limit of the expression \eqref{eq:global:tr} with this fixed contour. However, doing so for our family of spectral curves some other poles of the integrand enter the contour as $\epsilon\to 0$, namely the solutions of $\zeta^{r}-r\epsilon^{r-1}\zeta =  -(r-1)\epsilon^r e^{\frac{2\iu \pi k}{r-1}}$ such that $\zeta\neq \epsilon e^{\frac{2\iu \pi k}{r-1}}$. Those poles produce residues that should not be taken into account in order to prove Theorem \ref{thm:r:Airy}.

\subsection{Examples}
We treat various examples of spectral curves along with their topological recursion formula.

\subsubsection{The cases $r=2$, $r=3$}
\paragraph{The case $r=2$.}
Let us choose $V(z)=\frac{z^3}{3}$. We apply Theorem~\ref{thm:w01} to derive the spectral curve. First, the function $y$ has the following shape:
\[
y(\zeta)= \zeta + \alpha^{-3}\sum\limits_{j=1}^{N} \frac{1}{Q'(\xi_j)(\zeta-\xi_j)}
\] 
that we use to deduce the shape of the polynomial $Q$:
\beq\label{eq:example:1}
Q(\zeta)= \left[V'(y(\zeta))\right]_+ =\left[\left(\zeta+ \alpha^{-3}\sum\limits_{j=1}^{N} \frac{1}{Q'(\xi_j)(\zeta-\xi_j)}\right)^2\right]_+,
\eeq
where $Q(\xi_j)=V'(\lambda_j)=\lambda_j^2$ for all $j\in\{1,\dots,N\}$. We use the notation: 
\[
t_0 = \alpha^{-3} \sum\limits_{j=1}^{N} \frac{1}{Q'(\xi_j)},
\]
so that we can rewrite equation \eqref{eq:example:1} to get the shape of $Q$ and $Q'$:
\[
Q(\zeta) = \zeta^2 + 2\, t_0\;;\qquad Q'(\zeta)=2\zeta\,.
\]
We now need to determine $t_0$. First, from its definition and from the shape of $Q'(\zeta)$:
\[
t_0 = \alpha^{-3}\sum_{j=1}^{N} \frac{1}{2\xi_j}\,.
\]
Moreover, we have the equations $Q(\xi_j)=\lambda_j^2$, so $\xi_j = \sqrt{\lambda_j^2 - 2\,t_0}$. Inserting those relations into the expression above, we get:
\[
t_0 =\alpha^{-3}\sum\limits_{j=1}^{N} \frac{1}{2\sqrt{\lambda_j^2-2\, t_0}}\,.
\]
This equation allows to determine $t_0$, and then $Q$ and $\xi_j$. Supposing that we have a solution for $t_0$, the spectral curve has the following shape:
\[
\begin{cases}
x(\zeta) = \zeta^2+2\, t_0\,,\\
y(\zeta) = \zeta +\alpha^{-3} \sum\limits_{j=1}^{N} \frac{1}{2\sqrt{\lambda_j^2-2\, t_0}(\zeta-\sqrt{\lambda_j^2-2\, t_0})}\,.
\end{cases}
\]
The spectral curve has one ramification point $b_1$, solution of $Q'(\zeta)=0$, \emph{i.e.} $b_1=0$. We have $\zeta^{(1)}=-\zeta$, and the recursion kernel is:
\[
K_{b_1}(\zeta_1,\zeta) = \frac{\alpha^{-3}\,\dd \zeta_1}{4 \zeta (\zeta_1^2-\zeta^2)\Big(1+\alpha^{-3}\sum\limits_{j=1}^{N} \frac{1}{2\sqrt{\lambda_j^2-2t_0}(\zeta^2-\lambda_j^2+2t_0)}\Big) \dd \zeta}\,.
\]
Applying topological recursion, we get:
\[
\begin{split}
\omega^{[2]}_{0,3}(\zeta_1,\zeta_2,\zeta_3) &= -\frac{\alpha^{-3}\,\dd\zeta_1\, \dd\zeta_2\,\dd \zeta_3}{2\zeta_1^2\zeta_2^2\zeta_3^2 \Big(1-\alpha^{-3}\sum\limits_{j=1}^N \frac{1}{2\left(\lambda_j^2-2t_0\right)^{\frac{3}{2}}}\Big)}\,,\\
\omega^{[2]}_{1,1}(\zeta_1) &= -\frac{\alpha^{-3}\,\dd \zeta_1}{16\zeta_1^4\Big(1-\alpha^{-3}\sum\limits_{j=1}^N \frac{1}{2\left(\lambda_j^2-2t_0\right)^{\frac{3}{2}}}\Big)} \left[1+ \zeta_1^2 \frac{\alpha^{-3}\sum\limits_{j=1}^N \frac{1}{2\left(\lambda_j^2-2t_0\right)^{\frac{5}{2}}}}{1-\alpha^{-3}\sum\limits_{j=1}^N \frac{1}{2\left(\lambda_j^2-2t_0\right)^{\frac{3}{2}}}} \right].
\end{split}
\]
\paragraph{The case $r=3$.}
Now let us take $V(z)=\frac{z^4}{4}$. Notice that for this case, the parameters $v_1,\dots,v_4$ are not generic since $V''(z)=3z^2$ has a double zero at $z=0$. In the same manner as in the previous example, we derive the shapes of $y$ and $Q$ to get the spectral curve, by using Theorem \ref{thm:w01}. 
First, we know that:
\[
y(\zeta)=\zeta + \alpha^{-4}\sum\limits_{j=1}^{N} \frac{1}{Q'(\xi_j)(\zeta-\xi_j)}\,,
\]
from which we can deduce the shape of $Q$:
\[
Q(\zeta)= \left[V'(y(\zeta))\right]_+ = \left[ \Big(\zeta + \alpha^{-4}\sum\limits_{j=1}^{N} \frac{1}{Q'(\xi_j)(\zeta-\xi_j)}\Big)^3 \right]_+ = \zeta^3 + 3t_0\, \zeta + 3t_1\,,
\]
where we used again some notations. Namely:
\[
t_0= \alpha^{-4}\sum\limits_{j=1}^{N}\frac{1}{Q'(\xi_j)} \;; \qquad t_1= \alpha^{-4}\sum\limits_{j=1}^{N}\frac{\xi_j}{Q'(\xi_j)}\,.
\]
Once we know the shape of $Q$, it remains to determine $t_0$ and $t_1$ from their definition and from the equations $Q(\xi_j)=V'(\lambda_j)=\lambda_j^3$. More precisely, in order to completely know the spectral curve, one has to solve the system of equations:
\[
\begin{split}
\lambda_j^3 &= \xi_j^3+3t_0\, \xi_j + 3 t_1\,,  \\
t_0 &= \alpha^{-4}\sum\limits_{j=1}^{N}\frac{1}{Q'(\xi_j)}\,, \\
t_1 &= \alpha^{-4}\sum\limits_{j=1}^{N}\frac{\xi_j}{Q'(\xi_j)}\,.
\end{split}
\]
We see that if no assumption is done on the parameters $N$ and $\lambda_j$, the determination of the spectral curve can be involved even for small $r$. \medskip \\
The ramification points are the solutions of $Q'(\zeta)=0$, that is to say:
\[
\begin{split}
b_1=\sqrt{-t_0}\;;\qquad \zeta^{(1)}=-\frac{\zeta}{2}+\frac{\sqrt{3(-\zeta^2-4t_0)}}{2}\,;\\
b_2=-\sqrt{-t_0} \;; \qquad \zeta^{(2)}=-\frac{\zeta}{2} -\frac{\sqrt{3(-\zeta^2-4t_0)}}{2}\,.
\end{split}
\]
The recursion kernel is 
\[
K_{b_j}(\zeta_1;\zeta) = \frac{\dd \zeta_1}{6(\zeta_1-\zeta)(\zeta_1-\zeta^{(j)})(\zeta^2+t_0)\dd \zeta}\frac{\alpha^{-4}}{1-\alpha^{-4}\sum\limits_{j=i}^{N}\frac{1}{Q'(\xi_i)(\zeta-\xi_i)(\zeta^{(j)}-\xi_i)}}.
\]
We apply topological recursion and get:
\[
\begin{split}
\omega^{[3]}_{0,3}(\zeta_1,\zeta_2,\zeta_3) = & -\frac{\alpha^{-4}\,\dd \zeta_1\,\dd\zeta_2\,\dd\zeta_3}{3(b_1-b_2)} \Bigg[\frac{1}{(\zeta_1-b_1)^2(\zeta_2-b_1)^2(\zeta_3-b_1)^2\Big(1-\alpha^{-4}\sum\limits_{j=1}^{N}\frac{1}{Q'(\xi_j)(b_1-\xi_j)^2}\Big)} \\
& -\frac{1}{(\zeta_1-b_2)^2(\zeta_2-b_2)^2(\zeta_3-b_2)^2\Big(1-\alpha^{-4}\sum\limits_{j=1}^{N}\frac{1}{Q'(\xi_j)(b_2-\xi_j)^2}\Big)}  \Bigg].
\end{split}
\]


\subsubsection{The case $N=0$}
Let us consider the case where $N=0$, that is to say we consider Generalised Kontsevich Graphs without non-marked faces. According to Theorem \ref{thm:w01}, we have:
\[ y(\zeta) = \zeta \]
and $Q(\zeta) = \left[V'(y(\zeta))\right]_+ =  V'(\zeta)$, so the spectral curve in that case is simply:
\[
\begin{cases}
x(\zeta) = V'(\zeta)\,, \\
y(\zeta) = \zeta\,.
\end{cases}
\]
We also have $\omega^{[r]}_{0,2}(\zeta_1,\zeta_2) = \frac{\dd \zeta_1 \dd \zeta_2}{(\zeta_1-\zeta_2)^2}$. The ramification points $b_1,\dots,b_{r-1}$ are the solutions of $V''(\zeta)=0$, \emph{i.e.} $V''(\zeta)=r\,v_{r+1}\prod\limits_{k=1}^{r-1}(\zeta-b_k)$. The recursion kernel near the $k^{\mathrm{th}}$ ramification point is 
\[
K_{b_k}(\zeta_1,\zeta) = \frac{\alpha^{-(r+1)}\,\dd \zeta_1}{2 r\, v_{r+1} (\zeta_1-\zeta)(\zeta_1-\zeta^{(k)})\prod\limits_{j=1}^{r-1}(\zeta-b_j) \, \dd \zeta}\,.
\]
Therefore, the topological recursion formula has the following shape:
\[
\begin{split}
\omega^{[r]}_{g,n}(\zeta_1,I) =&\sum\limits_{k=1}^{r-1} \underset{\zeta= b_k}{\Res} \frac{\alpha^{-(r+1)}\,\dd \zeta_1}{2 r\, v_{r+1} (\zeta_1-\zeta)(\zeta_1-\zeta^{(k)})\prod\limits_{j=1}^{r-1}(\zeta-b_j) \, \dd \zeta} \Bigg[\omega^{[r]}_{g-1,n+1}(\zeta,\zeta^{(k)},I) \\
&  + \sum\limits_{\substack{h+h' =g \\ J\sqcup J'= I}}^{'} \omega^{[r]}_{h,1+\#J}(\zeta,J)\omega^{[r]}_{h',1+\#J'}(\zeta^{(k)},J') \Bigg].
\end{split}
\]

If we specialise even more and consider $r=3$, with the following potential for instance:
\[V(z) = \frac{z^4}{4}-\frac{3}{2}z^2. \]
Then, we get 
$$x(\zeta)=\zeta^3-3\zeta \;;\qquad V''(\zeta)=3(\zeta-1)(\zeta+1)\,.$$ 
Therefore $b_1=-1$, $b_2=+1$, and
\[
\zeta^{(1)}= -\frac{\zeta}{2}-\frac{\sqrt{3(4-\zeta^2)}}{2} \; ; \qquad \zeta^{(2)}= -\frac{\zeta}{2}+\frac{\sqrt{3(4-\zeta^2)}}{2}\,.
\]
We get:
\[
K_{b_j}(\zeta_1,\zeta) = \frac{\alpha^{-4}\,\dd \zeta_1}{6(\zeta_1-\zeta)(\zeta_1-\zeta^{(j)})(\zeta+1)(\zeta-1)\dd \zeta}\,.
\]
And if we apply the topological recursion formula to get $\omega^{[3]}_{0,3}(\zeta_1,\zeta_2,\zeta_3)$ and $\omega^{[3]}_{1,1}(\zeta_1)$ we obtain:
\[
\begin{split}
\omega^{[3]}_{0,3}(\zeta_1,\zeta_2,\zeta_3) &= \alpha^{-4} \frac{\dd \zeta_1 \dd \zeta_2 \dd \zeta_3}{6}\left(\frac{1}{(\zeta_1+1)^2(\zeta_2+1)^2(\zeta_3+1)^2}-\frac{1}{(\zeta_1-1)^2(\zeta_2-1)^2(\zeta_3-1)^2}\right)\,, \\
\omega^{[3]}_{1,1}(\zeta_1)&=-\alpha^{-4}\frac{\zeta_1(\zeta_1^2+2)}{9(\zeta_1+1)^4(\zeta_1-1)^4} \dd \zeta_1 \,.
\end{split}
\]

\subsubsection{The case $\lambda_j = \infty$}\label{sec:ex:lambda:infty}
This case is very similar to the previous one. Indeed, let us keep $N$ finite and let $\lambda_j\to\infty$ for all $j\in\{1,\dots,N\}$; using Theorem \ref{thm:w01} in the same way as before, we obtain the same spectral curve:
\[
\begin{split}
x(\zeta) &= Q(\zeta)\,,\\
y(\zeta)&=\zeta\,, \\
\omega^{[r]}_{0,2}(\zeta_1,\zeta_2)&=\frac{\dd\zeta_1\,\dd \zeta_2}{(\zeta_1-\zeta_2)^2}\,.
\end{split}
\]

This means that, applying the recursion formula, we get the same correlation functions $\omega^{[r]}_{g,n}(\zeta_1,\dots,\zeta_n)$ as in the case $N=0$. It is consistent at the combinatorial level: any Generalised Kontsevich Graph with at least one internal face has a vanishing weight since $\lambda_j=\infty$. Therefore, the only graphs that contribute to the generating functions are those which contain only marked faces. Hence we recover the same combinatorial model as in the previous example. The recursion kernels and topological recursion formula are the same as in the previous example. \medskip \\

Let us treat the example of the one-parameter family of potentials $V'_{\epsilon}(z)=z^r-r\,\epsilon^{r-1}z$, already discussed in Section \ref{sec:higher:tr}. For $\epsilon\neq 0$ the spectral curve has $r-1$ simple ramification points $b_1,\dots,b_{r-1}$, with $b_k = \epsilon\, e^{2\pi \mathrm{i}\frac{k}{r-1}}$. We also get:
\[K_{b_k, \epsilon}(\zeta_1,\zeta) =\frac{\alpha^{-(r+1)}\,\dd \zeta_1}{2r(\zeta_1-\zeta)(\zeta_1-\zeta^{(k)}) (\zeta^{r-1}-\epsilon^{r-1})\dd \zeta}\,.\]
Applying the topological recursion formula, we can compute $\omega^{[r],\epsilon}_{g,n}(\zeta_1,\dots,\zeta_n)$ recursively on $2g+n$. It is useful to introduce the basis of rational functions having poles only at the ramification points:
$$\phi_{\ell,m}(\zeta) = \frac{\zeta^{r-2-\ell}}{\left(\zeta^{r-1}-\epsilon^{r-1}\right)^{m+1}}\,,\qquad m\geq 0;\,\, \ell\in\{0,\dots,r-2\}.$$
Topological recursion for the topologies $(g,n)=(0,3),\,(1,1)$ yields:
\[
\begin{split}
\omega^{[r],\epsilon}_{0,3}(\zeta_1,\zeta_2,\zeta_3) &= \frac{\alpha^{-(r+1)}}{r}\Bigg[\sum\limits_{\substack{0\leq j,k,\ell\leq r-2 \\ j+k+\ell = (r-2)}} \dd \phi_{j,0}(\zeta_1)\dd \phi_{k,0}(\zeta_2)\dd \phi_{\ell,0}(\zeta_3) \\
&+ \epsilon^{r-1}\sum\limits_{\substack{0\leq j,k,\ell\leq r-2\\ j+k+\ell = 2r-3}}\dd \phi_{j,0}(\zeta_1)\dd \phi_{k,0}(\zeta_2)\dd \phi_{\ell,0}(\zeta_3) \Bigg],\\
\omega^{[r],\epsilon}_{1,1}(\zeta_1)&= \alpha^{-(r+1)} \frac{r-1}{24r} \left(\dd \phi_{1,1}(\zeta_1)+ (r-1)\epsilon^{r-1}\dd \phi_{1,2}(\zeta_1)\right).
\end{split}
\]
We observe that, when $\epsilon\to 0$, the differentials have well defined limits, namely:
\[
\begin{split}
\omega^{[r],0}_{0,3}(\zeta_1,\zeta_2,\zeta_3) &= \frac{\alpha^{-(r+1)}}{r}\sum\limits_{\substack{0\leq j,k,\ell\leq r-2 \\ j+k+\ell = r-2}} \dd \phi_{j,0}(\zeta_1)\dd \phi_{k,0}(\zeta_2)\dd \phi_{\ell,0}(\zeta_3), \\
\omega^{[r],0}_{1,1}(\zeta_1)&= \alpha^{-(r+1)}\frac{r-1}{24r} \dd \phi_{1,1}(\zeta_1)\,.
\end{split}
\]

\subsubsection{The case $\lambda_j=0$}
In this model, we suppose that 
$$N\,\, \mathrm{is}\,\, \mathrm{finite},\qquad \lambda_1=\dots=\lambda_N=0,\qquad V(z)=\sum\limits_{j=1}^{r+1}v_j\frac{z^j}{j},\qquad v_2\neq 0\,.$$
Applying Theorem \ref{thm:w01}, $y$ can be written as:
\[
y(\zeta) = \zeta + \frac{N\, \alpha^{-(r+1)}}{Q'(\xi)} \frac{1}{\zeta-\xi}\,,
\]
where we get that $Q(\xi)=V'(0)= v_1$. From this, we deduce that:
\[
\begin{split}
x(\zeta) &= Q(\zeta) = \left[V'(y(\zeta))\right]_+  \\
&= \sum\limits_{\ell=0}^{r} \sum\limits_{j=0}^{\ell} \sum\limits_{0\leq k_1+\dots+ k_j \leq \ell-2j} v_{\ell+1} \begin{pmatrix} \ell \\ j \end{pmatrix} \left(\frac{N\,\alpha^{-(r+1)}}{Q'(\xi)}\right)^j \zeta^{\ell-2j} \left(\frac{\xi}{\zeta}\right)^{k_1+\dots+k_j}.
\end{split}
\] 
In order to completely determine $x$ and $y$, one has to solve
$Q(\xi)=v_1 $. In the end, the spectral curve can be written as:
\[
\begin{cases}
x(\zeta) =  Q(\zeta) \,,\\
y(\zeta) =  \zeta + \frac{N\, \alpha^{-(r+1)}}{Q'(\xi)} \frac{1}{\zeta-\xi}\,.
\end{cases}
\]
If one carries out the change: 
$$\widetilde{x}(u) =  x(u+\xi) \;;\qquad \widetilde{y}(u)=y(u+\xi)-\xi \,.$$
then the correlation functions $\omega^{[r]}_{g,n}(\zeta_1,\dots,\zeta_n)$ computed by topological recursion from both spectral curves are the same. The latter form of the spectral curve has a nice shape:
\[
\begin{cases}
\widetilde{x}(u) =  Q(u+\xi) \,,\\
\widetilde{y}(u) =  u + \frac{N\, \alpha^{-(r+1)}}{Q'(\xi)} \frac{1}{u}\,.
\end{cases}
\]

We carry out an explicit example: let us consider the potential $V(z)= v_6\frac{z^6}{6}+v_4\frac{z^4}{4}+v_2\frac{z^2}{2}$ so that $r=5$. We can also suppose that $N=1$, since combinatorially there is no difference between the unmarked faces.
Then we get:
\[
\begin{split}
x(\zeta)=& \left((3v_4 \xi + 5 v_6 \xi^3)\frac{\alpha^{-6}}{Q'(\xi)}+ 20v_6 \xi \big(\frac{\alpha^{-6}}{Q'(\xi)}\big)^2\right) + \left(v_2 + (3v_4 + 5 v_6 \xi^2)\frac{  \alpha^{-6}}{Q'(\xi)}+10v_6 \big(\frac{  \alpha^{-6}}{Q'(\xi)}\big)^2\right) \zeta \\
&+5v_6 \xi \frac{  \alpha^{-6}}{Q'(\xi)} \zeta^2 + \left(v_4+ 5v_6\frac{  \alpha^{-6}}{Q'(\xi)}\right) \zeta^3 + v_6 \zeta^5\,, \\
y(\zeta)=& \zeta+\frac{  \alpha^{-6}}{Q'(\xi)(\zeta-\xi)} \,.
\end{split}
\]
We need to solve $Q(\xi)=v_1=0$, that is to say:
\[
\xi\left(v_6 \xi^4 + \big(v_4 +20 v_6 \frac{  \alpha^{-6}}{Q'(\xi)} \big)\xi^2 + v_2 + 6 v_4 \frac{  \alpha^{-6}}{Q'(\xi)} + 30 v_6 \big(\frac{  \alpha^{-6}}{Q'(\xi)}\big)^2\right) = 0\,.
\]
This equation may have several solutions, each depending on $\alpha$, and we focus on the solution such that at $\alpha=\infty$, $\xi = \lambda_j$ (recall that $\lambda_j=0$ here). This means that we need the solution $\xi=0$. From this solution, we deduce the value $Q'(0)$ by solving the degree 3 equation: 
\beq\label{eq:example:2}
Q'(0)= v_2 + 3 v_4 \frac{  \alpha^{-6}}{Q'(0)} + 10 v_6 \big(\frac{  \alpha^{-6}}{Q'(0)}\big)^2
\eeq
(we need the solution such that $Q'(0)=v_2$ when $\alpha=\infty$).
In the end, the polynomial $Q$ is completely determined, and so is the spectral curve. To sum up, we get the spectral curve:
\[
\begin{cases}
x(\zeta) =  Q'(0) \zeta + \left(v_4+ 5 v_6 \frac{\alpha^{-6}}{Q'(0)} \right) \zeta^3 + v_6 \zeta^5\,, \\
y(\zeta) = \zeta + \frac{\alpha^{-6}}{Q'(0)} \frac{1}{\zeta}\,,
\end{cases}
\]
where $Q'(0)$ satisfies equation \eqref{eq:example:2}. This spectral curve has 4 ramification points, that are the solutions of:
$$Q'(\zeta)= Q'(0) + 3\big(v_4+ 5 v_6 \frac{\alpha^{-6}}{Q'(0)}\big)\zeta^2+5v_6 \zeta^4=0\,.$$
In this example, one has to find the roots of two polynomials, one of degree 3 (equation \eqref{eq:example:2}) -- which can be solved explicitly -- and the one above, of degree 4. The latter is quite degenerate (there are no odd-power terms) and the roots can be found easily.

\section{Applications: matrix models, integrability and intersection theory}\label{sec:application}

As mentioned in the introduction, $\tau$-functions were introduced as a tool to linearise certain systems of non-linear partial differential equations, the associated linear problem often being related to quantum field theories \cite{SMJ79,JMU81,SegWil85}. Here we will see that the combinatorics developed in this paper relates to an explicit realisation of this correspondence between a solution of the $r$-KdV integrable hierarchy and a matrix model, namely a 0-dimensional quantum field theory given by a \textit{formal matrix integral} \cite{AvM92}. The latter appearing as a generating series for Kontsevich graphs, our only use of the integrable hierarchy will be to bridge the combinatorics of those graphs to the enumerative geometry of r-spin structures over stable Riemann surfaces.

\medskip

The relation between the correlation functions and the intersection numbers of Witten's class is carried out in 3 steps. First (Section \ref{sec:matrix:model}), we relate the combinatorial model of ciliated maps to the cumulants of a hermitian matrix model. Second (Section \ref{sec:avm:fsz}), using a result in \cite{AvM92}, we show that the partition function of this matrix model is a tau-function of a reduction of the KP integrable hierarchy (Section \ref{sec:rkdv:avm}). Thanks to Witten's conjecture \cite{Witt93,FSZ10}, this partition function can be identified with the partition function of intersection numbers of Witten's class (Section \ref{sec:rkdv:fsz}). Third (Section \ref{sec:elsv}), collecting the previous steps together, we obtain an ELSV-like formula, expressing the correlation functions of ciliated maps in terms of intersection numbers of Witten's class. This sequence of steps allows to show that intersection numbers of Witten's class can be computed via topological recursion (Section \ref{sec:extensive:examples}).

\subsection{Relation between generalised Kontsevich graphs and hermitian matrix models}\label{sec:matrix:model}

The Generalised Kontsevich Model (GKM) is the formal matrix integral defined by the partition function 
\begin{eqnarray}\label{eq:partition:function}
Z_{N,\alpha}(V;\lambda) = \int_{\lambda+H_N} \dd M \,\, e^{-N\alpha^{r+1}\,\mathrm{Tr}\,(V(M)-V(\lambda)-(M-\lambda)\,V'(\lambda))}\,,
\end{eqnarray}
loosely denoted $Z=Z_{N,\alpha}(V;\lambda)$, where $H_N$ is the space of $N\times N$ hermitian matrices, we introduced the diagonal matrix $\lambda=\operatorname{diag}(\lambda_1,\dots,\lambda_N)$ of size $N$ and $\Lambda=V'(\lambda)=\operatorname{diag}(\Lambda_1,\dots,\Lambda_N)$ is called the \textit{external field} of the model. In general, formal matrix integrals are defined from the combinatorics of Feynman graphs, themselves obtained by applying Wick's theorem to moments of Gaussian integrals \cite{Pol05}. This can be done only if the exponent has no linear dependence in $M$. We therefore rewrite $M=\lambda+\widetilde M$, with $\widetilde M\in H_N$, such that
\begin{eqnarray}
Z&=&\int_{H_N}  \dd \widetilde{M} \ e^{-N \alpha^{r+1} \left(\frac{1}{2}\sum\limits_{i,j=1}^N \widetilde{M}_{i,j} \widetilde{M}_{j,i} \frac{V'(\lambda_i)-V'(\lambda_j)}{\lambda_i-\lambda_j} - \sum\limits_{\ell=3}^{r+1} \frac{1}{\ell} \sum\limits_{i_1,\dots,i_\ell=1}^N  \widetilde{M}_{i_1,i_2}\widetilde{M}_{i_2,i_3}\dots \widetilde{M}_{i_\ell,i_1} \mathcal{V}_\ell(\lambda_{i_1},\dots,\lambda_{i_\ell})\right)}\\
&=& \sum_{k\geq0} \frac{(N\alpha^{r+1})^k}{k!} \int_{H_N}  dM\Big(\sum\limits_{\ell=3}^{r+1} \frac{1}{\ell} \sum\limits_{i_1,\dots,i_\ell=1}^N  \widetilde{M}_{i_1,i_2}\widetilde{M}_{i_2,i_3}\dots \widetilde{M}_{i_\ell,i_1} \mathcal{V}_\ell(\lambda_{i_1},\dots,\lambda_{i_\ell})\Big)^k\nonumber\\
&&\qquad\qquad\qquad\qquad\qquad\qquad\qquad\qquad\qquad\qquad \times\quad e^{-N \alpha^{r+1} \frac{1}{2}\sum\limits_{i,j=1}^N \widetilde{M}_{i,j} \widetilde{M}_{j,i} \frac{V'(\lambda_i)-V'(\lambda_j)}{\lambda_i-\lambda_j}}\,,
\end{eqnarray}
obtained by permuting sum and integral, is a series whose coefficients feature integrals that are all polynomial expressions of $\frac1\alpha$. Upon considering the maps that are dual to the Feynman graphs obtained through the application of Wick's theorem to the integrals appearing in the last expression, we get that it is related to the generating series of Kontsevich graphs
\begin{eqnarray}\label{eq:MM-graphs}
\log\frac{Z}{Z_0} =\sum\limits_{g\geq 0}\sum\limits_{G\in\mathcal{F}^{[r]}_{g,0}} \frac{N^{-\frac{\deg G}{r+1}}\alpha^{-\deg G}}{\# \mathrm{Aut} \, G}\prod\limits_{\substack{e\in\mathcal{E}(G)\\ e=(f_1,f_2)}} \mathcal{P}(\lambda_{f_1},\lambda_{f_2}) \prod\limits_{\substack{v\in\mathcal{V}(G)}} \mathcal{V}_{d_v}\left(\{\lambda_f\}_{f\mapsto v}\right),
\end{eqnarray}
where the logarithm restricts us to connected graphs as is customary and where we used the combinatorial notations of the paper. The normalisation is the Gaussian integral
\[
\begin{split}
Z^{N,\alpha}_0(V;\lambda) &= \int_{H_N}  \dd \widetilde{M} \,\, e^{-\frac{N\, \alpha^{r+1}}{2}\sum\limits_{i,j=1}^{N} \frac{\widetilde{M}_{i,j}\widetilde{M}_{j,i}}{\mathcal{P}(\lambda_i,\lambda_j)}} \\
&= \Big(\frac{2\pi}{N\, \alpha^{r+1}} \Big)^{\frac{N^2}{2}} \frac{1}{\prod\limits_{j=1}^N \sqrt{\mathcal{P}(\lambda_j,\lambda_j)}}\,\,\frac{1}{\prod\limits_{1\leq i<j\leq N} \mathcal{P}(\lambda_i,\lambda_j)}\,,
\end{split}
\]
loosely denoted $Z_0=Z^{N,\alpha}_0(V;\lambda)$. The coefficients appearing in front of each power of $\alpha$ in \eqref{eq:MM-graphs} involve sums over the finite sets of graphs of given degrees. $Z_{N,\alpha}(V;\lambda)$ is therefore the product of the Gaussian normalisation and the exponential of a well-defined Laurent series in $\frac1\alpha$. It takes the form
\begin{eqnarray}
Z_{N,\alpha}(V;\lambda) = Z_0^{N,\alpha}(V;\lambda)\,e^{\sum_{g\geq0}(N\alpha^{r+1})^{2-2g}F_g^\alpha(V;\lambda)}\,,
\end{eqnarray}
with \textit{free energies} $F_g^\alpha(V;\lambda)$ that are themselves Laurent series in $\frac1\alpha$\,.


\medskip

Let us stress that a matrix model with external field whose partition function is proportional to the one exposed in equation \eqref{eq:partition:function} has already been studied from the topological recursion point of view. More precisely, in \cite{EO07inv} and \cite{EO09}, Orantin and one of the authors derived the loop equations and spectral curve of the matrix model whose partition function is 
\begin{eqnarray}\label{eq:mat:mod}
Z'_{N,\alpha}(V;\lambda) = \int_{\lambda+H_N}  \dd M \,\, e^{-N\alpha^{r+1}\,\mathrm{Tr}\,(V(M)-M\,\Lambda)}.
\end{eqnarray}
Moreover, in \cite{CEO06}, together with Chekhov, they studied the 2-matrix model and some of the techniques employed for the solution were used in the previous sections.
The difference between our model and the model studied in \cite{EO07inv} lies in the correlation functions that we consider. In \cite{EO07inv} and \cite{EO09}, one of the authors and Orantin proved that the correlation functions:
\begin{eqnarray}
\left\langle \mathrm{Tr}\frac{1}{x_1-{M}}\cdots \mathrm{Tr}\frac{1}{x_n-{M}} \right\rangle _c\,,
\end{eqnarray}
obtained as the cumulants of the generalised resolvents 
\begin{eqnarray}
\left\langle \prod_{i=1}^n\big(\mathrm{Tr}\frac{1}{x_i-{M}}\big) \right\rangle = \frac1{Z_{N,\alpha}(V;\lambda)} \int_{H_N}  \dd M\,  \prod_{i=1}^n\big(\mathrm{Tr}\frac{1}{x_i-{M}}\big)\,e^{-N\alpha^{r+1}\,\mathrm{Tr}\,(V(M)-V(\lambda)-(M-\lambda)\,V'(\lambda))}\,,
\end{eqnarray}
defined for all $n>0$ as a Laurent series expansion when $x_k\rightarrow\infty$, for all $1\leq k\leq n$, admit topological expansions computed by topological recursion, while in the present article as we shall shortly see, we obtain that the correlators
\begin{eqnarray}\label{eq:correl:MM}
\big\langle \widetilde{M}_{i_1,i_1} \cdots \widetilde{M}_{i_n,i_n} \big\rangle_c = \frac1{(N\alpha^{r+1})^n} \frac\partial{\partial\Lambda_{i_1}}\cdots\frac\partial{\partial\Lambda_{i_n}} \log Z\,,
\end{eqnarray}
obtained as cumulants of the diagonal correlators
\begin{eqnarray}
\big\langle \widetilde{M}_{i_1,i_1} \cdots \widetilde{M}_{i_n,i_n} \big\rangle &=& \frac1{Z_{N,\alpha}(V;\lambda)} \int_{H_N}  \dd \widetilde M \ \widetilde{M}_{i_1,i_1} \cdots \widetilde{M}_{i_n,i_n}\,e^{-N\alpha^{r+1}\,\mathrm{Tr}\,(V(\lambda+\widetilde M)-V(\lambda)-\widetilde M\,V'(\lambda))}\nonumber\\
&=& \frac{(N\alpha^{r+1})^{-n}}{Z_{N,\alpha}(V;\lambda)}\frac\partial{\partial\Lambda_{i_1}}\cdots\frac\partial{\partial\Lambda_{i_n}} Z_{N,\alpha}(V;\lambda)\,,
\end{eqnarray}
for all $n\geqslant1$, have topological expansions computed by topological recursion. The subscript $c$ stands for the cumulant expectation value. 

\smallskip

These different choices of correlation functions are however not unrelated. Indeed, our spectral curve (Definition \ref{def:spectral:curve}) and the one obtained with Orantin (Section 5.4.2 of \cite{EO09}) are related by exchanging the roles played by $x$ and $y$.

\smallskip

This is especially valuable as it provides a clear identification of the topological recursion outputs corresponding to two spectral curves with defining $x$ and $y$ maps exchanged: the deepest and least understood of the transformations modifying embeddings $(x,y):\Sigma\longrightarrow\mathbb C^2$ of spectral curves while preserving the $2$-form $d x \wedge dy$ on $\mathbb C^2$ up to constant multiplicative factors. 

\smallskip

Its importance stems from the conjecture that there should exist a general correspondence between the respective outputs of the applications of the topological recursion to spectral curves related by such a transformation. This was coined as the yet to be established \textit{symplectic invariance property of topological recursion} \cite{EO07inv,CEO06,EOxy}. We will further explore the corresponding implications of our work in the future.

\br 
Another instance in which the relationship between the combinatorics corresponding to $x-y$ exchanged spectral curves was studied is that of the $1$-hermitian matrix model, which enumerates combinatorial maps, while the topological recursion after exchanging $x$ and $y$ conjecturally enumerates fully simple maps \cite{BG-F18}. We  investigate the relation to this project in an upcoming work. One important relation of these two combinatorial types of maps is through monotone Hurwitz numbers (see \cite{BCDG-F19} for a combinatorial proof of this statement). \hfill $\star$
\er

For $i_1\neq \cdots \neq i_n$, the precise relation with our combinatorial problem for $n>0$ turns out to be
\begin{equation}\label{MMcombi}
\left\langle \widetilde{M}_{i_1,i_1} \cdots \widetilde{M}_{i_n,i_n} \right\rangle_c = \sum_{g\geq 0} N^{2-2g-n} W_{g,n}^{[r]}(\lambda_{i_1},\ldots,\lambda_{i_1})\,,
\end{equation}
to be understood again as an equality between Laurent series in $\alpha^{-1}$ (and not large $N$ expansions despite what it might seem). At the combinatorial level, the condition $i_1\neq \cdots \neq i_n$ in the left-hand-side amounts to each of the cilia being in a different marked face labeled with $\lambda_{i_j}$, $j=1,\ldots,n$, which is one of the restrictions that ciliated Kontsevich graphs in $\mathcal{W}_{g,n}^{[r]}(\lambda_{i_1},\ldots,\lambda_{i_1})$ satisfy. 

\subsection{From the matrix model to intersection numbers}\label{sec:avm:fsz}

In this section, we establish the bridge from the matrix model from previous section to the intersection theory on the moduli space of stable curves $\overline{\mathfrak{M}}_{g,n}$ through the $r$-Gelfand--Dikii hierarchy, also known as the $r$-KdV hierarchy.

\subsubsection{From the matrix model to integrable hierarchies}\label{sec:rkdv:avm}

It was shown in \cite{AvM92} that a particular instance of the GKM provided the solution to a seemingly totally unrelated problem, namely that of finding the unique solution of the $r$-KdV hierarchy satisfying the string equation. We first expose the result of Adler--van Moerbeke, and then we give the relation to the GKM in our setup. The $r$-KdV hierarchy is an infinite sequence of compatible hamiltonian flows on a space of ordinary differential operators (see \cite{AvM92} for a review with slightly different notations from ours). 

\smallskip

Following the conventions of Adler--van Moerbeke, define the differential operators:
\begin{equation}\label{eq:diff:operators:avm}
D\coloneqq \frac{\partial}{\partial x}, \qquad \underline{\mathsf{T}}\coloneqq (\mathsf{T}_1 \equiv x, \mathsf{T}_2, \mathsf{T}_3 ,\cdots),\qquad L\coloneqq D^{r}+\sum\limits_{j=2}^{r} q_j(\underline{\mathsf{T}}) D^{r-j}. 
\end{equation}
The $r^{\textup{th}}$ reduction of the KP hierarchy (or $r$-reduced Gelfand--Dikii, or $r$-KdV for short) is the following set of equations on the coefficients $q_j(\underline{\mathsf{T}})$:
\begin{equation}\label{eq:rkdv:avm}
\frac{\partial L}{\partial \mathsf{T}_k} =\big[(L^{\frac{k}{r}})_+,L\big].
\end{equation}
There are infinitely many solutions to this integrable hierarchy, however, fixing the initial conditions allows to define a solution uniquely. One way to do so is to impose the \emph{string equation}. In the convention of \cite{AvM92}, the latter takes the following form:  
\begin{equation}\label{eq:string:avm}
[L,P]=1,
\end{equation}
where $P$ is an explicit differential operator.

\begin{thm}[\cite{AvM92}]\label{thm:avm}
The following partition function:
\begin{equation}\label{eq:avm}
Z^{\textup{AvM}} = e^{F^{\textup{AvM}}}\coloneqq \frac{\int_{H_N} \dd M e^{-N \iu \sqrt{r}\tr \Big( \frac{(M+y)^{r+1}}{r+1}-M\, y^r - \frac{y^{r+1}}{r+1}}\Big)}{\int_{H_N} \dd M e^{-\frac{N \iu\sqrt{r}}{2}\tr \sum\limits_{m=0}^{r-1}y^m M y^{r-1-m} M}}
\end{equation}
is the unique tau function of the $r$-KdV hierachy \eqref{eq:rkdv:avm} satisfying the string equation \eqref{eq:string:avm}. The times are expressed in terms of the (diagonal) external field matrix $y$:
\begin{equation}\label{eq:times:avm}
\mathsf{T}_k\coloneqq \frac{1}{k} \sum\limits_{\ell=1}^{N}\frac{1}{(N^{\frac{1}{r+1}}y_{\ell})^k}.
\end{equation}
\hfill $\star$
\end{thm}
\br
In \cite{AvM92}, the partition function is expressed in terms of an \textit{anti-hermitian} matrix model. In order to match their result with our setting, we use the expression for the partition function of \cite{BBCCN18}. \hfill $\star$
\er

\br
The relation between the tau-function and the differential operator $L$ is:
\[
\forall 0\leq j\leq r-2,\qquad \frac{\partial^2 F^{\textup{AvM}}}{\partial \mathsf{T}_1 \partial \mathsf{T}_{j+1}} = \Res L^{\frac{j+1}{r}}.
\]
\hfill $\star$
\er
The GKM we use to study the ciliated maps is related to the GKM of Adler--van Moerbeke, given that we carry out the following identifications:
\[
V(z) = \frac{z^{r+1}}{r+1}, \qquad \widetilde{M}= \frac{(\iu \sqrt{r})^{\frac{1}{r+1}}}{\alpha} M,\qquad \lambda =  \frac{(\iu \sqrt{r})^{\frac{1}{r+1}}}{\alpha} y.
\]
Then, the times are given by:
\begin{equation}\label{eq:times:avm:gkm}
\mathsf{T}_k = \frac{1}{k} \left(\frac{\iu \sqrt{r}}{\alpha^{r+1} N}\right)^{\frac{k}{r+1}} \sum\limits_{\ell=1}^{N} \frac{1}{\Lambda_{\ell}^{\frac{k}{r}}}.
\end{equation}
The free energy (or equivalently the logarithm of the tau-function) of the GKM is expressed as a formal power series in those times:
\[
F=\sum\limits_{g,n} \sum\limits_{k_1,\dots,k_n\geq 1}\frac{F_{g,n}(k_1,\dots,k_n)}{n!} \mathsf{T}_{k_1} \dots\mathsf{T}_{k_n} = \sum\limits_{g\geq 0} F_{g}.
\]

\subsubsection{From integrable hierarchies to intersection theory}\label{sec:rkdv:fsz}
In this section we relate the generating series of ciliated maps to the correlators of the r-spin cohomological field theory, that is all the intersections of Witten's r-spin class with $\psi$-classes over $\overline{\mathfrak{M}}_{g,n}$. This provides the relation between the A-side and the C-side that we presented in the introduction.

\medskip

In \cite{Witt93}, Witten conjectured that the partition function of the intersection numbers of Witten r-spin class with psi classes on $\overline{\mathfrak{M}}_{g,n}$ is the tau-function of the $r$-KdV hierarchy which satisfies the string equation. This conjecture was proved  by Faber, Shadrin and Zvonkine in \cite{FSZ10}. Since there is a unique solution to the hierarchy which also satisfies the string equation, the tau-functions from the GKM and from the intersection theory context are equal. We describe more precisely the steps here.

\medskip

Witten and Faber--Shadrin--Zvonkine use different conventions from Adler--van Moerbeke concerning the $r$-KdV hierarchy and the string equation. In their setup, the differential operators are defined by:
\begin{equation}\label{eq:diff:operators:fsz}
D\coloneqq \frac{\iu}{\sqrt{r}}\frac{\partial}{\partial x},\qquad \underline{\widetilde{t}} \coloneqq (\widetilde{t}_{d,j})_{\substack{d\in \mathbb{Z}_{\geq 0} \\ j\in\{0,\dots,r-2\}}},\quad t_{0,0}\equiv x, \qquad Q\coloneqq D^{r}-\sum\limits_{j=0}^{r-2}u_j(\underline{\widetilde{t}}) D^j.
\end{equation}
The $r$-KdV hierarchy is the following infinite set of equations:
\begin{equation}\label{eq:rkdv:fsz}
\forall d\in\mathbb{Z}_{\geq 0},\, j\in\{0,\dots,r-2\},\qquad c_{d,j}(d+\frac{j+1}{r})\iu \sqrt{r}\frac{\partial Q}{\partial \widetilde{t}_{d,j}} =\Big[\big(Q^{d+\frac{j+1}{r}}\big)_+,Q\Big],
\end{equation}
where $c_{d,j}$ is defined by: $c_{d,j}\coloneqq (-1)^d \frac{\Gamma\big(d+\frac{j+1}{r}\big)}{\Gamma\big(\frac{j+1}{r}\big)}$. Now, suppose that $Z=\exp(F)$ is a tau-function of the $r$-KdV hierarchy \eqref{eq:rkdv:fsz}. The string equation admits the following formulation:
\begin{equation}\label{eq:string:fsz}
\frac{\partial F(\underline{t})}{\partial \widetilde{t}_{0,0}} =\frac{1}{2}\sum\limits_{j,\ell=0}^{r-2} \eta^{j\,\ell} \widetilde{t}_{0,j}\widetilde{t}_{0,\ell}+\sum\limits_{d=1}^{\infty}\sum\limits_{j=0}^{r-2} \widetilde{t}_{d+1,j}\frac{\partial F}{\partial \widetilde{t}_{d,j}}\quad ; \quad \eta^{j\,\ell} = \delta_{j+\ell,r-2}.
\end{equation}
In what follows, we use the standard notation for the intersection numbers of Witten's class with $\psi$-classes:
\begin{equation}\label{eq:notation:witten:class}
\left\langle\tau_{d_1,j_1}\cdots \tau_{d_n,j_n}\right\rangle_{g} = \int_{\overline{\mathfrak{M}}_{g,n}}c_{\textup{W}}(j_1,\dots,j_n) \psi_{1}^{d_1}\cdots \psi_n^{d_n}.
\end{equation}
Notice that, for cohomological degree reasons, the intersection is non-vanishing only if:
\begin{equation}\label{eq:constraint:intersection}
\sum\limits_{i=1}^{n} (r\, d_i +j_i+1)= (r+1)(2g-2+n)\,.
\end{equation}
Witten's conjecture~\cite{Witt93} from 1993, proved by Faber, Shadrin and Zvonkine in~\cite{FSZ10}, states:
\begin{thm}[\cite{FSZ10}]\label{thm:fsz}
The generating series of intersection numbers of Witten's class with $\psi$-classes is the unique $\tau$-function of the $r$-KdV (or $r$-Gelfand--Dikii) hierarchy satisfying the string equation. Explicitly, the following generating series
\begin{equation}\label{eq:FSZ}
Z^{\textup{W}}=\exp \Big(F^{\textup{W}}(\underline{\widetilde{t}}) \Big)= \exp\Bigg(\sum\limits_{g\geq 0} F^{\textup{W}}_{g}(\underline{\widetilde{t}})\Bigg)=\exp\Bigg(\sum\limits_{g\geq 0} \Big\langle\exp\big(\sum\limits_{\substack{d\geq 0 \\ 0\leq j\leq r-1}} \widetilde{t}_{d,j}\tau_{d,j}\big)\Big\rangle_g\Bigg) 
\end{equation}
is the unique tau-function of the $r$-KdV hierarchy \eqref{eq:rkdv:fsz} which satisfies the string equation \eqref{eq:string:fsz}.
\hfill $\star$
\end{thm}
\br
Again, the tau-function is related to the functions $u_j$ (coefficients of $Q$ \eqref{eq:diff:operators:fsz}) satisfying the integrable hierarchy, through the following formula:
\[
\forall j\in\{0,\dots,r-2\},\qquad \frac{\partial^2 F^{\textup{W}}(\underline{\widetilde{t}})}{\partial \widetilde{t}_{0,0}\partial \widetilde{t}_{0,j}} = -\frac{r}{j+1} \Res Q^{d+\frac{j+1}{r}}.
\]
\hfill $\star$
\er
The free energy is expressed as a formal power series in the times:
\[
F^{\textup{W}}(\underline{\widetilde{t}}) = \sum\limits_{g,n}\sum\limits_{\substack{d_1,\dots,d_n\geq 0 \\ 0\leq j_1,\dots,j_n \leq r-2}} \frac{1}{n!}\left\langle \tau_{d_1,j_1}\dots \tau_{d_n,j_n}\right\rangle_g \widetilde{t}_{d_1,j_1}\dots \widetilde{t}_{d_n,j_n}. 
\]
In the previous equation, notice that the summation set of the indices $j$ is reduced to $\{0,\dots,r-2\}$, instead of $\{0,\dots,r-1\}$, because the intersection numbers vanish whenever one of the $j$s is equal to $r-1$.

\subsubsection{Identifying the parameters in the various setups}
For the GKM, let us denote $Z^{[r]}(\lambda) =  Z_{N,\alpha}\big(\frac{z^{r+1}}{r+1};\lambda\big)$ and 
\begin{equation}\label{eq:free:energy:MM}
F^{[r]}(\lambda)=\log Z^{[r]}(\lambda) = \sum\limits_{g\geq 0} F^{[r]}_{g}(\lambda)\,.
\end{equation}

Now, in order to relate the GKM with the descendant potential of Witten's class $F^{\textup{W}}(\underline{\widetilde{t}})$:
\begin{itemize}
	\item the solution of the $r$-KdV hierarchy and the string equation (equations \eqref{eq:rkdv:avm} and \eqref{eq:string:avm} with the conventions of Adler--van Moerbeke, equations \eqref{eq:rkdv:fsz} and \eqref{eq:string:fsz} with the conventions of Faber--Shadrin--Zvonkine) is unique;
	\item by Theorems \ref{thm:avm} and \ref{thm:fsz}, $Z^{\textup{AvM}}(\underline{\mathsf{T}})$ (Equation \eqref{eq:avm}) and $Z^{\textup{W}}(\underline{\widetilde{t}})$ (Equation \eqref{eq:FSZ}) are tau-functions of the $r$-KdV hierarchy and satisfy the string equation -- with different conventions though. 
\end{itemize}
We deduce that, with a certain identification of the times $\lambda, \,\,\underline{\mathsf{T}}$ and $\underline{\widetilde{t}}$ in order to make the conventions coincide, we obtain:
\begin{equation}\label{eq:identification:free:energies}
F^{[r]}(\lambda) = F^{\textup{AvM}}(\underline{\mathsf{T}})=F^{\textup{W}}(\underline{\widetilde{t}}),
\end{equation}
where the times are related through:
\begin{equation}\label{eq:identification:times}
\widetilde{t}_{d,j}= \iu \sqrt{r}\, c_{d,j}\big(d+\frac{j+1}{r}\big) \mathsf{T}_{rd+j+1} = \frac{\iu}{\sqrt{r}}c_{d,j} \left(\frac{\iu \sqrt{r}}{\alpha^{r+1}N}\right)^{\frac{rd+j+1}{r+1}} \sum\limits_{\ell=1}^{N}\frac{1}{\Lambda_{\ell}^{d+\frac{j+1}{r}}}.
\end{equation}
The first equality is derived by comparing the definitions of the differential operators in Equations \eqref{eq:diff:operators:avm} and \eqref{eq:diff:operators:fsz}; the second one uses Equation \eqref{eq:times:avm:gkm}. 

\medskip

Before proceeding, let us carry out some changes in the times. If one defines:
\begin{equation}\label{eq:renormalised:times}
t_{d,j} \coloneqq \frac{-1}{\alpha^{r+1}N} c_{d,j}\sum\limits_{\ell=1}^{N} \frac{1}{\Lambda_{\ell}^{d+\frac{j+1}{r}}}, 
\end{equation}
then, the constraint \eqref{eq:constraint:intersection} for the intersection numbers being non-zero, yields:
\begin{equation}\label{eq:fsz:renormalised}
F^{\textup{W}} =\sum\limits_{g\geq 0} \hslash^{g-1} \left\langle e^{\sum\limits_{d,j}t_{d,j}\tau_{d,j}} \right\rangle_g,\qquad \hslash \coloneqq \frac{-r}{\alpha^{2(r+1)}N^2}.
\end{equation}

Combining \eqref{eq:identification:free:energies} and \eqref{eq:fsz:renormalised}, we can express the free energy of the GKM in terms of intersection numbers. Namely,
\begin{equation}\label{eq:MM:int}
F^{[r]}(\lambda) = \sum\limits_{g\geq 0} \hslash^{g-1} \Big\langle\exp\big(\sum\limits_{\substack{d\geq 0 \\ 0\leq j\leq r-2}} t_{d,j}\tau_{d,j}\big)\Big\rangle_g,
\end{equation}
provided that
\begin{equation}\label{eq:def:t:c}
V(z)=\frac{z^{r+1}}{r+1};\quad\quad\hslash=\frac{-r}{N^2 \alpha^{2(r+1)}};\quad \quad t_{d,j}=\frac{-c_{d,j}}{\alpha^{r+1}N}\sum\limits_{k=1}^{N} \Lambda_k^{-d-\frac{j+1}{r}}; \quad \quad c_{d,j}=(-1)^d \frac{\Gamma\left(d+\frac{j+1}{r}\right)}{\Gamma\left(\frac{j+1}{r}\right)}.
\end{equation}
With those conventions, the string equation \eqref{eq:string:fsz} becomes (see Appendix \ref{Appendix} for a derivation of the equation in this context):
\begin{equation}\label{eq:string:ciliated}
\frac{\partial F^{[r]}}{\partial t_{0,0}} =\frac{\hslash^{-1}}{2}\sum\limits_{j,\ell=0}^{r-2} \eta^{j\,\ell} t_{0,j}t_{0,\ell}+\sum\limits_{d=1}^{\infty}\sum\limits_{j=0}^{r-2} t_{d+1,j}\frac{\partial F^{[r]}}{\partial t_{d,j}}\quad ; \quad \eta^{j\,\ell} = \delta_{j+\ell,r-2}.
\end{equation}

\subsection{ELSV-type formula for the Witten r-spin class}\label{sec:elsv}

Notice that Equation \eqref{eq:renormalised:times} allows to express the derivation with respect to the external field of the GKM in terms of the renormalised times:
\begin{equation}\label{eq:derivation:times}
\frac{\partial}{\partial \Lambda_{\ell}} =\frac{-1}{\alpha^{r+1}N}\sum\limits_{\substack{d\geq 0 \\ 0\leq j\leq r-1}} \frac{c_{d+1,j}}{\Lambda_{\ell}^{d+1+\frac{j+1}{r}}}\frac{\partial}{\partial t_{d,j}}. 
\end{equation}

We are able to express the generating series of ciliated maps in terms of intersection numbers of Witten's class with $\psi$-classes on $\overline{\mathfrak{M}}_{g,n}$. Namely, from equations \eqref{eq:correl:MM}, \eqref{MMcombi}, \eqref{eq:derivation:times} and \eqref{eq:MM:int} we get:
\begin{eqnarray}\label{eq:W:int}
W^{[r]}_{g,n}(\lambda_{i_1},\dots,\lambda_{i_n})&\overset{\eqref{eq:correl:MM}}{=}&N^{2g-2+n} \frac{\partial}{\partial \Lambda_{i_1}}\dots\frac{\partial}{\partial \Lambda_{i_n}} F^{[r]}_{g}(\lambda) \\
&\overset{\eqref{eq:derivation:times}}{=}&N^{2g-2+n} \left(\frac{-1}{\alpha^{r+1}N}\right)^{n} \sum\limits_{\substack{d_1,\dots,d_n\geq 0\\0\leq j_1,\dots, j_n \leq r-1}} \prod\limits_{\ell=1}^{n} \frac{c_{d_{\ell}+1,j_{\ell}}}{\Lambda_{i_{\ell}}^{d_{\ell}+1+\frac{j_{\ell}+1}{r}}}\frac{\partial}{\partial t_{d_1,j_1}}\dots\frac{\partial}{\partial t_{d_n,j_n}} F^{[r]}_{g}(\lambda)\cr
&\overset{\eqref{eq:MM:int}}{=}& \frac{(-1)^{n}(-r)^{g-1}}{\alpha^{(r+1)(2g-2+n)}} \sum\limits_{\substack{d_1,\dots,d_n\geq 0\\0\leq j_1,\dots, j_n \leq r-2}} \prod\limits_{\ell=1}^{n} \frac{c_{d_{\ell}+1,j_{\ell}}}{\Lambda_{i_{\ell}}^{d_{\ell}+1+\frac{j_{\ell}+1}{r}}} \left\langle \tau_{d_1,j_1}\dots \tau_{d_n,j_n}e^{\sum\limits_{d,j}t_{d,j}\tau_{d,j}}\right\rangle_{g}.\nonumber
\end{eqnarray}

Let us define:
\begin{equation}\label{eq:def:omega:int}
\omega^{\textup{W}}_{g,n}(\zeta_1,\dots,\zeta_n) \coloneqq \frac{(-1)^{n} (-r)^{g-1}}{\alpha^{(r+1)(2g-2+n)}}\sum\limits_{\substack{d_1,\dots,d_n\geq0\\ 0\leq j_1,\dots,j_n \leq r-1}}\prod\limits_{\ell=1}^{n}\frac{c_{d_{\ell}+1,j_{\ell}}\dd Q(\zeta_{\ell})}{Q(\zeta_{\ell})^{d_{\ell}+1+\frac{j_{\ell}+1}{r}}} \left\langle\prod\limits_{i=1}^{n}\tau_{d_i,j_i}e^{\sum\limits_{d,j} t_{d,j}\tau_{d,j}}\right\rangle_g\,.
\end{equation}

\begin{lem}
Let $g,n\geq 0$ such that $2g-2+n\geq 1$. Suppose that:
\begin{enumerate}
	\item $N\geq (r+1)(2g-1+n)$;
	\item $Q(\zeta)=\zeta^r$ and $\sum\limits_{k=1}^{N}\frac{1}{\lambda_{k}^{j+1}}=0$ for $j=0,\dots r$;
	\item for $i,j\geq n$, $\lambda_i\neq \lambda_j$ when $i\neq j$ (the $\lambda$'s are in \emph{generic position}). 
\end{enumerate}
Then
\begin{equation}\label{eq:omega:int:GKM}
\omega^{\textup{W}}_{g,n}(\zeta_1,\dots,\zeta_n) = \omega^{[r]}_{g,n}(\zeta_1,\dots,\zeta_n)\,.
\end{equation}
\hfill $\star$
\end{lem}

\br
The second assumption imposes $r+1$ constraints on the $\lambda$'s, so it is consistent with assumptions 1 and 3. Indeed, imposing $Q(\zeta)=\zeta^r$ is equivalent to the $r-1$ constraints
\[
\sum\limits_{k=1}^{N} \frac{\xi_k^j}{Q'(\xi_k)}=0,\qquad j=0,\dots,r-2
\]
(see Definition \ref{def:spectral:curve} and equation \eqref{eq:Q} for the definition of $Q$ in terms of $\xi_k$). Once they are satisfied, we have: $\xi_k=\lambda_k$ and $Q'(\xi_k)=r\, \lambda_k^{r-1}$, so the constraint $Q(\zeta)=\zeta^r$ gives the $r-1$ constraints
\[
\sum\limits_{k=1}^{N} \frac{1}{\lambda_k^{j+1}}=0,\qquad j=0,\dots,r-2.
\] 
Then we only need to impose 2 additional constraints to get the second assumption:
\[
\sum\limits_{k=1}^{N} \frac{1}{\lambda_k^{j+1}}=0,\qquad j=r-1,\,r.
\] 
\hfill $\star$
\er

\begin{proof}
Using the second assumption, we can write:
\[
\begin{split}
\omega^{\textup{W}}_{g,n}(\zeta_1,\dots,\zeta_n) = \frac{(-r)^{g-1+n}}{\alpha^{(r+1)(2g-2+n)}}\sum\limits_{\substack{d_1,\dots,d_n\geq0\\ 0\leq j_1,\dots,j_n \leq r-1}}\prod\limits_{\ell=1}^{n}\frac{c_{d_{\ell}+1,j_{\ell}}\dd \zeta_{\ell}}{\zeta_{\ell}^{r\, d_{\ell}+j_{\ell}+2}} \left\langle\prod\limits_{i=1}^{n}\tau_{d_i,j_i}e^{\sum\limits_{d,j} t_{d,j}\tau_{d,j}}\right\rangle_g.
\end{split}
\]
Expanding the exponential in the intersection numbers, we end up with terms of the form:
\[
\prod\limits_{k=1}^{m}t_{d'_k,j'_k}\left\langle\prod\limits_{i=1}^{n} \tau_{d_i,j_i} \prod\limits_{\ell=1}^{m} \tau_{d'_{\ell},j'_{\ell}}\right\rangle_{g}\,.
\]
The intersection is non-vanishing only if:
\begin{equation}
\sum\limits_{i=1}^{n} (r\, d_i +j_i+1)+\sum\limits_{\ell=1}^{n} (r\, d'_{\ell} +j'_{\ell}+1) = (r+1)(2g-2+n+m)\,.
\end{equation}
The second assumption states also that $t_{0,j}=0$ and $t_{1,0}=0$. As a consequence, the number of non-vanishing terms satisfying the constraint \eqref{eq:constraint:intersection} is finite. Therefore, $\omega^{\textup{W}}_{g,n}(\zeta_1,\dots,\zeta_n)$ is a finite sum.\\

Let us focus on the $1$-differentials $\frac{\omega^{\textup{W}}_{g,n}(\zeta,\xi_{1},\dots,\xi_{n-1})}{\dd \xi_1 \dots \dd \xi_{n-1}}$ and $\frac{\omega^{[r]}_{g,n}(\zeta,\xi_{1},\dots,\xi_{n-1})}{\dd \xi_1 \dots \dd \xi_{n-1}}$. Note that since $Q(\zeta)=\zeta^r$, $\xi_j=\lambda_j$.\\
They both have one pole, located at $\zeta=0$, of order $r(2g-2+n)+2g$. By Riemann--Roch formula, the dimension $\mathfrak{d}$ of the space of meromorphic differentials with a pole of order $r(2g-2+n)+2g$ is 
\[
\mathfrak{d}=r(2g-2+n)+2g-1.
\]
Since -- by equation \eqref{eq:W:int} -- the expressions $\frac{\omega^{\textup{W}}_{g,n}(\zeta,\xi_{1},\dots,\xi_{n-1})}{\dd \xi_1 \dots \dd \xi_{n-1}}$ and $\frac{\omega^{[r]}_{g,n}(\zeta,\xi_{1},\dots,\xi_{n-1})}{\dd \xi_1 \dots \dd \xi_{n-1}}$ coincide at $\zeta = \lambda_j$ for all $j\geq n$, and since, by the third assumption, $\lambda_j \neq \lambda_k$ for $j,k\geq n$, $j\neq k$, the differentials agree on $N-(n-1)$ different values. By the first assumption, $N-(n-1)\geq r(2g-1+n)+2g\geq r(2g-2+n)+2g=\mathfrak{d}$. Therefore the differentials agree for all $\zeta$:
\[
\omega^{\textup{W}}_{g,n}(\zeta,\xi_{1},\dots,\xi_{n-1}) = \omega^{[r]}_{g,n}(\zeta,\xi_{1},\dots,\xi_{n-1})\,.
\]
Following the same path for the other arguments, we end up with the equality:
\[
\omega^{\textup{W}}_{g,n}(\zeta_1,\dots,\zeta_{n}) = \omega^{[r]}_{g,n}(\zeta_1,\dots,\zeta_{n})\,.
\]
\end{proof}
Note that $\omega^{\textup{W}}_{g,n}(\zeta_1,\dots,\zeta_{n})$ and $\omega^{[r]}_{g,n}(\zeta_1,\dots,\zeta_{n})$ are both continuous with respect to the parameters $\lambda_1,\dots, \lambda_N$. Therefore, equation \eqref{eq:omega:int:GKM} is still true in the limit $\lambda_i=\lambda_j$, and hence we can relax the third hypothesis. We can also relax the first assumption: suppose $N<(r+1)(2g-1+n)$. Then we can introduce $\lambda_{N+1}=\dots=\lambda_{(r+1)(2g-1+n)} =\infty$. Those additional $\lambda$'s do not change the differentials, but they allow to always assume that the first hypothesis is true.

\begin{cor}
Let $g,n\geq 0$ be such that $2g-2+n\geq 1$. Suppose that:
\[
Q(\zeta)=\zeta^r\qquad \mathrm{and}\qquad \sum\limits_{k=1}^{N}\frac{1}{\lambda_{k}^{j+1}}=0\,,\qquad \mathrm{for }\;\;  j=0,\dots r\,.
\]
Then
\begin{equation}\label{eq:omega:int:GKM:relaxed}
\omega^{\textup{W}}_{g,n}(\zeta_1,\dots,\zeta_n) = \omega^{[r]}_{g,n}(\zeta_1,\dots,\zeta_n)\,.
\end{equation}
\hfill $\star$
\end{cor}
In the end, we proved that, under certain assumptions on the times, the generating series of Generalised Kontsevich Graphs for the potential $V(z)=\frac{z^{r+1}}{r+1}$ are generating series for the intersection numbers of Witten's class:
\begin{thm}\label{thm:GKM:int}
Let $g,n\geq 0$ such that $2g-2+n\geq 1$. Suppose that:
\[
Q(\zeta)=\zeta^r\qquad \mathrm{and}\qquad \sum\limits_{k=1}^{N}\frac{1}{\lambda_{k}^{j+1}}=0\,,\qquad \mathrm{for}\;\; j=0,\dots r.
\]
Then
\begin{eqnarray}\label{eq:GKM:int}
\omega^{[r]}_{g,n}(\zeta_1,\dots,\zeta_n) = \frac{(-r)^{g-1+n}}{\alpha^{(r+1)(2g-2+n)}}\sum\limits_{\substack{d_1,\dots,d_n\geq0\\ 0\leq j_1,\dots,j_n \leq r-1}}\prod\limits_{\ell=1}^{n}\frac{c_{d_{\ell}+1,j_{\ell}}\dd \zeta_{\ell}}{\zeta_{\ell}^{r\,d_{\ell}+j_{\ell}+2}} \left\langle\prod\limits_{i=1}^{n}\tau_{d_i,j_i}e^{\sum\limits_{d,j} t_{d,j}\tau_{d,j}}\right\rangle_g.
\end{eqnarray}
\hfill $\star$
\end{thm}

\br This formula can be interpreted as an ELSV-like formula for Witten $r$-spin intersection numbers, since it relates a generating series of ciliated maps, which are the objects of combinatorial nature, to a generating series of intersection numbers of Witten's class, while the original ELSV formula \cite{ELSV01} relates simple Hurwitz numbers to the intersection numbers of the Hodge class. This interpretation is not to be confused with the so-called $r$-ELSV formula from the literature \cite{SSZ15,BKLPS17,DKPS19,Leigh20}, which features an equality between the so-called Hurwitz numbers with completed cycles and intersection numbers over the moduli space of $r$-spin structures with a class that was expressed by Chiodo \cite{ChiodoClass} in terms of standard cohomology classes with coefficients given by Bernoulli numbers. \hfill $\star$
\er

\subsection{Topological recursion for r-spin intersection numbers and examples}\label{sec:extensive:examples}

The ELSV-type formula from previous section allows us to deduce that $r$-spin intersection numbers satisfy topological recursion.

\subsubsection{Topological recursion for Witten r-spin class}
In September 2014, it was announced that the output of the topological recursion for the $r$-Airy spectral curve can be expanded in terms of $r$-spin intersection numbers. This result was proved in \cite{D-BNOPS15}, as a consequence to topological recursion applied to Dubrovin's superpotential giving the same CohFT as for the $A_r$ singularity. In \cite{BouEy17} the authors proved that after regularising the $(0,1)$ contribution at $\infty$, the wave function $\psi=\psi^{\rm reg}(z,\infty)$, constructed from the output of topological recursion for the $r$-Airy spectral curve, satisfies the $r$-Airy quantum curve:
\begin{equation}\label{rAiryQC}
\left(\hslash^r\frac{d^r}{dx^r}-x\right)\psi=0.
\end{equation}
The work \cite{BouEy17} includes a sketch of proof of this theorem as a consequence of the quantum curve \eqref{rAiryQC}. 
Here we give a different proof of a more general version of this result, which recovers the statements from the literature by specialisation of the parameters.
\begin{thm}\label{thm:TR:int}
Suppose that 
\[
\sum\limits_{k=1}^{N} \frac{1}{\lambda_k^{j}}=0,\qquad j=r,\,r+1.
\]
The topological recursion applied to the spectral curve given by the data:
\[
\begin{cases}
x=z^r\,, \\
y=z+\frac{1}{r\,\alpha^{r+1}}\sum\limits_{k=1}^{N}\frac{1}{\lambda_k^{r-1}(z-\lambda_k)}\,, \\
\omega^{[r]}_{0,1}(z)=\alpha^{r+1} y(z)\, \dd x(z)\,, \\
\omega^{[r]}_{0,2}(z_1,z_2)=\frac{dz_1 dz_2}{(z_1-z_2)^2}
\end{cases}
\]
produces $\omega^{[r]}_{g,n}$ with expansions around $\infty$ given by
\[
\begin{split}
\omega_{g,n}^{[r]}(z_1,\dots,z_n)=&\frac{(-r)^{g-1+n}}{\alpha^{(r+1)(2g-2+n)}} \sum\limits_{\substack{0\leq d_1,\dots,d_n \\ 0\leq a_1,\dots,a_n\leq r-1}}\prod\limits_{i=1}^{n}\frac{c_{d_i+1,a_i}d z_i}{z_i^{r d_i + a_i+2}} \left\langle \prod\limits_{i=1}^{n} \tau_{d_i,a_i} e^{\sum\limits_{d,j}t_{d,j}\tau_{d,j}} \right\rangle_g,
\end{split}
\]
where:
\[
t_{d,j} = c_{d,j} \sum\limits_{k=1}^{N} \frac{1}{\lambda_k^{r\,d + j+1}}\,.
\]
\hfill $\star$
\end{thm}
We can specialise to the case where $\lambda_1=\dots=\lambda_N=\infty$ (in this case, all the times $t_{d,j}$ vanish):
\begin{cor}[\cite{D-BNOPS15,BouEy17}]\label{cor:TR:int} The topological recursion applied to the r-Airy spectral curve
\[
\begin{cases}
x=z^r\,, \\
y=z\,, \\
\omega^{[r]}_{0,1}(z)=\alpha^{r+1} y(z)\, \dd x(z)\,, \\
\omega^{[r]}_{0,2}(z_1,z_2)=\frac{dz_1 dz_2}{(z_1-z_2)^2}
\end{cases}
\]
produces $\omega^{[r]}_{g,n}$ with expansions around $\infty$ given by
\begin{eqnarray}\label{eq:cor:TR:int}
\omega_{g,n}^{[r]}(z_1,\dots,z_n)=&\frac{(-r)^{g-1+n}}{\alpha^{(r+1)(2g-2+n)}} \sum\limits_{\substack{0\leq d_1,\dots,d_n \\ 0\leq a_1,\dots,a_n\leq r-1}}\prod\limits_{i=1}^{n}\frac{c_{d_i+1,a_i}d z_i}{z_i^{r d_i + a_i+2}} \underset{\overline{\mathfrak{M}}_{g,n}}\int c_{\textup{W}}(a_1,\dots,a_n) \psi_1^{d_1}\dots \psi_{n}^{d_n}.
\end{eqnarray}
\hfill $\star$
\end{cor}

\begin{proof}[Proof of Theorem \ref{thm:TR:int}]
It follows from our formula \eqref{eq:GKM:int} relating Generalised Kontsevich Graphs to $r$-spin intersection numbers and the topological recursion of Theorem \ref{thm:top:rec} for Generalised Kontsevich Graphs in the special case $V(z)=\frac{z^{r+1}}{r+1}$.
\end{proof}

\begin{rem} It is known that topological recursion for spectral curves with simple ramifications is always related to intersection theory on the moduli space of curves, or more precisely, it always corresponds to a semi-simple cohomological field theory \cite{Eynard11, DOSS}. We know that, taking a certain limit of curves with simple ramifications within the large family of spectral curves that we study, gives us the r-Airy spectral curve, which is a spectral curve with higher ramification. From \cite{BoEy13} and Section~\ref{sec:higher:tr} we know the limit behavior at the level of topological recursion correlators. We will further investigate this limit at the level of intersection numbers in future work.
In general, it would be interesting to study the implications of the power of topological recursion and the graph identification in the context of intersection theory, for instance in order to achieve a better understanding of Witten's class.
\end{rem}

\subsubsection{Extensive examples of enumeration}

In this section we enumerate Generalised Kontsevich Graphs in the special case $V(z)=\frac{z^{r+1}}{r+1}$ for the lowest stable topologies $(g,n)=(0,3), (1,1)$, and deduce the corresponding $r$-spin intersection numbers. We consider the case $\lambda_j = \infty$: we mentioned in Subsection \ref{sec:ex:lambda:infty} that the only graphs that have non vanishing weights have only marked faces (no internal faces).

\subsubsection*{Genus 1, 1 marked face} 
\paragraph{First strategy: direct computation}
$\mathcal{W}^{[r]}_{1,1}(z_1)$ consists of 4 graphs of degree $r+1$ (they all have automorphism factors equal to 1 since they are ciliated):
\[
\includegraphics[width=\textwidth]{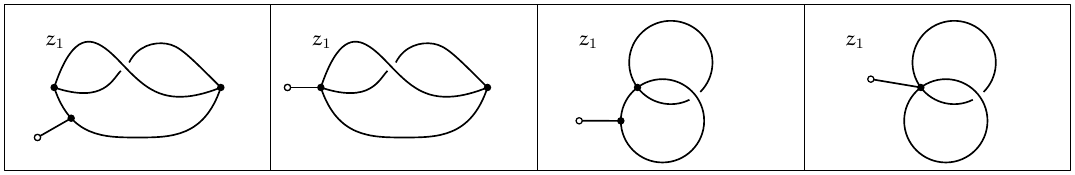}
\]
The sum of their weights is:
\[
\begin{split}
\sum\limits_{G\in\mathcal{W}^{[r],1}_{1,1}(z_1)}\frac{w(G)}{\# \Aut \,G} =& \mathcal{P}(z_1,z_1)^{5}\mathcal{V}_3(z_1,z_1,z_1)^3 +\mathcal{P}(z_1,z_1)^{4}\mathcal{V}_3(z_1,z_1,z_1) \mathcal{V}_4 (z_1,z_1,z_1,z_1) \\
&+ \mathcal{P}(z_1,z_1)^{4}\mathcal{V}_3(z_1,z_1,z_1) \mathcal{V}_4 (z_1,z_1,z_1,z_1) + \mathcal{P}(z_1,z_1)^{3}\mathcal{V}_5(z_1,z_1,z_1,z_1,z_1)\\
=& \frac{1}{6} \frac{V^{(3)}(z_1)V^{(4)}(z_1)}{V''(z_1)^4}-\frac{1}{8} \frac{V^{(3)}(z_1)^3}{V''(z_1)^5}-\frac{1}{24} \frac{V^{(5)}(z_1)}{V''(z_1)^3} \,.
\end{split}
\]
We then get, in the case where there is no internal face:
\[
W^{[r]}_{1,1}(z_1)=\alpha^{-(r+1)}\left[\frac{1}{6} \frac{V^{(3)}(z_1)V^{(4)}(z_1)}{V''(z_1)^4}-\frac{1}{8} \frac{V^{(3)}(z_1)^3}{V''(z_1)^5}-\frac{1}{24} \frac{V^{(5)}(z_1)}{V''(z_1)^3}\right]\,.
\]
In the case $V(z)=\frac{z^{r+1}}{r+1}$, we get:
\begin{equation}\label{eq:W:1:1}
W^{[r]}_{1,1}(z_1)=-\frac{\alpha^{-(r+1)}}{24}\frac{r^2-1}{r^2}\frac{1}{z_1^{2r+1}}
\end{equation}
\[
\omega^{[r]}_{1,1}(z_1)=-\frac{\alpha^{-(r+1)}}{24}\frac{r^2-1}{r}\frac{\dd z_1}{z_1^{r+2}}\,.
\]

\paragraph{Second strategy: compute $F$ then $W$.}
Note that we could have derived the result in another equivalent way. Indeed, we know from Lemma \ref{lem:uncil:to:cil} that:
\[
W^{[r]}_{1,1}(z_1)=\frac{1}{V''(z_1)} \frac{\dd}{\dd z_1} F^{[r]}_{1,1}(z_1)\,.
\]
So we actually only need to consider the graphs in $\mathcal{F}^{[r]}_{1,1}(z_1)$ in order to compute $W^{[r]}_{1,1}$. The interest of this strategy is that $\mathcal{F}^{[r]}_{1,1}$ contains less graphs than $\mathcal{W}^{[r]}_{1,1}$ since the graphs are unciliated. However, one has to take the automorphism factors into account. We use this strategy in order to compute $\mathcal{W}^{[r]}_{0,3}$ in the next subsection. \\

$\mathcal{F}^{[r]}_{1,1}(z_1)$ contains two graphs of degree $r+1$:
\[
\includegraphics[width=0.5\textwidth]{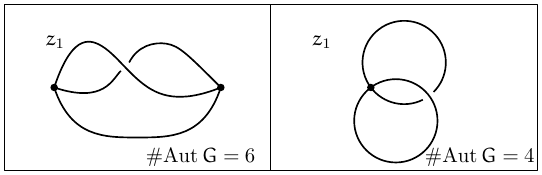}
\]
The sum of their weights (taking the automorphism factors into account) yields:
\[
\begin{split}
\sum\limits_{G\in\mathcal{F}^{[r]}_{1,1}(z_1)} \frac{w(G)}{\# \mathrm{Aut}\, G} =&\frac{1}{6}\mathcal{P}(z_1,z_1)^{3} \,\mathcal{V}_3(z_1,z_1,z_1)^2+\frac{1}{4}\mathcal{P}(z_1,z_1)^2\,\mathcal{V}_4 (z_1,z_1,z_1,z_1) \\
=& \frac{1}{24}\left[\frac{V^{(3)}(z_1)^2}{V''(z_1)^3}-\frac{V^{(4)}(z_1)}{V''(z_1)^2}\right]\,.
\end{split}
\]
So we get, in the case without internal faces:
\[
F^{[r]}_{1,1}(z_1) = \frac{\alpha^{-(r+1)}}{24}\left[\frac{V^{(3)}(z_1)^2}{V''(z_1)^3}-\frac{V^{(4)}(z_1)}{V''(z_1)^2}\right]\,.
\]
In the special case $V(z)=\frac{z^{r+1}}{r+1}$, this yields:
\begin{eqnarray}\label{eq:F:1:1}
F^{[r]}_{1,1}(z_1) &=& \frac{\alpha^{-(r+1)}}{24}\frac{r-1}{r}\frac{1}{z_1^{r+1}}\,, \cr
W^{[r]}_{1,1}(z_1) &=& \frac{1}{V''(z_1)} \frac{\dd}{\dd z_1} F^{[r]}_{1,1}(z_1) = -\frac{\alpha^{-(r+1)}}{24}\frac{r^2-1}{r^2}\frac{1}{z_1^{2r+1}}\,,
\end{eqnarray}
which is consistent with equation \eqref{eq:W:1:1}. 

\paragraph{Specific values of spin}
\begin{itemize}
	\item $r=2$: 
	\[F^{[2]}_{1,1}(z_1) = \frac{\alpha^{-3}}{48}\frac{1}{z_1^3} \;;\qquad W^{[2]}_{1,1}(z_1) = -\frac{\alpha^{-3}}{32}\frac{1}{z_1^{5}}\;;\qquad \omega^{[2]}_{1,1}(z_1) = -\frac{\alpha^{-3}}{16}\frac{\dd z_1}{z_1^{4}}\,.\]
	\item $r=3$: 
	\[F^{[3]}_{1,1}(z_1) = \frac{\alpha^{-4}}{36}\frac{1}{z_1^4} \;;\qquad W^{[3]}_{1,1}(z_1) = -\frac{\alpha^{-4}}{27}\frac{1}{z_1^{7}}\;;\qquad \omega^{[3]}_{1,1}(z_1) = -\frac{\alpha^{-4}}{9}\frac{\dd z_1}{z_1^{5}}\,.\]
	\item $r=4$: 
	\[F^{[4]}_{1,1}(z_1) = \frac{\alpha^{-5}}{32}\frac{1}{z_1^5} \;;\qquad W^{[4]}_{1,1}(z_1) = -\frac{5\alpha^{-5}}{128}\frac{1}{z_1^{9}}\;;\qquad \omega^{[4]}_{1,1}(z_1) = -\alpha^{-5}\frac{5}{32}\frac{\dd z_1}{z_1^{6}}\,.\]
\end{itemize}

\subsubsection*{Genus 0, 3 marked faces}
As mentioned previously, it is simpler to derive the quantity $F^{[r]}_{0,3}(z_1,z_2,z_3)$ in this case, since the cilia make the number of graphs to take into account grow seriously: in the case without internal faces, $\mathcal{F}^{[r]}_{0,3}(z_1,z_2,z_3)$ contains 7 graphs, while $\mathcal{W}^{[r]}_{0,3}(z_1,z_2,z_3)$ contains 280 graphs. \\
$\mathcal{F}^{[r]}_{0,3}(z_1,z_2,z_3)$ contains 7 graphs of degree $r+1$ whose the automorphism factors are all equal to one:
\[
\includegraphics[width=0.6\textwidth]{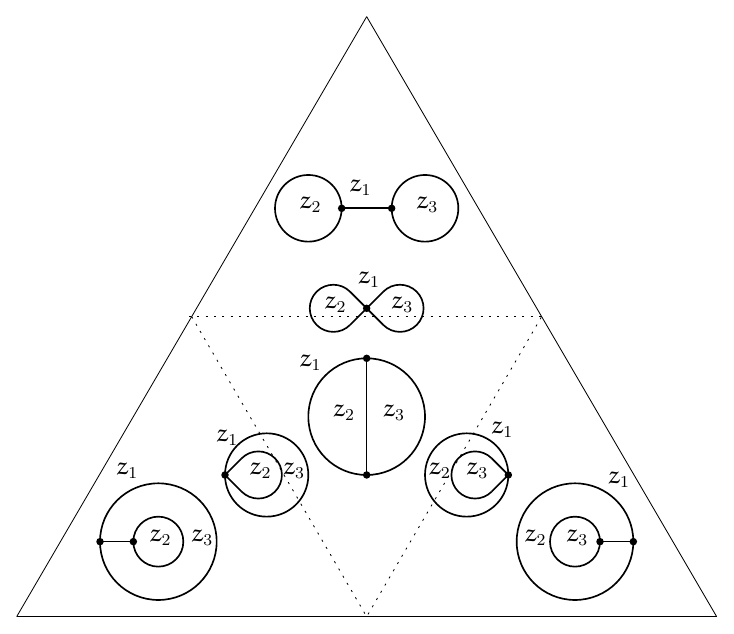}
\]
The sum of their weights is:
\[
\begin{split}
\sum\limits_{G\in\mathcal{F}^{[r]}_{0,3}(z_1,z_2,z_3)} \frac{w(G)}{\# \Aut \, G}&= \calP(z_1,z_1)\calP(z_1,z_2)\calP(z_1,z_3) \calV_3(z_1,z_1,z_2)\calV(z_1,z_1,z_3) + \calP(z_1,z_2)\calP(z_1,z_3) \calV_4(z_1,z_1,z_2,z_3) \\
&+ \calP(z_3,z_3)\calP(z_1,z_3)\calP(z_2,z_3) \calV_3(z_1,z_3,z_3)\calV(z_2,z_3,z_3) + \calP(z_1,z_3)\calP(z_2,z_3) \calV_4(z_1,z_2,z_3,z_3) \\
&+ \calP(z_2,z_2)\calP(z_2,z_3)\calP(z_1,z_2) \calV_3(z_2,z_2,z_3)\calV(z_1,z_2,z_2) + \calP(z_1,z_2)\calP(z_2,z_3) \calV_4(z_1,z_2,z_2,z_3) \\
&+ \calP(z_1,z_2)\calP(z_2,z_3)\calP(z_3,z_1) \calV_3(z_1,z_2,z_3)^2 \\
&= \frac{-1}{(z_1-z_2)(z_2-z_3)(z_3-z_1)}\left(\frac{z_2-z_3}{V''(z_1)}+\frac{z_3-z_1}{V''(z_2)}+\frac{z_1-z_2}{V''(z_3)}\right)\,.
\end{split}
\]
We then get, in the case where we exclude the internal faces:
\[
F^{[r]}_{0,3}(z_1,z_2,z_3) = \frac{-\alpha^{-(r+1)}}{(z_1-z_2)(z_2-z_3)(z_3-z_1)}\left(\frac{z_2-z_3}{V''(z_1)}+\frac{z_3-z_1}{V''(z_2)}+\frac{z_1-z_2}{V''(z_3)}\right)\,.
\]
Note that it has no pole at coinciding points $z_i\to z_j$. In the special case $V(z)=\frac{z^{r+1}}{r+1}$, we can rewrite it in this way:
\begin{equation}\label{eq:F:0:3}
F^{[r]}_{0,3}(z_1,z_2,z_3) = \frac{\alpha^{-(r+1)}}{r}\, \frac{1}{z_1\,z_2\,z_3} \sum\limits_{\substack{j_1+j_2+j_3=r-2 \\ j_i\geq 0}} \frac{1}{z_1^{j_1}z_2^{j_2}z_3^{j_3}}\,.
\end{equation}
We deduce the result for $W^{[r]}_{0,3}(z_1,z_2,z_3)$:
\begin{eqnarray}\label{eq:W:0:3}
W^{[r]}_{0,3}(z_1,z_2,z_3)&=& \frac{1}{V''(z_1)V''(z_2)V''(z_3)}\frac{\dd}{\dd z_3}\frac{\dd}{\dd z_2}\frac{\dd}{\dd z_1}F^{[r]}_{0,3}(z_1,z_2,z_3)\cr
&=&\frac{-\alpha^{-(r+1)}}{r^4 \left(z_1\,z_2\,z_3\,\right)^{r-1}}\sum\limits_{\substack{j_1+j_2+j_3=r-2 \\ j_i\geq 0}} \frac{(j_1+1)(j_2+1)(j_3+1)}{z_1^{j_1+2}z_2^{j_2+2}z_3^{j_3+2}}\,.
\end{eqnarray}

For specific values of spins, we get:
\begin{itemize}
	\item $r=2$:
	\[F^{[2]}_{0,3}(z_1,z_2,z_3)= \frac{\alpha^{-3}}{2}\frac{1}{z_1\, z_2\, z_3} \;;\qquad W^{[2]}_{0,3}(z_1,z_2,z_3)=-\frac{\alpha^{-3}}{16} \frac{1}{(z_1\,z_2\,z_3)^3}\,, \]
	\[\omega^{[2]}_{0,3}(z_1,z_2,z_3)=-\frac{\alpha^{-3}}{2}\frac{\dd z_1 \dd z_2 \dd z_3}{z_1^2 z_2^2 z_3^2}\,. \]
	\item $r=3$:
	\[F^{[3]}_{0,3}(z_1,z_2,z_3)= \frac{\alpha^{-4}}{3}\frac{1}{z_1\, z_2\, z_3}\left(\frac{1}{z_1}+\frac{1}{z_2}+\frac{1}{z_3}\right) \;;\qquad W^{[3]}_{0,3}(z_1,z_2,z_3)=-\frac{2\alpha^{-4}}{81} \frac{1}{(z_1\,z_2\,z_3)^4}\left(\frac{1}{z_1}+\frac{1}{z_2}+\frac{1}{z_3}\right)\,, \]
	\[\omega^{[3]}_{0,3}(z_1,z_2,z_3)=-\alpha^{-4}\frac{2}{3}\left(\frac{1}{z_1^3 z_2^2 z_3^2}+\frac{1}{z_1^2 z_2^3 z_3^2}+\frac{1}{z_1^2 z_2^2 z_3^3}\right)\dd z_1 \dd z_2 \dd z_3\,. \]
	\item $r=4$:
	\[
	\begin{split}
	F^{[4]}_{0,3}(z_1,z_2,z_3)&= \frac{\alpha^{-5}}{4}\frac{1}{z_1\, z_2\, z_3}\left(\frac{1}{z_1^2}+\frac{1}{z_2^2}+\frac{1}{z_3^2}+\frac{1}{z_1\, z_2}+\frac{1}{z_2\, z_3}+\frac{1}{z_3\, z_1}\right)\,, \\
	W^{[4]}_{0,3}(z_1,z_2,z_3)&=-\frac{\alpha^{-5}}{256}\frac{1}{(z_1\,z_2\,z_3)^{5}}\left(\frac{3}{z_1^2}+\frac{3}{z_2^2}+\frac{3}{z_3^2}+\frac{4}{z_1\, z_2}+\frac{4}{z_2\, z_3}+\frac{4}{z_3\, z_1}\right)\,,
	\end{split}
	\]
	\[\omega^{[4]}_{0,3}(z_1,z_2,z_3)=-\alpha^{-5}\left(\frac{3}{4}\frac{1}{z_1^4 z_2^2 z_3^2}+\frac{3}{4}\frac{1}{z_1^2 z_2^4 z_3^2}+\frac{3}{4}\frac{1}{z_1^2 z_2^2 z_3^4}+\frac{1}{z_1^3 z_2^3 z_3^2}+\frac{1}{z_1^2 z_2^3 z_3^3}+\frac{1}{z_1^3 z_2^2 z_3^3}\right)\dd z_1 \dd z_2 \dd z_3\,. \]
\end{itemize}

\subsubsection*{Getting the intersection numbers} 
In the case $V(z)=\frac{z^{r+1}}{r+1}$, we can apply equation \eqref{eq:GKM:int} of Theorem \ref{thm:GKM:int} which relates the differentials given by the enumeration of Generalised Kontsevich Graphs and intersection of Witten's class with $\psi$-classes. In the case with no internal faces, this relation specialises to equation \eqref{eq:cor:TR:int} of Corollary \ref{cor:TR:int}. 

\paragraph{Genus 1, 1 marked point}
From combinatorics and equation \eqref{eq:cor:TR:int}:
\[
\omega^{[r]}_{1,1}(z_1) = -\frac{\alpha^{-(r+1)}}{24} \frac{r^2-1}{r}\frac{\dd z_1}{z_1^{r+2}}= -\alpha^{-(r+1)}\frac{r+1}{r}\frac{\dd z_1}{z_1^{r+2}} \left\langle \tau_{1,0} \right\rangle_1\,.
\]
Therefore:
\begin{equation}
\left\langle \tau_{1,0} \right\rangle_1 = \frac{r-1}{24}\,.
\end{equation}

\paragraph{Genus 0, 3 marked points}
In the same manner:
\begin{itemize}
	\item $r=2$
	\[\omega^{[2]}_{0,3}(z_1,z_2,z_3)=-\frac{\alpha^{-3}}{2}\frac{\dd z_1 \dd z_2 \dd z_3}{z_1^2 z_2^2 z_3^2} =-\frac{\alpha^{-3}}{2}\left\langle \tau_{0,0}^3 \right\rangle_0 \frac{\dd z_1 \dd z_2 \dd z_3}{z_1^2 z_2^2 z_3^2}\,, \]
	so $\left\langle \tau_{0,0}^3 \right\rangle_0 = 1$.
	\item $r=3$
	\[
	\begin{split}
	\omega^{[3]}_{0,3}(z_1,z_2,z_3)&=-\alpha^{-4}\frac{2}{3}\left(\frac{1}{z_1^3 z_2^2 z_3^2}+\frac{1}{z_1^2 z_2^3 z_3^2}+\frac{1}{z_1^2 z_2^2 z_3^3}\right)\dd z_1 \dd z_2 \dd z_3 \\
	&=-\alpha^{-4}\frac{2}{3}\left\langle \tau_{0,0}^2 \tau_{0,1} \right\rangle_0\left(\frac{1}{z_1^3 z_2^2 z_3^2}+\frac{1}{z_1^2 z_2^3 z_3^2}+\frac{1}{z_1^2 z_2^2 z_3^3}\right)\dd z_1 \dd z_2 \dd z_3 \,,
	\end{split}
	\]
	so $\left\langle \tau_{0,0}^2 \tau_{0,1} \right\rangle_0 = 1$.
	\item $r=4$
	\[
	\begin{split}
	\omega^{[4]}_{0,3}(z_1,z_2,z_3)=&-\alpha^{-5}\left(\frac{3}{4}\frac{1}{z_1^4 z_2^2 z_3^2}+\frac{3}{4}\frac{1}{z_1^2 z_2^4 z_3^2}+\frac{3}{4}\frac{1}{z_1^2 z_2^2 z_3^4}+\frac{1}{z_1^3 z_2^3 z_3^2}+\frac{1}{z_1^2 z_2^3 z_3^3}+\frac{1}{z_1^3 z_2^2 z_3^3}\right)\dd z_1 \dd z_2 \dd z_3 \\
	=& -\alpha^{-5}\frac{3}{4}\left(\frac{1}{z_1^4 z_2^2 z_3^2}+\frac{1}{z_1^2 z_2^4 z_3^2}+\frac{1}{z_1^2 z_2^2 z_3^4}\right)\left\langle \tau_{0,0} \tau_{0,2} \right\rangle_0 \dd z_1 \dd z_2 \dd z_3\\
	&-\alpha^{-5}\left(\frac{1}{z_1^3 z_2^3 z_3^2}+\frac{1}{z_1^2 z_2^3 z_3^3}+\frac{1}{z_1^3 z_2^2 z_3^3}\right)\left\langle \tau_{0,0} \tau_{0,1}^2 \right\rangle_0 \dd z_1 \dd z_2 \dd z_3\,,
	\end{split}
	\]
	so $\left\langle \tau_{0,0} \tau_{0,1}^2 \right\rangle_0 = 1$ and  $\left\langle \tau_{0,0} \tau_{0,2} \right\rangle_0 = 1$.
\end{itemize}

\newpage 

\appendix

\section{Special geometry and combinatorics of the string equation}
\label{Appendix}

The formulation of \cite{AvM92} equivalently characterises $Z^{[r]}$ as a solution of the string equation among highest-weight vectors of the infinite-dimensional associative W-algebra $\mathbf W_r$ \cite{BoSch95}, namely solutions to $\mathbf W_r$-constraints as mentioned in the introduction. Genus expanded, the string equation \eqref{eq:string:ciliated} reads
\begin{eqnarray}\label{string}
\sum_{\substack{d\geq 0 \\ 0\leq j\leq r-1}} t_{d+1,j}\frac{\partial F^{[r]}_g}{\partial t_{d,j}} + \delta_{g,0} \frac{\hslash^{-1}}{2}\sum\limits_{j,\ell} \eta^{j\,\ell }t_{0,j}t_{0,\ell} = \frac{\partial F^{[r]}_g}{\partial t_{0,0}} \;,
\end{eqnarray}
where we denoted $F^{[r]}_{g}(\underline{t})$ the generating series of ciliated maps of genus $g\geq 0$ without boundaries and $\eta^{j\,\ell }=\delta_{j+\ell,r-2}$. By equation \eqref{eq:MM:int}, it is given by:
\begin{eqnarray}\label{efgees}
F^{[r]}_{g}(\underline{t}) & = & \hslash^{g-1} \Big\langle e^{\sum_{d,j}t_{d,j}\tau_{d,j}}\Big\rangle_g \cr
& = & \hslash^{g-1} \ \  \sum_{m\geq0}\frac{1}{m!} \sum_{\substack{0\leq j_1,\dots,j_m\leq r-2 \\ d_1,\dots,d_m\geq 0}}\Big(\prod_{i=1}^mt_{d_i,j_i}\Big)\left\langle\prod_{i=1}^m\tau_{d_i,j_i}\right\rangle_g,
\end{eqnarray}
expressed in terms of the variables $t_{d,j}=-\alpha^{-(r+1)}c_{d,j} \sum_{k=1}^N \Lambda_k^{-d-\frac{j+1}{r}}$ and the parameter $\hslash = \frac{-r}{(\alpha^{r+1})^2}$. To make the notations lighter, here we removed the dependence on $N$ from the coefficient of $t_{d,j}$ and from $\hslash$, in comparison to \eqref{eq:def:t:c}.

\medskip

This appendix is devoted to the derivation of the string equations \eqref{string}, without assuming the coefficients appearing in the definitions of the free energies \eqref{efgees} to be known or even identified as intersection numbers. It is done in three steps: first, we describe the expansion of the topological recursion differentials as the variables approach infinity (Lemma \ref{lem:behavior:infinity}) for $g\geq 1$; second we prove the string equations (Theorem \ref{thm:string:equation}) for $g\geq 1$; third we treat the case $g=0$. This appendix relies on the deformation properties of topological recursion and the tools are mainly combinatorial.

In this appendix, the potential is set to $V(z)= \frac{z^{r+1}}{r+1}$ and the times to $t_{d,j} = -\frac{c_{d,j}}{\alpha^{r+1}} \sum\limits_{k=1}^{N} \Lambda_k^{-d-\frac{j+1}{r}}$.


\subsection*{Structure of the correlators near infinity}
We will need this lemma, which unveils some deformation properties of the topological recursion output. 
\begin{lem}\label{lem:insertion:operator}
For $g\geq 1$, we have:
\begin{eqnarray}\label{eq:insertion:operator}
\frac{\partial}{\partial \Lambda_k} \omega^{[r]}_{g,n}(\zeta_1,\dots,\zeta_n) &=& \frac{\omega_{g,n+1}^{[r]}(\zeta_1,\dots,\zeta_n,\xi_k)}{Q'(\xi_k)\dd \xi_k}.
\end{eqnarray}
\hfill $\star$
\end{lem}
\begin{proof} For $g\geq 1$, we have
\[
\frac{\omega_{g,n+1}^{[r]}(\zeta_1,\dots,\zeta_n,\xi_k)}{Q'(\xi_k)\dd \xi_k}= W^{[r]}_{g,n+1}(z_1,\dots,z_n,\lambda_k)\prod_{i=1}^n\dd x(\zeta_i).
\]
Recall that $\Lambda_k=V'(\lambda_k)$. Therefore, the proof comes down to the following formula:
\[
\frac{1}{V''(\lambda_k)}\frac{\partial}{\partial \lambda_k}W^{[r]}_{g,n}(z_1,\dots,z_n) = W^{[r]}_{g,n+1}(z_1,\dots,z_n,\lambda_k) - \delta_{g,0}\delta_{n,0} \sum\limits_{\ell=1}^{N}\Big( \frac{1}{V''(\lambda_k)(\lambda_k-\lambda_{\ell})} - \frac{1}{V'(\lambda_k)-V'(\lambda_{\ell})}\Big)
\]
which holds also for $g=0$, and follows from the same arguments as equation \eqref{eq:uncil:to:cil} of Theorem \ref{thm:relations}. 
\end{proof}

The point at infinity is a simple pole of $y(\zeta)$ and therefore not a ramification point of the spectral curve. The topological recursion invariants being singular at the ramification points only, they are holomorphic at infinity and $\omega_{g,n}(\zeta_1,\dots,\zeta_n)$ can be expanded in positive powers of $\frac{1}{\zeta_i}\sim \frac{1}{Q(\zeta_i)^{1/r}}$ when $\zeta_i\rightarrow\infty$ for all $i\in\{1,\dots,n\}$ as
\begin{eqnarray*}
\omega^{[r]}_{g,n}(\zeta_1,\dots,\zeta_n) &\underset{\zeta_i \to \infty}{=}& \frac{(-1)^n\, (-r)^{g-1}}{\alpha^{(r+1)(2g-2+n)}}\sum_{\substack{d_1,\dots,d_n\geq0 \\ 0\leq j_i\leq r-1}} C_{g,n}(\mathbf d;\mathbf j)\prod_{i=1}^n c_{d_i,j_i} \dd\left(Q(\zeta_i)^{-d_i-\frac{j_i+1}{r}} \right),
\end{eqnarray*}
where we denoted collectively $\mathbf d=(d_1,\dots,d_n)$ and $\mathbf j=(j_1,\dots,j_n)$. We also introduced the combinatorial factor $c_{d,j}=(-1)^d \frac{\Gamma(d+\frac{j+1}{r})}{\Gamma(\frac{j+1}{r})}$ for later convenience. The expansion coefficients $C_{g,n}(\mathbf d;\mathbf j)$ still depend on the external field, which we make explicit by writing $C_{g,n}(\mathbf{d};\mathbf{j};\underline{t})$, where we denote collectively $\underline{t} = (t_{d,j})_{\substack{d\geq 0 \\ 0\leq j\leq r-1}}$. \medskip \\

At vanishing times, the expansion coefficients are denoted:
\[
\big\langle\prod_{i=1}^n\tau_{d_i,j_i}\big\rangle_g \coloneqq C_{g,n}(\mathbf d;\mathbf j;\underline{t}=0),
\]
where we treat the symbols $\big\langle\prod_{i=1}^n\tau_{d_i,j_i}\big\rangle_g$ as formal linear brackets, forgetting for the time being their interpretation in terms of intersection theory. 

\begin{lem}\label{lem:behavior:infinity}
For $g\geq 1$, the asymptotic expansion of $\omega_{g,n}(\zeta_1,\dots,\zeta_n)$ as $\zeta_i\to\infty$ has the following shape:
\begin{equation}\label{eq:behavior:infinity}
\omega^{[r]}_{g,n}(\zeta_1,\dots,\zeta_n) \underset{\zeta_i\to\infty}{=} \frac{(-1)^n (-r)^{g-1}}{\alpha^{(r+1)(2g-2+n)}}\sum\limits_{\substack{d_1,\dots,d_n\geq 0 \\ 0\leq j_i\leq r-1}} \left\langle\Big(\prod_{i=1}^n \tau_{d_i,j_i}\Big)  e^{\sum_{d,j}t_{d,j}\tau_{d,j}}\right\rangle_g\prod_{i=1}^n c_{d_i,j_i} \dd\left(Q(\zeta_i)^{-d_i-\frac{j_i+1}{r}} \right).
\end{equation}
\hfill $\star$
\end{lem}

\begin{proof}
We use Lemma \ref{lem:insertion:operator} to expand the variation equation in the $\zeta_i,\Lambda_k\rightarrow\infty$ regime (recall that $x$ is kept fixed). Identifying term by term gives
\[
\frac{\partial}{\partial\Lambda_k} C_{g,n}(\mathbf{d};\mathbf{j})=\sum_{\substack{d\geq 0 \\ 0\leq j\leq r-1}} C_{g,n+1}(\mathbf{d},d;\mathbf{j},j) \alpha^{-(r+1)}\frac{(d+\frac{j+1}{r})c_{d,j}}{\Lambda_k^{d+1+\frac{j+1}{r}}}.
\]
In terms of the times $t_{d,j}=-\alpha^{-(r+1)}c_{d,j} \sum_{k=1}^N \Lambda_k^{-d-\frac{j+1}{r}}$:
\begin{eqnarray*}
\frac{\partial}{\partial t_{d,j}} C_{g,n}(\mathbf d;\mathbf j) &=&  C_{g,n+1}(\mathbf d,d;\mathbf j,j).
\end{eqnarray*} 
The solution to this equation is easily expressed in terms of its value at vanishing times by
\[
C_{g,n}(\mathbf d;\mathbf j; \underline{t}) = \sum_{m\geq0}\frac{1}{m!} \sum_{\substack{d_{n+1},\dots,d_{n+m} \\ j_{n+1},\dots,j_{n+m} }}\Big(\prod_{i=n+1}^{n+m}t_{d_i,j_i}\Big) C_{g,n+m}(\mathbf d,d_{n+1},\dots,d_{n+m};\mathbf j,j_{n+1},\dots,j_{n+m};\underline{t}=0).
\]
Equivalently, with the formal brackets:
\begin{eqnarray*}
C_{g,n}(\mathbf d;\mathbf j; \underline{t}) &=& \sum_{m\geq0}\frac{1}{m!} \sum_{\substack{d_{n+1},\dots,d_{n+m}\\j_{n+1},\dots,j_{n+m}}}\Big(\prod_{i=n+1}^{n+m}t_{d_i,j_i}\Big)\left\langle\prod_{i=1}^{n+m}\tau_{d_i,j_i}\right\rangle_g \\
&=& \left\langle\Big(\prod_{i=1}^n \tau_{d_i,j_i}\Big)  e^{\sum_{d,j}t_{d,j}\tau_{d,j}}\right\rangle_g.
\end{eqnarray*}
Notice that it is consistent with equation \eqref{efgees}. It yields equation \eqref{eq:behavior:infinity} and proves the lemma.
\end{proof}

\subsection*{String equation}
\begin{thm}\label{thm:string:equation}
For $g\geq 1$ and $V(z) = \frac{z^{r+1}}{r+1}$:
\[
\sum\limits_{\substack{d\geq 0\\ 0\leq j\leq r-1}} t_{d+1,j} \frac{\partial}{\partial t_{d,j}} F^{[r]}_{g}(\underline{t}) = \frac{\partial}{\partial t_{0,0}} F^{[r]}_{g}(\underline{t})
\]
\hfill $\star$
\end{thm}
\begin{proof}
\begin{itemize}
\item On the one hand, from equations \eqref{efgees} and \eqref{eq:behavior:infinity}:
\begin{eqnarray}\label{eq:string:1}
\sum\limits_{\substack{d\geq 0\\ 0\leq j\leq r-1}} t_{d+1,j}\frac{\partial}{\partial t_{d,j}} F^{[r]}_{g}(\underline{t}) &=& \sum\limits_{k=1}^{N}\frac{(-r)^{g-1}}{\alpha^{(r+1)(2g-2+1)}} \sum\limits_{\substack{d\geq 0 \\ 0\leq j\leq r-1}}c_{d,j} \frac{d+\frac{j+1}{r}}{\Lambda_k^{d+1+\frac{j+1}{r}}}\left\langle \tau_{d,j} e^{\sum_{d',j'} t_{d',j'}\tau_{d',j'}}\right\rangle_{g} \cr
&=& \sum\limits_{k=1}^{N}\frac{\omega^{[r]}_{g,1}(\xi_k)}{Q'(\xi_k)\dd \xi_k} = \sum\limits_{k=1}^{N} W^{[r]}_{g,1}(\lambda_k).
\end{eqnarray}
Hence we obtain the sum of generating functions of ciliated maps of genus $g$ with one boundary decorated by $\lambda_1,\dots,\lambda_N$ respectively.

\item On the other hand, from equation \eqref{efgees}:
\[\frac{\partial}{\partial t_{0,0}}F^{[r]}_{g}(\underline{t}) = \frac{(-r)^{g-1}}{\alpha^{(r+1)(2g-2)}} \Big\langle \tau_{0,0} e^{\sum_{d,j}t_{d,j}\tau_{d,j}}\Big\rangle_g .\]
Therefore, from Lemma \ref{lem:behavior:infinity}, it appears as the coefficient of 
\[-\alpha^{-(r+1)} d\left(Q(\zeta)^{-\frac{1}{r}}\right)\underset{\zeta\to\infty}{\sim} \alpha^{-(r+1)} \frac{\dd\zeta}{\zeta^2}\] 
in the $\zeta\rightarrow\infty$ expansion of $\omega^{[r]}_{g,1}(\zeta)$.
\end{itemize}

We determine this coefficient in a combinatorial way. We have
\begin{equation}\label{eq:behavior:infty:1}
\omega^{[r]}_{g,1}(\zeta) = W^{[r]}_{g,1}(z) V''(z) \dd z \;;\qquad \zeta \underset{\zeta\to \infty}{=} z+ \mathcal{O}(1).
\end{equation}
We then get:
\begin{equation}\label{eq:limit:string}
\frac{\partial}{\partial t_{0,0}}F^{[r]}_{g}(\underline{t}) = \underset{z\to\infty}{\lim} z^{2}\, \alpha^{r+1}W^{[r]}_{g,1}(z) V''(z).
\end{equation}

We determine the graphs $G\in\mathcal{W}^{[r]}_{g,1}(z)$ that contribute to the limit in \eqref{eq:limit:string}.\\

Let $G\in \mathcal{W}^{[r]}_{g,1}(z)$, the cilium $z|z$ of the marked face has a weight $\mathcal{P}(z,z)=\frac{1}{V''(z)}$ that cancels the $V''(z)$ appearing in \eqref{eq:behavior:infty:1}. \\

As $z\to \infty$, an edge adjacent to the marked face behaves like $\mathcal{O}(z^{-(r-1)})$, while a vertex of degree $d$ adjacent to the marked face behaves like $\mathcal{O}(z^{r+1-d})$. Let us denote $n_d$ the number of vertices of degree $d$ adjacent to the marked face (not counting the univalent white vertex of the cilium), and $e$ the number of edges around the marked face (not including the cilium). We have:
\[
w(G) V''(z)\underset{z\to\infty}{\sim} \ \mathcal{O}\left(z^{\sum_{d=3}^{r+1}n_d(r+1-d)-e(r-1)}\right).
\]
Considering this asymptotic behavior, we distinguish three cases.
\begin{enumerate}
\item $e=1$: then $\sum_d n_d =1$. We have $n_3=0$, otherwise $G$ would be of genus 0. The exponent of $z$ near $\infty$ is $\sum_{d=3}^{r+1}n_d(r+1-d)-e(r-1) = 2-d$, where $3<d\leq r+1$ such that $n_d=1$. Therefore $\underset{z\to \infty}{\lim} z^2\,w(G) V''(z)\neq 0$ only if $d=4$, that is $n_4 =1$. The graphs $G$ of this type are depicted on the l.h.s. of Figure \ref{fig:string}.

\item $e=2$: then $\sum_d n_d =2$. The exponent of $z$ in this case is $\sum_{d=3}^{r+1}n_d(r+1-d)-e(r-1) = 4-\sum_d d\, n_d$. We obtain $\underset{z\to \infty}{\lim} z^2\,w(G) V''(z)\neq 0$ only if $n_3 =2$. The ciliated maps of that type are depicted on the r.h.s. of Figure \ref{fig:string}. 

\item $e\geq 3$: then $\sum_{d=3}^{r+1}n_d(r+1-d)-e(r-1) \leq -e\leq -3$, since $d\geq 3$ and the number of edges adjacent to a face is greater or equal than the number of vertices adjacent to it. Hence we get $\underset{z\to \infty}{\lim} z^2\,w(G) V''(z) = 0$.
\end{enumerate}

\begin{figure}
\centering
\includegraphics[scale=1]{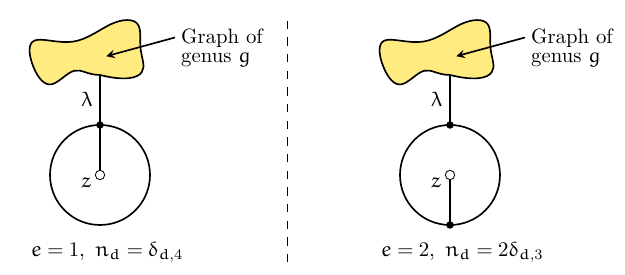}
\caption{The ciliated maps that contribute to $\frac{\partial}{\partial t_{0,0}} F^{[r]}_{g,1}(\underline{t})$.}
\label{fig:string}
\end{figure}

In the end
\[
\begin{split}
\frac{\partial}{\partial t_{0,0}} F^{[r]}_{g}(\underline{t}) =& \underset{z\to\infty}{\lim} \,\alpha^{r+1}\, z^{2} V''(z)\Big(\textup{l.h.s. of Fig. \ref{fig:string}}+\textup{r.h.s. of Fig. \ref{fig:string}}\Big) \\
=& \sum\limits_{k=1}^{N} \underset{z\to\infty}{\lim} \alpha^{r+1} z^2 V''(z) \Big(\mathcal{P}(z,z)  \mathcal{P}(z,\lambda_k) \mathcal{V}_4(\lambda_k,\lambda_k,z,z) \alpha^{-(r+1)} W^{[r]}_{g,1}(\lambda_k) \\
&\qquad \qquad + \mathcal{P}(z,z)\mathcal{P}(z,\lambda_k)^2 \mathcal{V}_3(z,\lambda_k,\lambda_k) \mathcal{V}_3(z,z,\lambda_k)\alpha^{-(r+1)} W^{[r]}_{g,1}(\lambda_k) \Big)\\
=& \sum\limits_{k=1}^{N} W^{[r]}_{g,1}(\lambda_k) \underset{z\to\infty}{\lim} z^2 \Big(-\frac{r^2 \lambda_k^{r-1} z^{r-1}}{(z^r-\lambda_k^r)^2}+\frac{1}{(\lambda_k-z)^2}\Big)\\
=& \sum\limits_{k=1}^{N} W^{[r]}_{g,1}(\lambda_k).
\end{split}
\]
Comparing with equation \eqref{eq:string:1}, this gives the string equation for $g\geq 1$.
\end{proof}

\subsection*{The planar case}
In the case $g=0$, Theorem \ref{thm:string:equation} needs to be adapted.
\begin{thm}\label{thm:string:planar}
For $V(z) = \frac{z^{r+1}}{r+1}$:
\[
\sum\limits_{\substack{d\geq 0\\ 0\leq j\leq r-1}} t_{d+1,j} \frac{\partial}{\partial t_{d,j}} F^{[r]}_{0}(\underline{t})+ \frac{\hslash^{-1}}{2}\sum\limits_{j,\ell} \delta_{j+\ell,r-2}t_{0,j}t_{0,\ell}  = \frac{\partial}{\partial t_{0,0}} F^{[r]}_{0}(\underline{t}).
\]
\hfill $\star$
\end{thm}

Before proving the theorem, let us stress that in the planar case, $F^{[r]}_{0}(\underline{t})$ does not behave well for the insertion of boundaries. Indeed:
\[
\frac{\partial F^{[r]}_{g}(\underline{t})}{\partial \Lambda_k} = W^{[r]}_{g,1}(\lambda_k) -\delta_{g,0}\sum\limits_{j=1}^{N} \Big(\frac{1}{V''(\lambda_k)(\lambda_k-\lambda_j)}-\frac{1}{V'(\lambda_k)-V'(\lambda_j)}\Big).
\]
Therefore, we introduce $\widetilde{F}^{[r]}_{0}(\underline{t})$:
\[
\widetilde{F}^{[r]}_{0}(\underline{t}) \coloneqq F^{[r]}_{0}(\underline{t}) +\sum\limits_{j=1}^{N}\log \Big(\frac{1}{V''(\lambda_j)}\Big)+\frac12 \sum\limits_{\substack{j,k=1\\ j\neq k}}^{N}\log \Big(\frac{\lambda_j-\lambda_k}{V'(\lambda_j)-V'(\lambda_k)}\Big),
\]
so that $\partial \widetilde{F}^{[r]}_0/\partial \Lambda_k = W^{[r]}_{0,1}(\lambda_k)$. In the same manner as for $g\geq 1$, we have:
\begin{align*}
\widetilde{F}^{[r]}_{0}(\underline{t}) &= \hslash^{-1} \left\langle e^{\sum_{d,j} t_{d,j}\tau_{d,j}} \right\rangle_0 ,\\
 W^{[r]}_{0,1}(z)V''(z)\dd z &\underset{z\to\infty}{=} \frac{-\hslash^{-1}}{\alpha^{r+1}} \sum\limits_{\substack{d\geq 0 \\ 0\leq j\leq r-1}} c_{d+1,j} \frac{\dd Q(\zeta)}{Q(\zeta)^{d+1+\frac{j+1}{r}}} \left\langle \tau_{d,j}e^{\sum_{d',j'}t_{d',j'}\tau_{d',j'}} \right\rangle_0.
\end{align*}

\begin{proof}[Proof of Theorem \ref{thm:string:planar}]
As in the higher genera cases, let us compute $\sum_{d,j}t_{d+1,j}\partial F^{[r]}_0(\underline{t})/\partial t_{d,j}$ on the one hand, and $\partial F^{[r]}_0(\underline{t})/\partial t_{0,0}$ on the other. We are in the case for which $V'(\lambda_j)=\lambda_j^r$.
\begin{itemize}
\item On the one hand, we get:
\[
\begin{split}
\sum\limits_{\substack{d\geq 0 \\ 0\leq j\leq r-1}}t_{d+1,j}\frac{\partial F^{[r]}_0(\underline{t})}{\partial t_{d,j}} &= \sum\limits_{k=1}^{N} \frac{\partial}{\partial \Lambda_k} \Bigg(\widetilde{F}^{[r]}_{0}(\underline{t})- \sum\limits_{j=1}^{N}\log \Big(\frac{1}{r\lambda_j^{r-1}}\Big)-\frac12 \sum\limits_{\substack{j,k=1\\ j\neq k}}^{N}\log \Big(\frac{\lambda_j-\lambda_k}{\lambda_j^r-\lambda_k^r}\Big)
\Bigg)\\
&= \sum\limits_{k=1}^{N} \Bigg( W^{[r]}_{0,1}(\lambda_k)-\sum\limits_{j=1}^{N}\Big(\frac{1}{r\lambda_k^{r-1}(\lambda_k-\lambda_j)}-\frac{1}{\lambda_k^r-\lambda_j^r}\Big)\Bigg).
\end{split}
\]

Now we obtain
\[
\begin{split}
&\sum\limits_{k, j=1}^{N}\Big(\frac{1}{r\lambda_k^{r-1}(\lambda_k-\lambda_j)}-\frac{1}{\lambda_k^r-\lambda_j^r}\Big)\\
&=\frac12 \sum\limits_{k, j=1}^{N}\Big(\frac{1}{r\lambda_k^{r-1}(\lambda_k-\lambda_j)}-\frac{1}{\lambda_k^r-\lambda_j^r}\Big)-\frac12\sum\limits_{k, j=1}^{N}\Big(\frac{1}{r\lambda_j^{r-1}(\lambda_j-\lambda_k)}-\frac{1}{\lambda_j^r-\lambda_k^r}\Big)\\
&=\frac{1}{2r}\sum\limits_{k, j=1}^{N}\frac{1}{\lambda_k-\lambda_j}\Big(\frac{1}{\lambda_k^{r-1}}-\frac{1}{\lambda_j^{r-1}}\Big)=-\frac{1}{2r}\sum\limits_{k, j=1}^{N}\frac{1}{\lambda_k\lambda_j}\frac{1}{\frac{1}{\lambda_k}-\frac{1}{\lambda_j}}\Big(\frac{1}{\lambda_k^{r-1}}-\frac{1}{\lambda_j^{r-1}}\Big)\\
&=-\frac{1}{2r}\sum_{a,b=0}^{r-2}\delta_{a+b,r-2}\sum_{j}^N\frac{1}{\lambda_j^{a+1}}\sum_{k}^N\frac{1}{\lambda_k^{a+1}}=\frac{\hslash^{-1}}{2}\sum\limits_{a,b=0}^{r-2} \eta^{a\, b}t_{0,a} t_{0,b}\,,
\end{split}
\]
where $\eta^{a\,b}=\delta_{a+b,r-2}$ and $\hbar=-\frac{r}{\alpha^{2(r+1)}}$.

Therefore:
\begin{equation}\label{eq:string:planar:1}
\sum\limits_{k=1}^{N} W^{[r]}_{0,1}(\lambda_k)= \sum\limits_{\substack{d\geq 0 \\ 0\leq j\leq r-1}}t_{d+1,j}\frac{\partial F^{[r]}_0(\underline{t})}{\partial t_{d,j}} + \frac{\hslash^{-1}}{2}\sum\limits_{i,j} \eta^{i\, j}t_{0,i} t_{0,j}.
\end{equation}
\item On the other hand:
\[
\frac{\partial F^{[r]}_{0}(\underline{t})}{\partial t_{0,0}} =\frac{\partial}{\partial t_{0,0}} \Bigg(\widetilde{F}^{[r]}_{0}(\underline{t})- \sum\limits_{j=1}^{N}\log \Big(\frac{1}{r\lambda_j^{r-1}}\Big)-\frac12 \sum\limits_{\substack{j,k=1\\ j\neq k}}^{N}\log \Big(\frac{\lambda_j-\lambda_k}{\lambda_j^r-\lambda_k^r}\Big)
\Bigg).
\]
Now, as in the higher genera cases, $\frac{\partial \widetilde{F}^{[r]}_{0}(\underline{t})}{\partial t_{0,0}}$ is the coefficient of $\alpha^{-(r+1)}\dd z/ z^2$ in the $z\to \infty$ expansion of $W^{[r]}_{0,1}(z)V''(z)\dd z$. Carrying out the same case inspection on the maps that contribute to this coefficient, we find that the graphs of Figure \ref{fig:string} contribute, but also the graph that appeared in Remark \ref{rem:pole:01}:
\[
\frac{\partial \widetilde{F}^{[r]}_{0}(\underline{t})}{\partial t_{0,0}} = \sum\limits_{k=1}^{N} W^{[r]}_{0,1}(\lambda_k) + [\alpha^{-(r+1)}z^{-2}] \sum\limits_{k=1}^{N}\Big(\frac{1}{z-\lambda_k}-\frac{r\,z^{r-1}}{z^r-\lambda_k^r}\Big) =\sum\limits_{k=1}^{N} W^{[r]}_{0,1}(\lambda_k) +\alpha^{r+1}\sum\limits_{k=1}^{N}\lambda_k. 
\]
Last, using the fact that $\log \Big(\frac{\lambda_i^r-\lambda_j^r}{\lambda_i-\lambda_j}\Big) = \log(\lambda_i^{r-1}) + \log\Big(1-\frac{\lambda_j^{r}}{\lambda_i^{r}}\Big)-\log\Big(1-\frac{\lambda_j}{\lambda_i}\Big)=-\frac{\lambda_j}{\lambda_i}+\ldots$, where we have expanded the logarithms using $|\lambda_i|>|\lambda_j|$. Since $\log \Big(\frac{\lambda_i^r-\lambda_j^r}{\lambda_i-\lambda_j}\Big)$ is symmetric when we exchange $\lambda_i$ and $\lambda_j$, we have 
$$
\sum\limits_{j=1}^{N}\log \Big(\frac{1}{r\lambda_j^{r-1}}\Big)+\frac12\sum\limits_{\substack{j,k=1\\ j\neq k}}^{N} \log \Big(\frac{\lambda_i-\lambda_j}{\lambda_i^r-\lambda_j^r}\Big)=\alpha^{-(r+1)} \Bigg(\sum\limits_{i=1}^{N} \lambda_i^{-1}\Bigg)\alpha^{r+1} \sum\limits_{j=1}^{N} \lambda_j+\ldots
$$
Then
\[
\frac{\partial }{\partial t_{0,0}}\Bigg(\sum\limits_{j=1}^{N}\log \Big(\frac{1}{r\lambda_j^{r-1}}\Big)+\frac12 \sum\limits_{\substack{j,k=1\\ j\neq k}}^{N}\log \Big(\frac{\lambda_j-\lambda_k}{\lambda_j^r-\lambda_k^r}\Big)\Bigg)= \alpha^{r+1} \sum\limits_{k=1}^{N} \lambda_k
\]
and hence
\begin{equation}\label{eq:string:planar:2}
\frac{\partial F^{[r]}_{0}(\underline{t})}{\partial t_{0,0}} = \sum\limits_{k=1}^{N} W^{[r]}_{0,1}(\lambda_k).
\end{equation}
\end{itemize}
Combining equations \eqref{eq:string:planar:1} and \eqref{eq:string:planar:2}, we get the theorem.
\end{proof}

\section{Index of notations}

{\small
\renewcommand{\arraystretch}{1.3}
\begin{center}
\begin{tabular}{c|c|l}
$G$, $\mathcal{V}(G)$, $\mathcal{E}(G)$, $\mathcal{F}(G)$
	& Def. \ref{def:map}
	& Map, sets of vertices, edges and faces of a map\\
$V(z)$, $v_j $
	& Def. \ref{def:potential}
	& Potential of the model and its coefficients\\
$\lambda_1,\dots,\lambda_N$
	& Sec. \ref{sec:Kgraphs}
	& Parameters decorating the unmarked faces\\
$\mathcal{V}_k (a_1,\dots,a_k)$
	& Def. \ref{def:vertex}
	& Weight of a black vertex of degree $3\leq k\leq r+1$ \\
$\alpha$
	& Def. \ref{def:generating:functions}
	& Formal parameter used for the enumeration of maps\\
	& & \\
$x(z),\, y(z)$
	& Def. \ref{def:x:y}
	& Projection maps of the spectral curve \\
$Q(\zeta)$
	& Thm. \ref{thm:w01}
	& Polynomial for the rational parametrisation of the spectral curve\\
$\xi_j$
	& Thm. \ref{thm:w01}
	& Solution of $Q(\xi_j)=V'(\lambda_j)$ \\
$z^{(k)}$, $\zeta^{(k)}$
	& Sec. \ref{sec:properties}, \ref{sec:spectral:curve}
	& $k^{\textup{th}}$ preimage of $V'(z)$, $Q(\zeta)$\\
$\tau(\zeta) $, $\tau_{m}(\zeta)$ 
	& Sec. \ref{sec:check:P:H}
	& Set of preimages, set of preimages except $\zeta^{(m)}$ \\
$b_1,\dots, b_{r-1} $
	& Sec. \ref{subsec:top:rec}
	& Ramification points of the spectral curve\\
$K_{b_k}(\zeta_1,\zeta)$
	& Def. \ref{def:recursion:kernel}
	& Recursion kernel at the ramification point $b_k$\\
& & \\
$\mathcal{F}^{[r]}_{g,n}(z_1,\dots,z_n)$
	& Def. \ref{def:uncil:konts}
	& Unciliated Kontsevich graphs\\
$F^{[r]}_{g,n}(z_1,\dots,z_n)$
	& Def. \ref{def:generating:functions}
	& Their generating functions \\
$\mathcal{W}^{[r]}_{g,n}(z_1,\dots,z_n)$
	& Def. \ref{def:cil:konts}
	& Ciliated Kontsevich graphs \\
$W^{[r]}_{g,n}(z_1,\dots,z_n)$
	& Def. \ref{def:generating:functions}
	& Their generating functions \\
$\widetilde{W}^{[r]}_{g,n}(\zeta_1,\dots,\zeta_n) $
	& Def. \ref{def:tilde:E}
	& Generating functions deduced from the previous ones \\
$\omega^{[r]}_{g,n}(\zeta_1,\dots,\zeta_n) $
	& Def. \ref{def:differentials}
	& The $n$-differentials which satisfy topological recursion\\
$\mathcal{S}^{[r]}_{g,\underline{k}}(S_1,\dots,S_n)$
	& Def. \ref{def:multicil:konts}
	& Multi-ciliated Kontsevich graphs  \\
$S^{[r]}_{g,\underline{k}}(S_1,\dots,S_n)$
	& Def. \ref{def:generating:functions}
	& Their generating functions\\
$\mathcal{U}^{[r]}_{g,n}(u;z_1,\dots,z_n)$
	& Def. \ref{def:sq:cil:konts}
	& Square ciliated Kontsevich graphs \\
$U^{[r]}_{g,n}(u;z_1,\dots,z_n)$
	& Def. \ref{def:generating:functions}
	& Their generating functions \\
$H^{[r]}_{g,n}(u;z_1,\dots,z_n) $
	& Def. \ref{def:auxiliary:functions}
	& Auxiliary generating function for the set $\mathcal{U}^{[r]}_{g,n}(u;z_1,\dots,z_n)$\\
$P^{[r]}_{g,n}(u;\zeta_1,\dots,\zeta_n) $ 
	& Def. \ref{def:P01}, \ref{def:Pgn}
	& Polar part of $H^{[r]}_{g,n}(u;\zeta_1,\dots,\zeta_n)$ \\
$\mathcal{E}^{(k)}W^{[r]}_{g,n}(\underline{t};I) $
	& Def. \ref{def:tilde:E}
	& Ways of removing a $k$-pointed sphere from a generating function\\
$\mathcal{R}^{(k)}W^{[r]}_{g,n}(\underline{t};I)$
	& Def. \ref{def:higher:recursion}
	& Other way of removing a $k$-pointed sphere \\
$\check{H}^{[r]}_{g,n}(u;\zeta_1,\dots,\zeta_n)$
	& Def. \ref{def:check:P:H}
	& Equivalent expression for $H^{[r]}_{g,n}(u;z_1,\dots,z_n)$\\
$\check{P}^{[r]}_{g,n}(u;\zeta_1,\dots,\zeta_n)$
	& Def. \ref{def:check:P:H}
	& Equivalent expression for $P^{[r]}_{g,n}(u;z_1,\dots,z_n)$\\
& & \\
$\left\langle M_{i_1,i_1}\dots M_{i_n,i_n}\right\rangle_{c}$
	& Eq. \eqref{eq:correl:MM}
	& Cumulants of the diagonal correlators of the matrix model\\
$\left\langle\tau_{d_1,j_1}\dots\tau_{d_n,j_n}\right\rangle_{g}$
	& Eq. \eqref{eq:notation:witten:class}
	& Alternative notation for the intersection of Witten's class with $\psi$-classes\\
$c_{d,j}$
	&Eq. \eqref{eq:def:t:c}
	&Short hand notation for a ratio of $\Gamma$-functions \\
$t_{d,j}$
	&Eq. \eqref{eq:def:t:c}
	& 2-index family of times obtained from the matrix model\\
$Z=Z_{N,\alpha}\big(V;\lambda\big)$
	&Eq. \eqref{eq:partition:function}
	& Partition function of the matrix model \\
$F_{g}^{[r]}(\underline{t})$
	&Eq. \eqref{eq:free:energy:MM}
	& Free energy of the matrix model for $V(z)=\frac{z^{r+1}}{r+1}$\\
$F_{g}^{\textup{W}}(\underline{t})$
	&Eq. \eqref{eq:FSZ}
	& Free energy for the intersection numbers\\
$\omega_{g,n}^{\textup{W}}(\zeta_1,\dots,\zeta_n)$
	&Eq. \eqref{eq:def:omega:int}
	& The $n$-differentials defined from intersection numbers
\end{tabular}
\end{center}
}

\newpage

\bibliographystyle{plain}
\bibliography{Biblirspin}

\end{document}